\documentclass[12pt]{article}
\usepackage{verbatim,amsmath,amssymb,mathrsfs,amsthm,amsfonts,bm}
\usepackage{varwidth}
\usepackage{extarrows}
\usepackage{enumitem}
\usepackage{fancyhdr}
\usepackage{booktabs}
\usepackage[dvipsnames]{xcolor}
\usepackage{graphicx,wrapfig}
\usepackage{subfigure}
\usepackage{caption} 
\usepackage{amsmath}
\usepackage{makecell}
\usepackage{multirow}
\usepackage{array}
\usepackage{setspace}
\usepackage{float}
\usepackage{cite}
\usepackage{color,mathtools}
\usepackage{epstopdf}
\usepackage{lipsum}
\usepackage[title]{appendix}
\usepackage{algorithm}
\usepackage{algorithmic}
\usepackage{blindtext}
\usepackage{ulem}
\usepackage[scr=boondox,cal=cm]{mathalfa} 

\numberwithin{equation}{section}
\numberwithin{figure}{section}
\numberwithin{table}{section}
\allowdisplaybreaks

\newtheorem{theorem}{Theorem}[section]

\newtheorem{definition}{Definition}[section]
\newtheorem{remark}{Remark}[section]
\newtheorem{example}{Example}[section]

\renewcommand{\vec}[1]{\boldsymbol{#1}}


\usepackage{hyperref}
\hypersetup{
    colorlinks=true, linkcolor=black,
    urlcolor=cyan, citecolor=black,
}

\begin{document}
\title{Well-balanced path-conservative discontinuous Galerkin methods with equilibrium preserving space  for two-layer shallow water equations}
\date{}

 \author{}
 \author{Jiahui Zhang\thanks{School of Mathematical Sciences, University of Science and Technology of China, Hefei, Anhui 230026, P.R. China.  E-mail: zjh55@mail.ustc.edu.cn.}
 	\and Yinhua Xia\thanks{Corresponding author. School of Mathematical Sciences, University of Science and Technology of China, Hefei, Anhui 230026, P.R. China.  E-mail: yhxia@ustc.edu.cn.
 	}
 	\and Yan Xu\thanks{School of Mathematical Sciences, University of Science and Technology of China, Hefei, Anhui 230026, P.R. China.  E-mail: yxu@ustc.edu.cn.
 	}
 }

\maketitle

\begin{abstract}
This paper introduces well-balanced path-conservative discontinuous Galerkin (DG) methods for two-layer shallow water equations, ensuring exactness for both still water and moving water equilibrium steady states. The approach involves approximating the equilibrium variables within the DG piecewise polynomial space, while expressing the DG scheme in the form of path-conservative schemes. To robustly handle the nonconservative products governing momentum exchange between the layers, we incorporate the theory of Dal Maso, LeFloch, and Murat (DLM) within the DG method. Additionally, linear segment paths connecting the equilibrium functions are chosen to guarantee the well-balanced property of the resulting scheme. The simple ``lake-at-rest" steady state is naturally satisfied without any modification, while a specialized treatment of the numerical flux is crucial for preserving the moving water steady state. Extensive numerical examples in one and two dimensions validate the exact equilibrium preservation of the steady state solutions and demonstrate its high-order accuracy. The performance of the method and high-resolution results further underscore its potential as a robust approach for nonconservative hyperbolic balance laws.

  \smallskip
  \textbf{Keywords}: Two-layer shallow water equations; discontinuous Galerkin method;  equilibrium preserving space; path-conservative;  well-balanced.
\end{abstract}

\clearpage
\section{Introduction}\label{se:in}
This paper considers the two-layer shallow water equations which become a useful tool for solving hydrodynamical flows such as rivers, bays, estuaries, and other nearshore areas where water flows interact with the bed geometry and wind shear stresses.  This model is derived from the compressible isentropic Euler equations through vertical averaging across each layer depth. Unlike the single-layer shallow water equations, the two-layer equations account for the vertical variation in density and velocity, making them suitable for modeling the flow dynamics over the nonflat topography of two superposed layers of immiscible fluids.

Numerical solutions for the two-layer shallow water equations present several challenges. The primary difficulty stems from the nonconservative products within the governing equations, preventing them from being expressed in divergence form. This characteristic places the equations within the category of nonconservative hyperbolic partial differential equations which have the form
\begin{equation}\label{ncs}
   \boldsymbol {u}_t + \mathcal{A}(\boldsymbol {u}) \boldsymbol {u}_x = \boldsymbol{r}(\boldsymbol {u},b).
\end{equation}
Here, $\boldsymbol {u}$ denotes the conservative variable, $\boldsymbol{r}(\boldsymbol {u},b)$ represents the source terms owing to the geometric bottom, and $\mathcal{A}(\boldsymbol {u}) \boldsymbol {u}_x$ is the nonconservative product. The coupling terms, which involve derivatives of the unknown physics variables, pose challenges in the distribution framework, particularly when the solution is discontinuous. As a result, the conventional concept of the weak solution is not applicable, and the classical Rankine-Hugoniot jump conditions cannot be defined. To address these issues, the theory introduced by Dal Maso, LeFloch, and Murat (DLM) \cite{dal1995definition} is utilized for nonconservative products. This approach connects the left and right states of the discontinuity with a family of paths in the phase space, providing a rigorous definition of the weak solution. 
A theoretical framework established by Par{\'e}s in \cite{pares2006numerical} introduces the concept of a path-conservative numerical scheme, which extends the traditional conservative method for conservative laws. Subsequently, significant mathematical research, as seen in \cite{munoz2011convergence,pares2004well,gosse2001well,castro2020well}, has been dedicated to the examination of path-conservative finite volume schemes and their application to various nonconservative hyperbolic systems. More recent advancements in this area are discussed in review papers such as \cite{pares2008path,castro2017well}. However, it is important to note that the selection of the path significantly influences the numerical solution. In \cite{abgrall2010comment}, Abgrall et al. demonstrated that due to the critical nature of path selection, a path-conservative scheme may not accurately compute nonconservative hyperbolic systems. Furthermore, it is generally not feasible to develop a scheme that converges to the assumed path, even when the correct path is presumed to be known.

The second challenging problem relates to the loss of hyperbolicity, leading to Kelvin-Helmholtz instabilities and the onset of mixing. This system is conditionally hyperbolic, and explicit eigenvalues cannot be determined. Addressing this difficulty has garnered significant interest in recent years. Various attempts have been made to overcome this issue. For example, Abgrall and Karni developed a one-dimensional relaxation approach to create an unconditionally hyperbolic model, restoring pressure and velocity interface conditions based on infinite relaxation parameters \cite{abgrall2009two}. This approach was further extended to a two-dimensional finite volume HLL scheme on triangular grids by Liu \cite{liu2021new}. 
In another approach, as described in \cite{castro2011numerical}, an additional amount of friction was introduced via a predictor or corrector strategy to enforce the hyperbolicity of the two-layer shallow water system. 
Another strategy involves introducing an intermediate layer around the original interface to enhance the stability properties of the system and restore hyperbolicity in regions of strong shear, as discussed in \cite{chertock2013three,castro2012hyperbolicity} and related references. Despite applying an arbitrary scheme to each layer, a well-balanced time splitting scheme was developed in \cite{bouchut2008entropy} using a fully discrete entropy inequality to control stability. However, it has been found that this scheme often yields incorrect solutions. To address this, a simple uncoupled method based on a source-centered hydrostatic reconstruction was devised in \cite{bouchut2010robust}. Subsequently, a new splitting strategy was introduced in \cite{berthon2015efficient}, defining two variables to solve two independent single-layer shallow water equations coupled by nonconservative products.

Another significant aspect of shallow water equations is their tendency to exhibit non-trivial steady-state solutions, where the non-zero flux gradient precisely balances the source term. Many standard numerical methods may introduce spurious oscillations due to the presence of the source term, leading to inadequate capture of the steady state unless highly refined meshes are used.  Therefore, it is crucial to design numerical methods that can exactly preserve equilibrium solutions at a discrete level and accurately resolve small perturbations to steady state solutions with relatively coarse meshes. Methods with these properties are referred to as ``well-balanced".  

Designing a well-balanced scheme for two-layer shallow water equations is conceptually more complex than for the single-layer system due to the presence of nonconservative interlayer momentum exchange terms. Over the past decade, significant attention has been devoted to the study and application of the two-layer shallow water equations. Kurganov et al. developed a second-order central-upwind (CU) scheme in \cite{kurganov2009central}, which was well-balanced and preserved positivity for the depth of each fluid layer. However, in certain situations, as illustrated in Section \ref{se:nu}, the scheme in \cite{kurganov2009central} may fail to accurately capture the solution. To address this issue, \cite{diaz2019path} developed a robust second-order path-conservative central-upwind (PCCU) scheme by incorporating the central-upwind scheme into the path-conservative framework. Additionally, in \cite{chu2022fifth}, a fifth-order finite difference AWENO scheme based on the path-conservative central-upwind method was constructed to preserve the still water steady state (``lake-at-rest").

Furthermore, \cite{kurganov2023well} introduced a new second-order well-balanced scheme for the generalized moving equilibria of the nozzle flow system and the two-layer shallow water equations by expressing the target system in an equivalent quasi-conservative form with global fluxes. Well-balanced path-consistent schemes were also presented in \cite{dudzinski2011well,dudzinski2013well} within the framework of finite volume evolution Galerkin (FVEG) methods. Several other methods have been proposed for two-layer shallow water equations, as detailed in \cite{mandli2013numerical,hernandez2021central,liu2021well}.
To the best of our knowledge, most existing schemes are well-balanced for the still water steady state. However, for the moving equilibria, the schemes are at most second-order accurate. 

In this paper, we will explore the finite element discontinuous Galerkin (DG) method, a popular high-order approach for solving hyperbolic conservation laws. This method offers a crucial advantage in simulating unstructured meshes without sacrificing high-order accuracy and conservation properties. Several studies have investigated the application of the DG framework to the two-layer shallow water equations. In \cite{izem2016discontinuous1}, the authors studied a DG method for approximating the two-layer system by employing an operator splitting method. Subsequently, \cite{izem2016discontinuous2} extended the DG scheme to the two-dimensional case on unstructured meshes using nodal polynomial basis functions. 
By introducing two auxiliary variables to reformulate two-layer systems into a new form, \cite{cheng2020high} developed a high-order central DG scheme for the preservation of the still water solution. A well-balanced path-conservative one-step ADER-DG scheme was developed in \cite{Dumbser2009ader}, where high-order accuracy was achieved with the $P_NP_M$ reconstruction operator on unstructured meshes and the local space-time Galerkin predictor. Recently, a novel discretization for multi-layer shallow water equations with variable density was also designed in \cite{guerrero2022arbitrary} under the ADER-DG framework, which is also well-balanced for the still water equilibria.  
Given that the shallow water equations can be expressed as nonconservative hyperbolic partial differential equations, \cite{rhebergen2008discontinuous} presented space and space-time discontinuous Galerkin finite element formulations. Similar path-conservative DG (PCDG) schemes have been successfully derived for other mathematical models. For example, in \cite{cheng2022bound}, a high-order positivity-preserving path-conservative discontinuous Galerkin method was developed for compressible two-medium flows, solving the five-equation transport model. Furthermore, in \cite{zhao2023path}, a path-conservative Arbitrary DERivatives in Space and Time (ADER) DG method was proposed to solve single-layer shallow water equations. For the Baer-Nunziato model, further information can be found in \cite{franquet2012runge}.

This paper aims to introduce high-order, well-balanced, path-conservative discontinuous Galerkin methods designed for two-layer shallow water equations to achieve equilibrium steady states in both still and moving water scenarios. These methods are developed based on the weak solution of the DLM theory \cite{dal1995definition}, which introduces a family of paths in the phase space. For the hydrostatic steady state, known as the ``lake-at-rest", we reframe the model in terms of equilibrium variables, an alternative and equivalent formulation of the original system. The numerical flux employed is the simplest Lax-Friedrichs without further modification. To preserve the generalized non-hydrostatic steady state, we incorporate the equilibrium-preserving space in DG methods. This concept, initially established for single-layer shallow water equations \cite{zhang2023moving}, has been extended to other hyperbolic balance laws such as the Ripa model and Euler equations with gravitation \cite{2023arXivequilibrium}. In this work, we extend these concepts into the path-conservative framework by reformulating the DG space in terms of equilibrium variables instead of conservative variables. Special modifications are made to the numerical flux, and linear segment paths are applied to the equilibrium functions. Theoretical analysis and numerical tests substantiate the well-balanced nature of both proposed schemes. Furthermore, the DG methods exhibit natural high-order accuracy and shock-capturing abilities with the incorporation of appropriate limiters. Numerical tests are presented to demonstrate the robustness and effectiveness of the proposed methods in achieving high resolution.

An outline of the rest paper is as follows.
Section \ref{se:mo} describes the mathematical model of the two-layer system and the DLM theory for nonconservative products. 
Section \ref{se:wb_s} presents a new PCDG method, which can maintain the ``lake-at-rest" steady state for two-layer shallow water equations.
In Section \ref{se:wb_m}, the PCDG method is proposed for the generalized moving equilibria.
Numerical results in different circumstances are illustrated in Section \ref{se:nu} to validate that the PCDG approach is well-balanced, high-order accurate, robust and shock-capturing.
Finally, some concluding remarks are provided in Section \ref{se:co}.

\section{Model description and DLM theory}\label{se:mo}
To begin with, we introduce the governing equations considered in this paper, along with the corresponding steady state solutions. We also provide a brief summary of the DLM theory.

\subsection{Two-layer system}\label{subse:tl}
The one-dimensional version of two-layer shallow water equations is given by
\begin{equation}\label{2LSWE_1d}
  \left\{\begin{array}{l}
    (h_1)_t+(h_1 u_1)_{x}=0, \\
    (h_1 u_1)_t+(h_1 u_1^2+\frac{1}{2}gh_1^2 )_x = -gh_1 b_x - gh_1 (h_2)_x,\\
    (h_2)_t+(h_2 u_2)_{x}=0, \\
    (h_2 u_2)_t+(h_2 u_2^2+\frac{1}{2}gh_2^2 )_x = -gh_2 b_x - grh_2 (h_1)_x,
  \end{array}\right.
\end{equation}
where the subscript 1 refers to the upper layer and index 2 to the lower one. $h_i$, ${u}_i$ and ${m}_i = h_i {u}_i$ represent the thickness, depth-averaged velocities and the discharges in each layer, while $b$ represents the nonflat bottom topography, and $g$ is the gravitational acceleration. Each layer is assumed to have a constant density $\rho_i$ and $r:={\rho_1}/{\rho_2}<1$ is the density ratio.
For simplicity, topography effects as well as friction with the bottom and between layers have been omitted here. The conservative variables  are denoted as
\begin{equation}\label{2LSWE_co}
\boldsymbol {u} = (h_1,m_1,h_2,m_2)^T,
\end{equation}
and the superscript $T$ denoting the transpose.
It can be expressed in the form of nonconservative balance laws as
\begin{equation}\label{2LSWE}
  \begin{aligned} 
   \boldsymbol {u}_t + \boldsymbol{f}(\boldsymbol {u})_x +  \mathcal{G}(\boldsymbol {u})\boldsymbol {u}_x =\boldsymbol {r}(\boldsymbol {u},b)
  \end{aligned}
\end{equation}
with
\begin{equation*}\scriptsize
   \boldsymbol{f}(\boldsymbol {u}) =
      \left[\begin{array}{c}
        m_1 \\
        h_1 u_1^2 + \frac{1}{2}gh_1^2 \\
        m_2 \\
        h_2 u_2^2 + \frac{1}{2}gh_2^2
      \end{array}\right], \
  \mathcal{G}(\boldsymbol {u}) =
      \left[\begin{array}{cccc}
            0 & 0 & 0 & 0  \\
            0 & 0 & gh_1 & 0 \\
            0 & 0 & 0 & 0  \\
            grh_2 & 0 & 0 & 0  \\
      \end{array}\right], \
  \boldsymbol {r}(\boldsymbol {u},b) =
    \left[\begin{array}{c}
    0\\
    -gh_1 b_x\\
    0 \\
    -gh_2 b_x
  \end{array}\right].
\end{equation*}

This model allows for non-trivial steady state solutions. The well-known generalized moving water equilibrium state solutions in one dimension are given by
\begin{equation}\label{moving_water:2LSWE}
  \begin{aligned}
    E_1 := \frac{1}{2}u_1^2 + g(h_1 + h_2 + b) = \text{constant}, \quad &m_1 = \text{constant}, \\
    E_2 := \frac{1}{2}u_2^2 + g(r h_1 + h_2 + b) = \text{constant}, \quad &m_2= \text{constant}.
  \end{aligned}
  \end{equation}
Here $m_1,m_2$ and $E_1, E_2$ are the moving water equilibrium variables that represent the specific mass flow rate and energy for the layers. Similar to the single-layer shallow water equations, no general form of the moving water equilibria exists in two dimensions. When the velocity reduces to zero, we get a special case of the generalized steady state (\ref{moving_water:2LSWE})
\begin{equation}\label{still_water:2LSWE}
  \begin{aligned}
    h_1 = \text{constant}, \quad w := h_2 + b = \text{constant}, \quad &u_1 = u_2 = 0,
  \end{aligned}
\end{equation}
which is known as the still water equilibrium, often likened to the ``lake at rest". Here $w$ denotes the interface between the upper and lower layers. We also denote $\varepsilon:= w+h_1$ to be the water surface.

The Jacobi matrix $\mathcal{A}(\boldsymbol u)$ in (\ref{ncs}), which contains the flux gradient terms as well as the coupling terms is defined by
\begin{equation}\label{Jaco}
    \mathcal{A}(\boldsymbol u)  =  \frac{\partial \boldsymbol{f}}{\partial \boldsymbol u} + \mathcal{G}(\boldsymbol {u})  =
    \left[\begin{array}{cccc}
      0 & 1 & 0 & 0  \\
      c_1^2 - u_1^2 & 2u_1 & c_1^2 & 0 \\
      0 & 0 & 0 & 1  \\
      rc_2^2 & 0 & c_2^2 - u_2^2 & 2u_2  \\
  \end{array}\right],
\end{equation}
where $c_i^2 = gh_i,i=1,2$ is the propagation celerity of internal and external waves. The eigenvalues $\lambda_{k}, k=1,\dots,4$ are to be determined from the characteristic polynomial
\begin{equation}\label{eig}
  P(\lambda) = \left((\lambda-u_1)^2- gh_1\right)\left((\lambda-u_2)^2- gh_2\right) - rg^2h_1 h_2.
\end{equation}
However, the explicit roots of the quartic equation (\ref{eig}) are too complex and not readily available. 
In the case of $r\thicksim 0$, the coupling terms do not significantly affect the nature of the system. Therefore,  we are mainly concerned with the case of $r\thicksim1$ and $u_1\thicksim u_2$, which is typical for oceanographic flows. Via applying the first-order expansion to $u_2-u_1$, \cite{schijf1953theoretical} provides four eigenvalues, two external and two internal,  corresponding to barotropic and baroclinic components of the flow,
\begin{equation}\label{ev}
  \begin{aligned}
    \lambda_{ext}^{\pm} = & U_m \pm \sqrt{g(h_1+h_2)}, \\
    \lambda_{int}^{\pm} = & U_c \pm \sqrt{g_r\dfrac{h_1 h_2}{h_1+h_2}\left[1-\dfrac{(u_1-u_2)^2}{g_r(h_1+h_2)} \right]}.
  \end{aligned}
\end{equation}
Here
\begin{equation*}
  U_m = \dfrac{h_1u_1+h_2u_2}{h_1+h_2}, \quad U_c = \dfrac{h_1u_2+h_2u_1}{h_1+h_2},
\end{equation*}
and $g_r = (1-r)g$ is the reduced gravity. From (\ref{ev}), it suggests that the system (\ref{2LSWE_1d}) is conditional hyperbolic, provided
\begin{equation}\label{ch}
(u_1-u_2)^2<g_r(h_1+h_2).
\end{equation}
Recently, Krvavica et al. \cite{krvavica2018analytical} provided an efficient alternative to a numerical eigensolver \cite{schijf1953theoretical}, based on the analytical solution to the quartic equation (\ref{eig}). We utilize the approach from \cite{krvavica2018analytical} for our simulation.

\subsection{DLM theory}\label{subse:DLM}
Before proceeding further, we provide a brief introduction to the definition of nonconservative products associated with the choice of a family of paths in the phase domain $ \Omega\subset\mathbb{R}^l$, here $l$ represents the number of components. This definition was developed by Dal Maso, LeFloch, and Murat in \cite{dal1995definition}.
\begin{definition}
  A family of paths in $\Omega\subset\mathbb{R}^l$ is a locally Lipschitz map
  $$\boldsymbol{\phi}: [0,1]\times \Omega \times \Omega \mapsto \Omega $$
  such that
  \begin{itemize}
    \setlength{\parskip}{0.0em}
    \item [(1)] For any $\boldsymbol {u}_L, \boldsymbol {u}_R\in \Omega$, $\boldsymbol{\phi}(0;\boldsymbol {u}_L,\boldsymbol {u}_R) = \boldsymbol {u}_L, \ \boldsymbol{\phi}(1;\boldsymbol {u}_L,\boldsymbol {u}_R) = \boldsymbol {u}_R$.
    \item [(2)] For every bounded set $\mathcal{O}\subset\Omega$, there exists a constant $k$ such that for any $s\in[0,1]$ and $\boldsymbol {u}_L, \boldsymbol {u}_R\in \mathcal{O}$,
    \begin{equation*}
      \left| \dfrac{\partial \boldsymbol{\phi}}{\partial \tau}(\tau;\boldsymbol {u}_L,\boldsymbol {u}_R) \right| \leq k |\boldsymbol {u}_R-\boldsymbol {u}_L|.
    \end{equation*}
    \item [(3)] For every bounded set $\mathcal{O}\subset\Omega$, there exists a constant $K$ such that for any $s\in[0,1]$ and $\boldsymbol {u}_L^1, \boldsymbol {u}_L^2, \boldsymbol {u}_R^1, \boldsymbol {u}_R^2\in \mathcal{O}$,
    \begin{equation*}
      \left| \dfrac{\partial \boldsymbol{\phi}}{\partial \tau}(\tau;\boldsymbol {u}_L^1,\boldsymbol {u}_R^1) - \dfrac{\partial \boldsymbol{\phi}}{\partial \tau}(\tau;\boldsymbol {u}_L^2,\boldsymbol {u}_R^2) \right| \leq K \left(|\boldsymbol {u}_L^2-\boldsymbol {u}_L^1| + |\boldsymbol {u}_R^2-\boldsymbol {u}_R^1|\right).
    \end{equation*}
  \end{itemize}
\end{definition}
\begin{theorem}(Dal Maso, LeFloch, and Murat (DLM\cite{dal1995definition}))
  Let $\boldsymbol {u}: [a,b]\rightarrow \mathbb{R}^l$ be a function of bounded variation and $\mathcal{A}: \mathbb{R}^l \rightarrow \mathbb{R}^{l\times l}$ be a smooth locally bounded matrix-valued function. Then, there exists a unique real-valued bounded Borel measure $\mu$ on $[a,b]$ characterized by the following two properties:
  \begin{itemize}
    \item [(1)] If $\boldsymbol {u}$ is continuous on a Borel set $B\subset [a,b]$, then
    \begin{equation*}
      \mu(B) = \int_{B} \mathcal{A}(\boldsymbol {u})\boldsymbol {u}_x \ dx.
    \end{equation*}
    \item [(2)] If $\boldsymbol {u}$ is discontinuous at a point $x_0\in [a,b]$, then
    \begin{equation*}
      \mu(x_0) = \left(\int_{0}^{1}  \mathcal{A}(\boldsymbol{\phi}(\tau;\boldsymbol {u}(x_0^-),\boldsymbol {u}(x_0^+))) \dfrac{\partial \boldsymbol{\phi}}{\partial \tau}(\tau; \boldsymbol {u}(x_0^-),\boldsymbol {u}(x_0^+)) \ d\tau \right) \delta(x_0).
    \end{equation*}
    Here $\delta(x_0)$ is the Dirac measure placed at $x_0$.
  \end{itemize}
 The Borel measure $\mu$ is called nonconservative product and can be written as $[\mathcal{A}(\boldsymbol {u}) \boldsymbol {u}_x]_{\boldsymbol{\phi}} $.
\end{theorem}
According to the definition, across a discontinuity, a weak solution must satisfy the generalized Rankine-Hugoniot jump condition:
\begin{equation}\label{grh}
  \int_{0}^{1} \left(\xi \mathcal{I} - \mathcal{A}(\boldsymbol{\phi}(\tau;\boldsymbol {u}^-,\boldsymbol {u}^+)) \right)\dfrac{\partial \boldsymbol{\phi}}{\partial \tau}(\tau; \boldsymbol {u}^-,\boldsymbol {u}^+) \ d\tau  = 0,
\end{equation}
where $\xi$ is the speed of propagation of the discontinuity, $\mathcal{I}$ is the identity matrix, and $\boldsymbol {u}^-$, $\boldsymbol {u}^+$ are the left and right limits of the solution at the discontinuity. Notice that if the nonconservative products can be written into divergence form, (\ref{grh}) will reduce to the traditional Rankine-Hugoniot jump condition for conservation laws.

\begin{remark}
  We can observe that the choice of paths strongly influences the weak solution as well as the numerical scheme. How to choose a good family of paths is still an open question. See \cite{castro2017well,abgrall2010comment} for further discussion. It is not in the scope of this paper. The segment path which adopted in this paper is the simplest and most commonly used.
\end{remark}

\section{Still water equilibria preserving PCDG scheme}\label{se:wb_s}
In this section, we design a new well-balanced path-conservative discontinuous Galerkin method for the two-layer shallow water equations, aimed at precisely maintaining the still water equilibrium state (\ref{still_water:2LSWE}).

\subsection{Notations}\label{subse:no}
We consider the numerical discretization on a $d$-dimensional computational domain $\mathcal{D}\subset\mathbb{R}^d$. Let $\mathcal {T}$ be a family of partitions such that
$$\mathcal {D} = \bigcup\{\mathcal {K}|\mathcal {K}\in \mathcal {T}\}.$$
Specifically speaking, in the one-dimensional case, assuming that the domain is divided into $nx$ non-overlapping subintervals, 
namely the element $\mathcal {K}$ is taken as the cell $\mathcal I_j$ and
$$\mathcal {D} = \bigcup   _{j=1}^{nx} \mathcal I_j = \bigcup  _{j=1}^{nx}[x_{j-\frac{1}{2}}, x_{j+\frac{1}{2}}].$$
The $x_{j\pm \frac{1}{2}}$ is the left or right cell interface.
In two-dimensional case, the Cartesian mesh $\mathcal {T}$ is discretized by rectangular cells $\mathcal I_{ij}$, i.e.
$$\mathcal {D} = \bigcup  _{i=1}^{nx}\bigcup _{j=1}^{ny}{\mathcal I_{ij}} = \bigcup_{i=1}^{nx}\bigcup _{j=1}^{ny} [x_{i-\frac{1}{2}},x_{i+\frac{1}{2}}] \times [y_{j-\frac{1}{2}},y_{j+\frac{1}{2}}].$$

The finite dimensional test function space is defined in the following way
 \begin{equation}\label{fes1}
\mathcal{W}_h^k:= \left\{ w: w|_{\mathcal {K}}\in P^k(\mathcal {K}), \forall \mathcal {K}\in \mathcal {T}  \right\},
\end{equation}
where $P^k(\mathcal {K})$ denotes the space of polynomials in cell $\mathcal {K}$ of degree at most $k$, and
 \begin{equation}\label{fes2}
\boldsymbol{\mathcal{V}}_h^k:= \left\{ \boldsymbol \gamma: \boldsymbol \gamma = (\gamma_1,\dots, \gamma_{l})^T | \ \gamma_1,\dots,\gamma_{l} \in  \mathcal{W}_h^k  \right\}
\end{equation}
is the piecewise polynomial finite element space with vectors consisting of $l$ components.
For any unknown variable $u\in \mathcal{W}_h^k$, we still denote its numerical approximation as $u$ with abuse of notation.
The values $u^{-},u^{+}$ denote the left and right side limit states at a given cell interface $x_{j+\frac{1}{2}}$ in one-dimension, with the definition
\begin{equation*}
  \begin{aligned}
  u_{j+\frac{1}{2}}^- := \lim _{\epsilon \rightarrow 0^+} u(x - \epsilon), \ u_{j+\frac{1}{2}}^+ := \lim _{\epsilon \rightarrow 0^+} u(x + \epsilon),
  \end{aligned}
  \end{equation*}
  and similar notations apply to the two-dimensional case.

\subsection{One-dimensional PCDG scheme for still water equilibrium}\label{subse:1d}
Inspired by the framework developed in \cite{rhebergen2008discontinuous}, we first write the system in the following quasilinear form
\begin{equation}\label{re_2LSWE1D}
  \begin{aligned} 
  \boldsymbol{v}_t + {\boldsymbol{f}}(\boldsymbol{v})_x + \mathcal{G}(\boldsymbol{v})\boldsymbol{v}_x =0.
  \end{aligned}
\end{equation}
In order to derive a still water equilibrium preserving path-conservative DG space discretization, as well as avoiding the introduction of the numerical diffusion on the bottom function $b$, 
we rewrite the equation (\ref{2LSWE_1d}) as
\begin{equation}\label{re2LSWE_1d}
  \left\{\begin{array}{l}
    (h_1)_t+(m_1)_{x}=0, \\
    (m_1)_t+(\frac{m_1^2}{h_1}+\frac{1}{2}gh_1^2 )_x=-gh_1 w_x,\\
    (w)_t+(m_2)_{x}=0, \\
    (m_2)_t+(\frac{m_2^2}{w-b}+\frac{1}{2}gw^2 )_x = gb w_x - gr(w-b) (h_1)_x,
  \end{array}\right.
\end{equation}
which has the form of (\ref{re_2LSWE1D}). Here
\begin{equation}\label{equ_s}
  \boldsymbol{v}=(h_1,m_1,w,m_2)^T
\end{equation}
is the one-dimensional still water equilibrium variables that are constant at the ``lake-at-rest'' steady state. The flux gradient terms and the coupling terms in (\ref{re_2LSWE1D}) are given by
\begin{equation*}
  \boldsymbol{f}(\boldsymbol{v}) =
      \left[\begin{array}{c}
        m_1 \\
        \frac{m_1^2}{h_1}+\frac{1}{2}gh_1^2 \\
        m_2 \\
        \frac{m_2^2}{w-b}+\frac{1}{2}gw^2
      \end{array}\right], \quad
       \mathcal{G}(\boldsymbol{v}) =
      \left[\begin{array}{cccc}
            0 & 0 & 0 & 0  \\
            0 & 0 & gh_1 & 0 \\
            0 & 0 & 0 & 0  \\
            gr(w-b) & 0 & -gb & 0  \\
      \end{array}\right].
\end{equation*}
By applying the definition for the nonconservative products given in \cite{dal1995definition}, we define $\boldsymbol{\phi}^S_{j+\frac{1}{2}}(\tau)$ to be a Lipschitz continuous path function in the phase space for still water equilibrium preserving PCDG scheme. The well-balanced PCDG scheme reads as: 
Find $\boldsymbol{v}\in \boldsymbol{\mathcal{V}}_h^k$, s.t. for any test function $\boldsymbol{\varphi}\in \boldsymbol{\mathcal{V}}_h^k$ we have 
\begin{equation}\label{scheme:2LSWE_still}
  \dfrac{d}{dt}\int_{\mathcal I_j}\boldsymbol{v} \cdot \boldsymbol{\varphi} \ dx = \text{RHS}_j^S(\boldsymbol{v},b,\boldsymbol{\varphi})
\end{equation}
with
\begin{equation}\label{rhs_s}
  \begin{aligned}
   \text{RHS}_j^S(\boldsymbol{v},b,\boldsymbol{\varphi})=& 
    \int_{\mathcal I_j}\boldsymbol{f}(\boldsymbol{v}) \cdot \boldsymbol{\varphi}_x \ dx - 
    \widehat{\boldsymbol{f}}_{j+\frac{1}{2}} \cdot \boldsymbol{\varphi}_{j+\frac{1}{2}}^- + 
    \widehat{\boldsymbol{f}}_{j-\frac{1}{2}} \cdot \boldsymbol{\varphi}_{j-\frac{1}{2}}^+ \\
    -&\int_{\mathcal I_j}  \mathcal{G}(\boldsymbol{v})\boldsymbol{v}_x \cdot \boldsymbol{\varphi}\ dx 
      - \dfrac{1}{2} \boldsymbol{\varphi}_{j+\frac{1}{2}}^- \cdot \mathscr{g}_{\boldsymbol{\phi}^S,{j+\frac{1}{2}}} 
      - \dfrac{1}{2} \boldsymbol{\varphi}_{j-\frac{1}{2}}^+ \cdot \mathscr{g}_{\boldsymbol{\phi}^S,{j-\frac{1}{2}}}. 
  \end{aligned}
\end{equation}
Taking the stability of the scheme into consideration, we choose the simplest monotone Lax-Friedrichs numerical flux,
\begin{equation}\label{LF}
  \begin{aligned}
    \widehat{\boldsymbol{f}}_{j+\frac{1}{2}} =\widehat {\boldsymbol{f}}(\boldsymbol{v}_{j+\frac{1}{2}}^{-},\boldsymbol{v}_{j+\frac{1}{2}}^{+}) &= \dfrac{1}{2}\left(\boldsymbol{f}(\boldsymbol{v}_{j+\frac{1}{2}}^-) + \boldsymbol{f}( \boldsymbol{v}_{j+\frac{1}{2}}^+)  \right) -\dfrac{\alpha}{2}(\boldsymbol{v}_{j+\frac{1}{2}}^{+} - \boldsymbol{v}_{j+\frac{1}{2}}^{-}), 
  \end{aligned}
\end{equation}
where 
\begin{equation}\label{LF_alpha}
  \alpha = \max\limits_{\boldsymbol{v}} \{|\lambda_1^S(\boldsymbol{v})|,\dots,|\lambda_4^S(\boldsymbol{v})|\}.
\end{equation}
Here, $\lambda_k^S, \; k = 1, \cdots, 4$ are the eigenvalues of the Jacobi matrix $\frac{\partial \boldsymbol{f}}{\partial \boldsymbol{v}}$, and the maximum is taken globally in our computation. 

We adopt the linear path function to connect the left and right equilibrium state values, which is defined as follows: 
\begin{equation}\label{linear_path}
  \boldsymbol{\phi}^S_{j+\frac{1}{2}}(\tau) :=
  \boldsymbol{\phi}^S(\tau; \boldsymbol{v}_{j+\frac{1}{2}}^-,\boldsymbol{v}_{j+\frac{1}{2}}^+) =
  \boldsymbol{v}_{j+\frac{1}{2}}^- + \tau(\boldsymbol{v}_{j+\frac{1}{2}}^+ - \boldsymbol{v}_{j+\frac{1}{2}}^-).
\end{equation}
The path integral term in (\ref{rhs_s}) which reflects the contributions of the nonconservative products, is given by
\begin{equation}\label{path_s}
  \mathscr{g}_{\boldsymbol{\phi}^S,{j+\frac{1}{2}}} =  
    \int_0^1 \mathcal{G}(\boldsymbol{\phi}^S_{j+\frac{1}{2}}(\tau)) 
    \frac{\partial}{\partial \tau}
  { \boldsymbol{\phi}}^S_{j+\frac{1}{2}}(\tau) \ d\tau.
\end{equation}
By using the linear segment path, the second and fourth components of (\ref{path_s}) can be simplified into 
\begin{equation*}
  \begin{aligned}
    \mathscr{g}_{\boldsymbol{\phi}^{S},{j+\frac{1}{2}}}^{[2]} 
    = & \int_0^1  g \left((h_1)_{j+\frac{1}{2}}^{-} + \tau \left((h_1)_{j+\frac{1}{2}}^{+} - (h_1)_{j+\frac{1}{2}}^{-} \right) \right) (w_{j+\frac{1}{2}}^{+} - w_{j+\frac{1}{2}}^{-})\ d\tau  \\
    = &\  \dfrac{g}{2}  \left((h_1)_{j+\frac{1}{2}}^{-} + (h_1)_{j+\frac{1}{2}}^{+} \right) (w_{j+\frac{1}{2}}^{+} - w_{j+\frac{1}{2}}^{-}),
  \end{aligned}
\end{equation*}
and
\begin{align*}
    \mathscr{g}_{\boldsymbol{\phi}^{S},{j+\frac{1}{2}}}^{[4]} 
    = -\dfrac{g}{2}(b_{j+\frac{1}{2}}^{-} + b_{j+\frac{1}{2}}^{+}) (w_{j+\frac{1}{2}}^{+} - w_{j+\frac{1}{2}}^{-}) + 
       \dfrac{gr}{2}(w_{j+\frac{1}{2}}^{-} - b_{j+\frac{1}{2}}^{-} + w_{j+\frac{1}{2}}^{+} - b_{j+\frac{1}{2}}^{+}) \left((h_1)_{j+\frac{1}{2}}^{+} - (h_1)_{j+\frac{1}{2}}^{-} \right).
\end{align*} 

For hyperbolic conservation laws, the semi-discrete scheme is usually coupled with high-order strong stability preserving Runge-Kutta (SSP-RK) or multistep time discretization methods \cite{gottlieb2001strong, shu1988total}. This also applies to nonconservative hyperbolic systems. Here, we employ a third-order SSP-RK method for the time discretization,
\begin{equation}\label{3rk_s}
  \begin{aligned}
  & \int_{I_j} \boldsymbol{v}^{(1)} \cdot \boldsymbol{\varphi} \ dx 
  = \int_{I_j} \boldsymbol{v}^n \cdot \boldsymbol{\varphi} \ dx 
  + \Delta t  \ \text{RHS}_j^S(\boldsymbol{v}^n,b,\boldsymbol{\varphi}), \\ 
  & \int_{I_j} \boldsymbol{v}^{(2)} \cdot \boldsymbol{\varphi} \ dx 
  = \dfrac{3}{4}\int_{I_j} \boldsymbol{v}^{n} \cdot \boldsymbol{\varphi} \ dx 
  + \dfrac{1}{4}\left(\int_{I_j} \boldsymbol{v}^{(1)} \cdot \boldsymbol{\varphi} \ dx 
  + \Delta t  \ \text{RHS}_j^S(\boldsymbol{v}^{(1)},b,\boldsymbol{\varphi}) \right),\\ 
  & \int_{I_j} \boldsymbol{v}^{n+1} \cdot \boldsymbol{\varphi} \ dx 
  = \dfrac{1}{3}\int_{I_j} \boldsymbol{v}^{n} \cdot \boldsymbol{\varphi} \ dx 
  + \dfrac{2}{3}\left(\int_{I_j} \boldsymbol{v}^{(2)} \cdot \boldsymbol{\varphi} \ dx 
  + \Delta t  \ \text{RHS}_j^S(\boldsymbol{v}^{(2)},b,\boldsymbol{\varphi}) \right), 
  \end{aligned}
\end{equation}
where $\text{RHS}_j^S(\boldsymbol{v},b,\boldsymbol{\varphi})$ is the spatial operator in (\ref{rhs_s}).

In the DG framework, spurious numerical oscillations can arise for high-order methods when encountering discontinuous solutions. Consequently, slope limiters may be necessary after each time stage of SSP-RK methods. In our experiments, we adopt the total variation bounded (TVB) limiter \cite{cockburn1989tvb3,cockburn1998runge5}. The limiting procedure does not compromise the well-balanced property, as we have transformed the variables of our two-layer shallow water equations to equilibrium variables (\ref{equ_s}), which remain constant for still water equilibrium. 

\subsection{Two-dimensional PCDG scheme for still water equilibrium}\label{subse:2d}
In this section, we extend the well-balanced PCDG method for two-layer shallow water equations to two-dimensional space, which has the following form: 
\begin{equation}\label{re2LSWE_2d}
  \left\{\begin{array}{l}
    (h_1)_t + (m_1)_{x} + (n_1)_{y} = 0, \\
    (m_1)_t + (\frac{m_1^2}{h_1}+\frac{1}{2}gh_1^2 )_x + (\frac{m_1n_1}{h_1})_y = -gh_1 w_x,\\
    (n_1)_t + (\frac{m_1n_1}{h_1})_x + (\frac{n_1^2}{h_1}+\frac{1}{2}gh_1^2 )_y = -gh_1 w_y,\\
    (w)_t   + (m_2)_{x} + (n_2)_{y} = 0, \\
    (m_2)_t + (\frac{m_2^2}{w-b}+\frac{1}{2}gw^2 )_x + (\frac{m_2n_2}{w-b})_y = gb w_x - gr(w-b) (h_1)_x, \\
    (n_2)_t + (\frac{m_2n_2}{w-b})_x + (\frac{n_2^2}{w-b}+\frac{1}{2}gw^2 )_y = gb w_y - gr(w-b) (h_1)_y,
  \end{array}\right.
\end{equation}
where $m_i = h_iu_i$, $n_i = h_i v_i$ are the discharges in $x$- and $y$-directions. 
It can be written in the quasilinear form 
\begin{equation}\label{re_2LSWE2D}
  \begin{aligned}
  \boldsymbol{v}_t + \boldsymbol{f}_1(\boldsymbol{v})_x + \boldsymbol{f}_2(\boldsymbol{v})_y + \mathcal{G}_1(\boldsymbol{v})\boldsymbol{v}_x + \mathcal{G}_2(\boldsymbol{v})\boldsymbol{v}_y =0
  \end{aligned}
\end{equation}
with $\boldsymbol{v} = (h_1,m_1,n_1,w,m_2,n_2)^T$ being the two-dimensional still water equilibrium variables, which corresponds to those defined in \eqref{equ_s}, and the flux terms
\begin{equation*}
  \begin{aligned}
  \boldsymbol{f}_1(\boldsymbol{v}) &= \left(m_1, \frac{m_1^2}{h_1}+\frac{1}{2}gh_1^2, \frac{m_1n_1}{h_1}, m_2, \frac{m_2^2}{w-b}+\frac{1}{2}gw^2,\frac{m_2n_2}{w-b} \right)^T,\\
  \boldsymbol{f}_2(\boldsymbol{v}) &= \left(n_1, \frac{m_1n_1}{h_1}, \frac{n_1^2}{h_1}+\frac{1}{2}gh_1^2, n_2, \frac{m_2n_2}{w-b}, \frac{n_2^2}{w-b}+\frac{1}{2}gw^2 \right)^T,
  \end{aligned}
\end{equation*}
and the coupling terms
\begin{equation*}\scriptsize
  \mathcal{G}_1(\boldsymbol{v}) = 
  \left[\begin{array}{cccccc}
    0 & 0 & 0 & 0 & 0 & 0 \\
    0 & 0 & 0 & -gh_1 & 0 & 0 \\
    0 & 0 & 0 & 0 & 0 & 0 \\
    0 & 0 & 0 & 0 & 0 & 0 \\
    -gr(w-b) & 0 & 0 & gb & 0 & 0 \\
    0 & 0 & 0 & 0 & 0 & 0 \\
  \end{array}\right], \
  \mathcal{G}_2(\boldsymbol{v}) = 
  \left[\begin{array}{cccccc}
    0 & 0 & 0 & 0 & 0 & 0 \\
    0 & 0 & 0 & 0 & 0 & 0 \\
    0 & 0 & 0 & -gh_1 & 0 & 0 \\
    0 & 0 & 0 & 0 & 0 & 0 \\
    0 & 0 & 0 & 0 & 0 & 0 \\
    -gr(w-b) & 0 & 0 & gb & 0 & 0 \\
  \end{array}\right].
\end{equation*}
The two-dimensional well-balanced PCDG method on the Cartesian mesh is then obtained precisely as its one-dimensional prototype developed in Section \ref{subse:1d}. Find $\boldsymbol{v}\in \boldsymbol{\mathcal{V}}_h^k$, s.t. for any test function $\boldsymbol{\varphi}\in \boldsymbol{\mathcal{V}}_h^k$ we have
\begin{equation}\label{scheme:2LSWE_s2d}
  \dfrac{d}{dt}\iint_{\mathcal I_{ij}}\boldsymbol{v}\cdot \boldsymbol{\varphi} \ dx dy = \text{RHS}_{ij}^S(\boldsymbol{v},b,\boldsymbol{\varphi})
\end{equation}
with
\begin{equation*}
  \begin{aligned}
   \text{RHS}_{ij}^S
    =& 
    \iint_{\mathcal I_{ij}}\boldsymbol{f}_1(\boldsymbol{v}) \cdot \boldsymbol{\varphi}_x \ dx dy - 
    \int_{y_{j-\frac{1}{2}}}^{y_{j+\frac{1}{2}}}
    \widehat{\boldsymbol{f}_1}_{i+\frac{1}{2},j} \cdot \boldsymbol{\varphi}(x_{i+\frac{1}{2}}^-,y) - 
    \widehat{\boldsymbol{f}_1}_{i-\frac{1}{2},j} \cdot \boldsymbol{\varphi}(x_{i-\frac{1}{2}}^+,y) \ dy \\
    +&
    \iint_{\mathcal I_{ij}}\boldsymbol{f}_2(\boldsymbol{v}) \cdot \boldsymbol{\varphi}_y \ dx dy - 
    \int_{x_{i-\frac{1}{2}}}^{x_{i+\frac{1}{2}}}
    \widehat{\boldsymbol{f}_2}_{i,j+\frac{1}{2}} \cdot \boldsymbol{\varphi}(x,y_{i+\frac{1}{2}}^-) - 
    \widehat{\boldsymbol{f}_2}_{i,j-\frac{1}{2}} \cdot \boldsymbol{\varphi}(x,y_{i-\frac{1}{2}}^+) \ dx \\
    -&\iint_{\mathcal I_{ij}}  \mathcal{G}_1(\boldsymbol{v})\boldsymbol{v}_x  \cdot \boldsymbol{\varphi}\ dx dy 
      - \int_{y_{j-\frac{1}{2}}}^{y_{j+\frac{1}{2}}} 
        \dfrac{1}{2} \boldsymbol{\varphi}(x_{i+\frac{1}{2}}^-,y) \cdot {\mathscr{g}_1}_{\boldsymbol{\phi}^S,{i+\frac{1}{2},j}} + 
        \dfrac{1}{2} \boldsymbol{\varphi}(x_{i-\frac{1}{2}}^+,y) \cdot {\mathscr{g}_1}_{\boldsymbol{\phi}^S,{i-\frac{1}{2},j}}\ dy \\
    -&\iint_{\mathcal I_{ij}} \mathcal{G}_2(\boldsymbol{v})\boldsymbol{v}_y  \cdot \boldsymbol{\varphi}\ dx dy 
      - \int_{x_{i-\frac{1}{2}}}^{x_{i+\frac{1}{2}}} 
        \dfrac{1}{2} \boldsymbol{\varphi}(x,y_{i+\frac{1}{2}}^-) \cdot {\mathscr{g}_2}_{\boldsymbol{\phi}^S,{i,j+\frac{1}{2}}} + 
        \dfrac{1}{2} \boldsymbol{\varphi}(x,y_{i-\frac{1}{2}}^+) \cdot {\mathscr{g}_2}_{\boldsymbol{\phi}^S,{i,j-\frac{1}{2}}}\ dx.
  \end{aligned}
\end{equation*}
We consistently utilize the Lax-Friedrichs numerical flux for the flux gradient terms $ \widehat{\boldsymbol{f}_1}$ and  $\widehat{\boldsymbol{f}_2}$ throughout our calculations, as it is both consistent and easily implemented. 

We define  the linear path function 
\begin{align}
  \boldsymbol{\phi}^S_{i+\frac{1}{2},j}(\tau) :=
  \boldsymbol{\phi}^S(\tau; \boldsymbol{v}_{i+\frac{1}{2},j}^-,\boldsymbol{v}_{i+\frac{1}{2},j}^+) =
  \boldsymbol{v}_{i+\frac{1}{2},j}^- + \tau(\boldsymbol{v}_{i+\frac{1}{2},j}^+ - \boldsymbol{v}_{i+\frac{1}{2},j}^-), \label{linear_path_y2d} \\
  \boldsymbol{\phi}^S_{i,j+\frac{1}{2}}(\tau) :=
  \boldsymbol{\phi}^S(\tau; \boldsymbol{v}_{i,j+\frac{1}{2}}^-,\boldsymbol{v}_{i,j+\frac{1}{2}}^+) =
  \boldsymbol{v}_{i,j+\frac{1}{2}}^- + \tau(\boldsymbol{v}_{i,j+\frac{1}{2}}^+ - \boldsymbol{v}_{i,j+\frac{1}{2}}^-). \label{linear_path_x2d} 
\end{align}
And the path integral terms are given by
  \begin{align} 
  {\mathscr{g}_1}_{\boldsymbol{\phi}^S,{i+\frac{1}{2},j}} = 
    \int_0^1 \mathcal{G}_1(\boldsymbol{\phi}^S_{i+\frac{1}{2},j}(\tau)) 
    \frac{\partial }{\partial \tau}{\boldsymbol{\phi}^S_{i+\frac{1}{2},j}(\tau)} \ d\tau, \label{path_s12d} \\
  {\mathscr{g}_2}_{\boldsymbol{\phi}^S,{i,j+\frac{1}{2}}} = 
    \int_0^1 \mathcal{G}_2(\boldsymbol{\phi}^S_{i,j+\frac{1}{2}}(\tau)) 
     \frac{\partial }{\partial \tau}
     \boldsymbol{\phi}^S_{i,j+\frac{1}{2}}(\tau)\ d\tau, \label{path_s22d}
  \end{align}
 Similar to the one-dimensional case, obtaining simplification results for applying the aforementioned two paths to (\ref{path_s12d}) and (\ref{path_s22d}) is straightforward. Therefore, for the sake of brevity, we omit the details.

\subsection{Well-balanced property}
We present the still water well-balanced property for the semi-discrete PCDG method in the following theorems.

\begin{theorem}\label{wbs_semi}
  The semi-discrete path-conservative discontinuous Galerkin method (\ref{scheme:2LSWE_still}) and (\ref{scheme:2LSWE_s2d}) is exact for one- and two-dimensional still water equilibrium states (\ref{still_water:2LSWE}).
\end{theorem}
\begin{proof}
 We will only prove the case in one dimension, as the two-dimensional situation can be addressed using a similar approach. Assuming that the initial values are in an exact hydrostatic state, i.e. 
 \begin{equation}\label{vpvm}
  {\boldsymbol{v}}(x) = \text{constant} := c \  \Longrightarrow \ {\boldsymbol{v}}_{j+\frac{1}{2}}^- = {\boldsymbol{v}}_{j+\frac{1}{2}}^+.
 \end{equation}
 A straightforward calculation from (\ref{LF}) demonstrates that
 \begin{equation}\label{fh}
  \widehat{\boldsymbol{f}}_{j+\frac{1}{2}} =  \boldsymbol{f}(\boldsymbol{v}_{j+\frac{1}{2}}^-) = \boldsymbol{f}( \boldsymbol{v}_{j+\frac{1}{2}}^+).
 \end{equation}
 The definition of the path integral term (\ref{path_s}) and the linear segment path function (\ref{linear_path}), together with the fact (\ref{vpvm}) yields
 \begin{equation}\label{fg}
  \mathscr{g}_{\boldsymbol{\phi}^{S},{j+\frac{1}{2}}} = 0.
 \end{equation}
 By substituting (\ref{fh}) and (\ref{fg}) into (\ref{rhs_s}), we have
 \begin{align*}
  \text{RHS}_j 
 =&\int_{\mathcal I_j}\boldsymbol{f}(\boldsymbol{v}) \cdot \boldsymbol{\varphi}_x \ dx - 
   \boldsymbol{f}(\boldsymbol{v}_{j+\frac{1}{2}}^-) \cdot \boldsymbol{\varphi}_{j+\frac{1}{2}}^- + 
   \boldsymbol{f}(\boldsymbol{v}_{j-\frac{1}{2}}^+) \cdot \boldsymbol{\varphi}_{j-\frac{1}{2}}^+ -
   \int_{\mathcal I_j} \mathcal{G}(\boldsymbol{v})\boldsymbol{v}_x \cdot \boldsymbol{\varphi}\ dx \\
 =&-\int_{\mathcal I_j}\boldsymbol{f}(\boldsymbol{v})_x \cdot \boldsymbol{\varphi} \ dx -
    \int_{\mathcal I_j}  \mathcal{G}(\boldsymbol{v})\boldsymbol{v}_x \cdot \boldsymbol{\varphi}\ dx \\
 =&-\int_{\mathcal I_j} \left(\boldsymbol{f}(\boldsymbol{v})_x +  \mathcal{G}(\boldsymbol{v})\boldsymbol{v}_x\right) \cdot \boldsymbol{\varphi} \ dx = 0.
 \end{align*}
 The second equality follows from integration by parts. The last equality is owing to the steady state solution. The proof is completed.
\end{proof}

The same property holds for the fully-discrete PCDG scheme. A straightforward induction leads to the following property, and we omit the proof here.

\begin{theorem}\label{wbs_fully}
  The fully-discrete path-conservative discontinuous Galerkin method, space discretization (\ref{scheme:2LSWE_still}) - (\ref{linear_path}) and (\ref{scheme:2LSWE_s2d}) - (\ref{linear_path_x2d}) together with the third-order SSP-RK time discretization (\ref{3rk_s}), is exact for one- and two-dimensional still water equilibrium states (\ref{still_water:2LSWE}).
  \end{theorem}

\section{Moving water equilibria preserving PCDG scheme}\label{se:wb_m}
This section develops a path-conservative discontinuous Galerkin method for the moving water equilibria. Similar to the single-layer shallow water equations, we only discuss the case of one-dimensional since there is no general form of the moving water equilibrium states in two dimensions.

\subsection{Conservative and equilibrium variables}
The scheme designed in the last section for still water preservation cannot be generalized for moving water equilibria, which is significantly more challenging. To maintain the steady state of moving water (\ref{moving_water:2LSWE}), we first follow the same approach used in many path-conservative methods, such as \cite{munoz2011convergence, rhebergen2008discontinuous, diaz2013high}, by adding an additional equation that is trivially satisfied, 
\begin{equation*}
  b_t = 0
\end{equation*} 
to the system (\ref{2LSWE_1d}). Hence, we redefine the conservative variables to be 
$${\boldsymbol u}=(h_1,m_1,h_2,m_2,b)^T,$$ and rewrite the equation (\ref{2LSWE_1d}) as the quasilinear form
\begin{align}\label{re_2LSWE1D_ME}
  \boldsymbol {u}_t + \boldsymbol{f}(\boldsymbol {u})_x +  \mathcal{G}(\boldsymbol {u})\boldsymbol {u}_x =0,
  \end{align}
 with
\begin{equation*}
\scriptsize
  \boldsymbol{f}(\boldsymbol {u}) =
  \left[\begin{array}{c}
      m_1 \\
      \frac{m_1^2}{h_1} +\frac{1}{2}gh_1^2 \\
      m_2 \\
      \frac{m_2^2}{h_2}+\frac{1}{2}gh_2^2 \\
      0
    \end{array}\right], \quad
    \mathcal{G}(\boldsymbol {u}) =
    \left[\begin{array}{ccccc}
      0 & 0 & 0 & 0 & 0 \\
      0 & 0 & gh_1 & 0 & gh_1 \\
      0 & 0 & 0 & 0 & 0 \\
      grh_2 & 0 & 0 & 0 & gh_2 \\
      0 & 0 & 0 & 0 & 0 \\
    \end{array}\right].
\end{equation*}
For many interesting systems of balance laws, in general, the left-hand side of (\ref{re_2LSWE1D_ME}) can be expressed as
\begin{equation}\label{relation}
  \boldsymbol{f}(\boldsymbol {u})_x +  \mathcal{G}(\boldsymbol {u})\boldsymbol {u}_x = \mathcal{L}(\boldsymbol {u})\widetilde{\boldsymbol{v}}(\boldsymbol {u})_x,
\end{equation}
where $\mathcal{L}(\boldsymbol {u})$ is a matrix and $\widetilde{\boldsymbol{v}}(\boldsymbol {u})$ is the vector of equilibrium functions. Here we consider the situation when the set of equilibrium functions $\widetilde{\boldsymbol{v}}(\boldsymbol {u})$ are given by a simple algebraic relation. The caveat is that, generally, $\widetilde{\boldsymbol{v}}$ may be a global (in space) integral quantity which is much more complicated and will be taken into consideration in our future work. Simple calculation yields that the steady state (\ref{moving_water:2LSWE}) satisfies the relation (\ref{relation}), with
$$\scriptsize
\mathcal{L}(\boldsymbol {u}) = \left[\begin{array}{ccccc}
  0 & 1 & 0 & 0 & 0 \\
  h_1 & u_1 & 0 & 0 & 0 \\
  0 & 0 & 0 & 1 & 0 \\
  0 & 0 & h_2 & u_2 & 0 \\
  0 & 0 & 0 & 0 & 1 \\
\end{array}\right]$$ and
\begin{equation*}
\widetilde{\boldsymbol{v}} = (E_1,m_1,E_2,m_2,0)^T 
\end{equation*}
being the moving water equilibrium variables, which are constant at the generalized moving steady state. 

The key ingredient of constructing our well-balanced path-conservative DG method is the transformation between the conservative variables and the equilibrium variables. For the sake of better calculation, we add the bottom function to the equilibrium variables and denote as ${\boldsymbol{v}}=(E_1,m_1, E_2,m_2,b)^T$. On the one hand, we can directly obtain the transform function from conservative variables $\boldsymbol{u}$ to equilibrium variables $\boldsymbol{v}$, which is denoted as
\begin{equation}\label{u_v:2LSWE}
\boldsymbol{v} = \boldsymbol{v}(\boldsymbol {u})=\left(\begin{array}{c}
               E_1 \\
               m_1\\
               E_2\\
               m_2\\
               b
             \end{array}\right) =
             \left(\begin{array}{c}
               \frac{m_1^2}{2h_1^2} + g(h_1 + h_2 + b) \\
               m_1\\
               \frac{m_2^2}{2h_2^2} + g(r h_1 + h_2 + b) \\
               m_2\\
               b
             \end{array}\right).
\end{equation}
On the other hand, given $\boldsymbol{v}$ in advance, the inverse transform function $\boldsymbol{u}(\boldsymbol{v})$ (or simply $h_1$, $h_2$) cannot be directly computed due to the nonlinearity and coupling of the energy $E_1$, $E_2$. We rewrite the expression (\ref{moving_water:2LSWE}) as the following system of two cubic equations:
\begin{equation}\label{v_u:2LSWE}
  \begin{aligned}
    &Q_1 = gh_1^3 + (g(h_2+b)-E_1)h_1^2 + \dfrac{1}{2}m_1^2 = 0, \\
    &Q_2 = gh_2^3 + (g(rh_1+b)-E_2)h_2^2 + \dfrac{1}{2}m_2^2 = 0.
  \end{aligned}
\end{equation}
Different from the single-layer shallow water equations \cite{zhang2023moving}, despite the system consisting of two cubic equations, the upper and lower water heights do not have explicit expressions. Notice that if either $m_1$ or $m_2$ is zero, the system (\ref{v_u:2LSWE}) is significantly simplified. In the general case when $m_1\neq 0$ and $m_2\neq 0$, the system (\ref{v_u:2LSWE}) needs to be solved by a nonlinear iterative solver, such as Newton iteration. Another challenge is that one cannot analyze the properties of (\ref{v_u:2LSWE}), including convexity and monotonicity.
Thus the proper initial guess is important for the convergence to a correct solution, more details will be presented below.

\subsection{The semi-discrete PCDG scheme} 
For the sake of moving equilibrium preserving, we follow the procedure established in \cite{zhang2023moving} to rewrite the system (\ref{re_2LSWE1D_ME}) as
\begin{equation}\label{remov_2LSWE1D}
  \begin{aligned}
    \boldsymbol {u}(\boldsymbol{v})_t + \boldsymbol{f}(\boldsymbol {u}(\boldsymbol{v}))_x +  \mathcal{G}(\boldsymbol {u}(\boldsymbol{v}))\boldsymbol {u}(\boldsymbol{v})_x =0.
  \end{aligned}
\end{equation}
We would like to point out that the form of the equation itself remains unchanged in our calculation; only the conservative variables are handled as nonlinear functions of the equilibrium variables. Hence, instead of seeking an approximation to the conservative variables $\boldsymbol {u}$, we approximate the equilibrium variables $\boldsymbol{v}$ in the DG piecewise polynomial space $\boldsymbol{\mathcal{V}}_h^k$. This marks a significant departure from the traditional high-order well-balanced path-conservative DG methods. 

\subsubsection{PCDG scheme}
Let $\boldsymbol{\phi}^M_{j+\frac{1}{2}}(\tau)$ be a Lipschitz continuous path function in the phase space defined for moving equilibrium preserving PCDG scheme. If we directly extend the approach used in \cite{zhang2023moving} to this system, the resulting method leads to a non-well-balanced scheme. In order to preserve a wider class of moving equilibria exactly, several other techniques are employed.  The new path-conservative DG scheme is  formulated in the following way: 
Find $\boldsymbol{v}\in \boldsymbol{\mathcal{V}}_h^k$, s.t. for any test function $\boldsymbol{\varphi} = ({\varphi}_1,\dots,{\varphi}_5)^T \in \boldsymbol{\mathcal{V}}_h^k$ we have
\begin{equation}\label{scheme:2LSWE_mov}
\dfrac{d}{dt}\int_{\mathcal I_j}\boldsymbol{u}(\boldsymbol{v})\cdot \boldsymbol{\varphi} \ dx = \text{RHS}_j^M(\boldsymbol{v},\boldsymbol{\varphi})
\end{equation}
with
\begin{equation}\label{rhs_m}
  \begin{aligned}
  \text{RHS}_j^M(\boldsymbol{v},\boldsymbol{\varphi})=
  & \int_{\mathcal I_j}\boldsymbol{f}(\boldsymbol{u}(\boldsymbol{v})) \cdot \boldsymbol{\varphi}_x \ dx - 
    \widehat{\boldsymbol{f}}_{j+\frac{1}{2}}^{\text{mod}} \cdot \boldsymbol{\varphi}_{j+\frac{1}{2}}^- + 
    \widehat{\boldsymbol{f}}_{j-\frac{1}{2}}^{\text{mod}} \cdot \boldsymbol{\varphi}_{j-\frac{1}{2}}^+  \\
   -&\int_{\mathcal I_j}  \mathcal{G}(\boldsymbol {u}(\boldsymbol{v}))\boldsymbol {u}(\boldsymbol{v})_x  \cdot \boldsymbol{\varphi}\ dx 
    - \dfrac{1}{2} \boldsymbol{\varphi}_{j+\frac{1}{2}}^-  \cdot \mathscr{g}_{\boldsymbol{\phi}^M,{j+\frac{1}{2}}}
    - \dfrac{1}{2} \boldsymbol{\varphi}_{j-\frac{1}{2}}^+  \cdot \mathscr{g}_{\boldsymbol{\phi}^M,{j-\frac{1}{2}}}.
  \end{aligned}
\end{equation}

\subsubsection{Well-balanced treatment}
{\noindent \textbf{(a) Flux modification}}
\smallskip

Specifically, different from the hydrostatic reconstruction strategy used in \cite{zhang2023moving}, we only modify the Lax-Friedrichs numerical flux as
\begin{equation}\label{LF_modify}
  \begin{aligned}
    \widehat{\boldsymbol{f}}_{j+\frac{1}{2}}^{\text{mod}} =\widehat {\boldsymbol{f}}^{\text{mod}}(\boldsymbol u_{j+\frac{1}{2}}^{-},\boldsymbol u_{j+\frac{1}{2}}^{+}) &= \dfrac{1}{2}\left(\boldsymbol{f}(\boldsymbol u_{j+\frac{1}{2}}^-) + \boldsymbol{f}( \boldsymbol u_{j+\frac{1}{2}}^+)  \right) -\dfrac{\alpha}{2}(\boldsymbol u_{j+\frac{1}{2}}^{*,+} - \boldsymbol u_{j+\frac{1}{2}}^{*,-}),
  \end{aligned}
\end{equation}
with $\alpha$ given by 
\begin{equation}\label{LF_alpha_M}
  \alpha = \max\limits_{\boldsymbol{u}} \{|\lambda_1^M(\boldsymbol{u})|,\dots,|\lambda_5^M(\boldsymbol{u})|\},
\end{equation}
and $\lambda_k^M$ are the eigenvalues of the Jacobi matrix $\frac{\partial \boldsymbol{f}}{\partial \boldsymbol{u}}$.
Herein the boundary values $\boldsymbol u_{j+\frac{1}{2}}^{\pm}$ are acquired from the transform relationship (\ref{v_u:2LSWE}). The modified cell interface values are defined by
\begin{equation}\label{u_modify}
  \boldsymbol{u}_{j+\frac{1}{2}}^{*,\pm} = ((h_1)_{j+\frac{1}{2}}^{*,\pm},(m_1)_{j+\frac{1}{2}}^{\pm},(h_2)_{j+\frac{1}{2}}^{*,\pm},(m_2)_{j+\frac{1}{2}}^{\pm},b_{j+\frac{1}{2}}^{*}),
\end{equation}
with $\left( (h_1)_{j+\frac{1}{2}}^{*,\pm}, (h_2)_{j+\frac{1}{2}}^{*,\pm} \right)$ solved by the following nonlinear systems
\begin{equation}\label{u_cal}
  \begin{aligned}
   & (E_1)_{j+\frac{1}{2}}^{\pm} = \frac{((m_1)_{j+\frac{1}{2}}^{\pm})^2}{2((h_1)_{j+\frac{1}{2}}^{\pm,*})^2} + g((h_1)_{j+\frac{1}{2}}^{\pm,*} + (h_2)_{j+\frac{1}{2}}^{\pm,*} + b_{j+\frac{1}{2}}^{*}),\\
   & (E_2)_{j+\frac{1}{2}}^{\pm} = \frac{((m_2)_{j+\frac{1}{2}}^{\pm})^2}{2((h_2)_{j+\frac{1}{2}}^{\pm,*})^2} + g(r(h_1)_{j+\frac{1}{2}}^{\pm,*} + (h_2)_{j+\frac{1}{2}}^{\pm,*} + b_{j+\frac{1}{2}}^{*}), 
  \end{aligned}
\end{equation}
as well as
\begin{equation}\label{b_modify}
  b_{j+\frac{1}{2}}^{*} = \min{(b_{j+\frac{1}{2}}^{-},b_{j+\frac{1}{2}}^{+})}.
\end{equation}

\bigskip
{\noindent \textbf{(b) Evaluating the path integral term}}
\smallskip

We  define
\begin{equation}\label{linear_path-m}
  \boldsymbol{\phi}^M_{j+\frac{1}{2}}(\tau) :=
  \boldsymbol{\phi}^M(\tau; \boldsymbol{v}_{j+\frac{1}{2}}^-,\boldsymbol{v}_{j+\frac{1}{2}}^+) .
\end{equation}
Moreover, $\mathscr{g}_{\boldsymbol{\phi}^M,{j+\frac{1}{2}}}$ is the approximation of the path integral term
\begin{equation}\label{it_ap}
  \int_0^1 \mathcal{G}(\boldsymbol{\phi}^M_{j+\frac{1}{2}}(\tau))  \frac{\partial }{\partial \tau} \boldsymbol{\phi}^M_{j+\frac{1}{2}}(\tau)\ d\tau. 
\end{equation}
Straightforward evaluation of (\ref{it_ap}) is not capable of exactly preserving the steady state. To this end, we follow \cite{kurganov2023well} to provide a general framework to calculate this term in a well-balanced manner. We can see that at the cell interface $x=x_{j+\frac{1}{2}}$, the relation (\ref{relation}) is still valid and we have
\begin{equation}\label{re_x}
 \frac{\partial }{\partial \tau} \boldsymbol{f}(\boldsymbol{\phi}^M_{j+\frac{1}{2}}(\tau)) +\mathcal{G}(\boldsymbol{\phi}^M_{j+\frac{1}{2}}(\tau))  \frac{\partial }{\partial \tau}\boldsymbol{\phi}^M_{j+\frac{1}{2}}(\tau)  = \mathcal{L}(\boldsymbol{\phi}^M_{j+\frac{1}{2}}(\tau)) \frac{\partial }{\partial \tau}\widetilde{\boldsymbol{v}}(\boldsymbol{\phi}^M_{j+\frac{1}{2}}(\tau)).
\end{equation}
Plugging (\ref{re_x}) into the integral term (\ref{it_ap}), we obtain
\begin{equation}\label{ncp_m}
  \begin{aligned}
    \mathscr{g}_{\boldsymbol{\phi}^M,{j+\frac{1}{2}}} 
    =&\int_0^1 \mathcal{G}(\boldsymbol{\phi}^M_{j+\frac{1}{2}}(\tau)) \frac{\partial }{\partial \tau}  \boldsymbol{\phi}^M_{j+\frac{1}{2}}(\tau) \ d\tau \\
    \overset{(\ref{re_x})}{=} &
    \int_0^1 \mathcal{L}(\boldsymbol{\phi}^M_{j+\frac{1}{2}}(\tau)) \frac{\partial }{\partial \tau} \widetilde{\boldsymbol{v}}(\boldsymbol{\phi}^M_{j+\frac{1}{2}}(\tau)) \ d\tau - 
    \int_0^1 \frac{\partial }{\partial \tau} \boldsymbol{f}(\boldsymbol{\phi}^M_{j+\frac{1}{2}}(\tau)) \ d\tau \\
    =&\int_0^1 \mathcal{L}(\boldsymbol{\phi}^M_{j+\frac{1}{2}}(\tau)) \frac{\partial }{\partial \tau} \widetilde{\boldsymbol{v}}(\boldsymbol{\phi}^M_{j+\frac{1}{2}}(\tau)) \ d\tau - 
   \boldsymbol{f}(\boldsymbol{u}_{j+\frac{1}{2}}^+) + \boldsymbol{f}(\boldsymbol{u}_{j+\frac{1}{2}}^-).
  \end{aligned}
\end{equation}
The left and right states are connected with a linear segment path which is applied to the equilibrium functions
\begin{equation}\label{linear_path_V}
  \boldsymbol{\phi}^M_{j+\frac{1}{2}}(\tau) = \widetilde{\boldsymbol{v}}_{j+\frac{1}{2}}^- + \tau(\widetilde{\boldsymbol{v}}_{j+\frac{1}{2}}^+ - \widetilde{\boldsymbol{v}}_{j+\frac{1}{2}}^-).
\end{equation}
Such choice of the path connecting equilibrium variables (\ref{linear_path_V}) plays an important role in a well-balanced manner. Based on it, the integral term in the last line in (\ref{ncp_m}) will reduced to 
  \begin{equation*}
    \int_0^1 \mathcal{L}(\boldsymbol{\phi}^M_{j+\frac{1}{2}}(\tau)) \left(\widetilde{\boldsymbol{v}}_{j+\frac{1}{2}}^{+} - \widetilde{\boldsymbol{v}}_{j+\frac{1}{2}}^{-}\right)\ d\tau,
  \end{equation*}
  which is evaluated numerically, e.g., by the Simpson formula.
  
\bigskip
{\noindent \textbf{(c) Evaluating the integral of the coupling term}}
\smallskip

To ensure well-balanced preservation, another critical component is the evaluation of the integral of the nonconservative product term
  \begin{equation*}
    \int_{\mathcal I_j}  \mathcal{G}(\boldsymbol {u}(\boldsymbol{v}))\boldsymbol {u}(\boldsymbol{v})_x \cdot \boldsymbol{\varphi}\ dx,
  \end{equation*}
  which contains the coupling terms
  \begin{equation*}
    \int_{\mathcal I_j} -g h_1(\boldsymbol{v})(h_2(\boldsymbol{v}))_x, \quad \int_{\mathcal I_j} -gr h_2(\boldsymbol{v})(h_1(\boldsymbol{v}))_x.
  \end{equation*} 
  Since our computational variables are equilibrium variables $\boldsymbol{v}$, the conservative variables are non-polynomial functions of the equilibrium variables. Briefly speaking, $\boldsymbol{v}\in\boldsymbol{\mathcal{V}}_h^k$ while $\boldsymbol{u}\notin\boldsymbol{\mathcal{V}}_h^k$, so that there is no explicit expression for $(h_1)_x$ or $(h_2)_x$. Take $(h_1)_x$ as an example, chain rule gives 
  \begin{equation*}
    \dfrac{\partial h_1}{\partial x} = \dfrac{\partial h_1}{\partial \boldsymbol{v}}\dfrac{\partial \boldsymbol{v}}{\partial x} = 
    \dfrac{\partial h_1}{\partial E_1}\dfrac{\partial E_1}{\partial x} + 
    \dfrac{\partial h_1}{\partial m_1}\dfrac{\partial m_1}{\partial x} + 
    \dfrac{\partial h_1}{\partial E_2}\dfrac{\partial E_2}{\partial x} + 
    \dfrac{\partial h_1}{\partial m_2}\dfrac{\partial m_2}{\partial x} + 
    \dfrac{\partial h_1}{\partial b}\dfrac{\partial b}{\partial x}.
  \end{equation*} 
  The derivatives $\frac{\partial \boldsymbol{v}}{\partial x}$ can be obtained directly and $\frac{\partial h_1}{\partial \boldsymbol{v}}$, $\frac{\partial h_2}{\partial \boldsymbol{v}}$ are calculated by implicit function theorem. Specifically, 
  \begin{equation*}
    \left[ \begin{array}{c}
      \frac{\partial h_1}{\partial E_1}\\
      \\
      \frac{\partial h_2}{\partial E_1} 
      \end{array}\right] = 
    -\left[\begin{array}{cc}
      \frac{\partial Q_1}{\partial h_1} & \frac{\partial Q_1}{\partial h_2}\\
      \\
      \frac{\partial Q_2}{\partial h_1} & \frac{\partial Q_2}{\partial h_2}
    \end{array} \right]^{-1}
    \left[ \begin{array}{c}
      \frac{\partial Q_1}{\partial E_1}\\
      \\
      \frac{\partial Q_2}{\partial E_1} 
      \end{array}\right], 
  \end{equation*}
  where $Q_1,Q_2$ are defined in (\ref{v_u:2LSWE}), and the others can be acquired similarly. 

\subsection{The fully-discrete PCDG scheme}
The full-discrete scheme is obtained by applying the third-order SSP-RK time discretization for \eqref{scheme:2LSWE_mov},
\begin{equation}\label{3rk_m}\footnotesize
  \begin{aligned}
  & \int_{I_j} \boldsymbol{u}(\boldsymbol{v}^{(1)}) \cdot \boldsymbol{\varphi} \ dx 
  = \int_{I_j} \boldsymbol{u}(\boldsymbol{v}^{n}) \cdot \boldsymbol{\varphi} \ dx 
  + \Delta t  \ \text{RHS}_j^M(\boldsymbol{v}^n,\boldsymbol{\varphi}), \\
  & \int_{I_j} \boldsymbol{u}(\boldsymbol{v}^{(2)}) \cdot \boldsymbol{\varphi} \ dx 
  = \dfrac{3}{4}\int_{I_j} \boldsymbol{u}(\boldsymbol{v}^{n}) \cdot \boldsymbol{\varphi} \ dx 
  + \dfrac{1}{4}\left(\int_{I_j} \boldsymbol{u}(\boldsymbol{v}^{(1)}) \cdot \boldsymbol{\varphi} \ dx 
  + \Delta t  \ \text{RHS}_j^M(\boldsymbol{v}^{(1)},\boldsymbol{\varphi}) \right),\\
  & \int_{I_j} \boldsymbol{u}(\boldsymbol{v}^{n+1}) \cdot \boldsymbol{\varphi} \ dx 
  = \dfrac{1}{3}\int_{I_j} \boldsymbol{u}(\boldsymbol{v}^{n}) \cdot \boldsymbol{\varphi} \ dx 
  + \dfrac{2}{3}\left(\int_{I_j} \boldsymbol{u}(\boldsymbol{v}^{(2)}) \cdot \boldsymbol{\varphi} \ dx 
  + \Delta t  \ \text{RHS}_j^M(\boldsymbol{v}^{(2)},\boldsymbol{\varphi}) \right),
  \end{aligned}
\end{equation}
where $\text{RHS}_j^M(\boldsymbol{v}^n,\boldsymbol{\varphi})$ is the spatial operator in (\ref{rhs_m}). 
Slight deviations from the fully discrete scheme for still water equilibria, (\ref{3rk_m}) entail nonlinear systems of equations at each inner stage of the Runge-Kutta time stepping,  given that the solutions are expressed in terms of the equilibrium variables $\boldsymbol{v}$ which are nonlinear with respect to $\boldsymbol{u}$. Thus, we adopt Newton's method for solving (\ref{3rk_m}) locally in each cell. 

\subsubsection{Newton's method}
The SSP-RK method can be seen as a convex combination of the forward Euler method. In order to illustrate the process of Newton iteration, we will use the forward Euler method as an example, as shown in the first equation in (\ref{3rk_m}),
\begin{equation}\label{1rk_ncp}
 \int_{I_j} \boldsymbol{u}(\boldsymbol{v}^{n+1}) \cdot \boldsymbol{\varphi} \ dx = \mathcal{R}_{j}(\boldsymbol{v}^{n},\boldsymbol{\varphi})
\end{equation}
with
\begin{equation}\label{res_ncp}
 \mathcal{R}_{j}(\boldsymbol{v}^{n},\boldsymbol{\varphi}) = \int_{I_j} \boldsymbol{u}(\boldsymbol{v}^{n}) \cdot \boldsymbol{\varphi} \ dx + \Delta t  \ \text{RHS}_j^M(\boldsymbol{v}^n,\boldsymbol{\varphi}).
\end{equation}
We denote $\mathcal{R}_{j}^{[l]} = \mathcal{R}_{j}(\boldsymbol{v}^{n},{\varphi}_l \boldsymbol {e}_l)$ as the right hand side of equations satisfied by the $l$-th equilibrium variable to be solved, here 
$\boldsymbol {e}_l:=(0,\dots,0,1,0,\dots,0)^T$
represents the unit vector with 1 in the $l$-th position and 0 in the rest. 
We first note that despite the bottom function $b$ being added in our moving equilibrium variables, a trivial calculation gives $ \mathcal{R}_{j}^{[5]}$ equals to zero so that no extra numerical diffusion present in our proposed PCDG scheme and the computed $b$ can remain time-independent. We still denote it as $b$ for ease of notation. Next, the specific mass flow rate $m_1,m_2$ defined in equilibrium variables can be acquired directly from the second and fourth components of (\ref{res_ncp}). The unknowns is the coefficients of the polynomial $E_1^{n+1}, E_2^{n+1} \in \mathcal{W}_h^k$, which are obtained by solving the following nonlinear systems iteratively 
\begin{equation}\label{newiter_ncp}\small
\begin{aligned}
& H_1(E_1^1,\dots,E_1^k,E_2^1,\dots,E_2^k)=\displaystyle{\int_{I_j} h_1\left(\sum _{i=1}^k E_1^i \varphi_1^i,\sum _{i=1}^k E_2^i \varphi_3^i,m_1^{n+1},m_2^{n+1},b \right) {\varphi}_1 \ dx} - \mathcal{R}_{j}^{[1]}=0,\\
& H_2(E_1^1,\dots,E_1^k,E_2^1,\dots,E_2^k)=\displaystyle{\int_{I_j} h_2\left(\sum _{i=1}^k E_1^i \varphi_1^i,\sum _{i=1}^k E_2^i \varphi_3^i,m_1^{n+1},m_2^{n+1},b \right) {\varphi}_3 \ dx} - \mathcal{R}_{j}^{[3]}=0.
\end{aligned}
\end{equation}
For more details and a good initial guess from which Newton iteration of the systems (\ref{newiter_ncp}) starts, we refer to \cite{zhang2023moving} for a thorough discussion.

\begin{remark}
  We return to the initial guesses for the Newton solver of (\ref{v_u:2LSWE}). When solving (\ref{1rk_ncp}), we can also acquire 
  $$\widetilde {\boldsymbol{u}}(x) = (\widetilde h_1(x),m_1(x),\widetilde h_2(x), m_2(x),b(x))  \in \boldsymbol{\mathcal{V}}_h^k$$ directly from the equation 
  $$\int_{I_j} \widetilde{\boldsymbol{u}}(x) \cdot \boldsymbol{\varphi} \ dx = \mathcal{R}_j(\boldsymbol{v}^{n},\boldsymbol{\varphi}).$$
  Then the obtained polynomials $(\widetilde h_1,\widetilde h_2)$ are chosen as the proper initial guesses of the Newton iteration for solving $(h_1,h_2)$. We emphasize that this option might not be optimal, but our numerical tests suggest that it is robust for small perturbations at or close to steady flows.
\end{remark}

\subsubsection{Slope limiter}
To mitigate spurious numerical oscillations caused by shock waves, a post-processing procedure is necessary. In this paper, we implement the TVB limiter after each stage of the SSP-RK method. To avoid disrupting the generalized moving equilibrium state, as demonstrated in \cite{zhang2023moving}, we modify the limiter procedure specifically on the local characteristic fields of the equilibrium functions 
\begin{equation}\label{v_e}
\boldsymbol{v}^e  = (E_1,m_1,E_2,m_2)^T,
\end{equation}
while keeping the bottom function $b$ unchanged. The corresponding eigenstructures are detailed in Appendix \ref{a2}.

The TVB limiter procedure consists of two main parts. Initially, we assess the necessity of the limiter in each cell based on the equilibrium functions $\boldsymbol{v}^e$ within the DG piecewise polynomial space at each time step. Subsequently, we apply the limiter exclusively to those troubled cells on the polynomials $\boldsymbol{v}^e({x})$. When the steady state $\boldsymbol{v}^e=\text{constant}$ is reached, it signifies that no limiting is required. This approach ensures that the limiting procedure does not compromise the well-balanced property of the PCDG scheme.

\subsection{Well-balanced property}
The moving water well-balanced properties for both the semi-discrete and fully-discrete PCDG methods are outlined in the following theorems.
\begin{theorem}\label{wbp_semi}
  The semi-discrete path-conservative discontinuous Galerkin method (\ref{scheme:2LSWE_mov}) - (\ref{rhs_m}) with the special modified numerical flux (\ref{LF_modify}) and (\ref{ncp_m}) is exact for one-dimensional moving equilibrium state (\ref{moving_water:2LSWE}).
\end{theorem}

\begin{proof}
  We suppose that the initial values are in an exact moving equilibrium state, that is $\widetilde{\boldsymbol{v}}(x) = \text{constant} := c$. It is straightforward that the approximation denoted by $\widetilde{\boldsymbol{v}}(x)$ again is constant and equals $c$ as well. The fact that $\widetilde{\boldsymbol{v}}_{j+\frac{1}{2}}^- = \widetilde{\boldsymbol{v}}_{j+\frac{1}{2}}^+$ coupled with linear segment path (\ref{linear_path_V}) yields the integral in (\ref{rhs_m})
  \begin{equation}\label{int}
    \int_0^1 \mathcal{L}(\boldsymbol{\phi}^M_{j\pm\frac{1}{2}}(\tau)) \frac{\partial }{\partial \tau}
    \widetilde{\boldsymbol{v}}(\boldsymbol{\phi}^M_{j\pm\frac{1}{2}}(\tau)) \ d\tau = 0.
  \end{equation}
  According to the definition (\ref{u_modify}), we immediately obtain $\boldsymbol{u}_{j+\frac{1}{2}}^{*,-} = \boldsymbol{u}_{j+\frac{1}{2}}^{*,+}$, and (\ref{LF_modify}) reduces to
  \begin{equation}\label{fwb_ncp}
    \widehat{\boldsymbol{f}}_{j+\frac{1}{2}}^{\text{mod}} = \dfrac{1}{2}\left(\boldsymbol{f}(\boldsymbol u_{j+\frac{1}{2}}^-) + \boldsymbol{f}( \boldsymbol u_{j+\frac{1}{2}}^+) \right).
  \end{equation}
  Hence substituting (\ref{int}) and (\ref{fwb_ncp}) back into (\ref{rhs_m}), we have
  \begin{equation*}
  \begin{aligned}
  \text{RHS}_{j}^M
  = &\int_{\mathcal I_j}\boldsymbol{f}(\boldsymbol{u}(\boldsymbol{v})) \cdot \boldsymbol{\varphi}_x \ dx 
    -\int_{\mathcal I_j}  \mathcal{G}(\boldsymbol {u}(\boldsymbol{v})) \boldsymbol {u}(\boldsymbol{v})_x \cdot \boldsymbol{\varphi}\ dx \\
    & -\dfrac{1}{2}\left(\boldsymbol{f}(\boldsymbol u_{j+\frac{1}{2}}^-) + \boldsymbol{f}( \boldsymbol u_{j+\frac{1}{2}}^+) \right)  \cdot \boldsymbol{\varphi}_{j+\frac{1}{2}}^- + 
     \dfrac{1}{2}\left(\boldsymbol{f}(\boldsymbol u_{j-\frac{1}{2}}^-) + \boldsymbol{f}( \boldsymbol u_{j-\frac{1}{2}}^+) \right)  \cdot \boldsymbol{\varphi}_{j-\frac{1}{2}}^+ \\
    & - \dfrac{1}{2} \boldsymbol{\varphi}_{j+\frac{1}{2}}^-  \cdot \left( - \boldsymbol{f}(\boldsymbol{u}_{j+\frac{1}{2}}^+) + \boldsymbol{f} (\boldsymbol{u}_{j+\frac{1}{2}}^-)\right)  - 
     \dfrac{1}{2} \boldsymbol{\varphi}_{j-\frac{1}{2}}^+  \cdot \left( - \boldsymbol{f}(\boldsymbol{u}_{j-\frac{1}{2}}^+) + \boldsymbol{f}(\boldsymbol{u}_{j-\frac{1}{2}}^-)\right)\\
  = & \int_{I_j}\boldsymbol{f}(\boldsymbol{u}(\boldsymbol{v})) \cdot \boldsymbol{\varphi}_x \ dx -
      \boldsymbol{f}(\boldsymbol{u}_{j+\frac{1}{2}}^-) \cdot \boldsymbol{\varphi}_{j+\frac{1}{2}}^- + 
      \boldsymbol{f}(\boldsymbol{u}_{j-\frac{1}{2}}^+) \cdot  \boldsymbol{\varphi}_{j-\frac{1}{2}}^+
      -\int_{\mathcal I_j}  \mathcal{G}(\boldsymbol {u}(\boldsymbol{v}))\boldsymbol {u}(\boldsymbol{v})_x \cdot \boldsymbol{\varphi}\ dx  \\
  = & -\int_{I_j} \partial_x\boldsymbol{f}(\boldsymbol{u}(\boldsymbol{v})) \cdot \boldsymbol{\varphi} \ dx -
      \int_{\mathcal I_j}  \mathcal{G}(\boldsymbol {u}(\boldsymbol{v}))\boldsymbol {u}(\boldsymbol{v})_x \cdot \boldsymbol{\varphi}\ dx  \\
  = & -\int_{I_j} \left(\partial_x\boldsymbol{f}(\boldsymbol{u}(\boldsymbol{v})) +  \mathcal{G}(\boldsymbol {u}(\boldsymbol{v}))\boldsymbol {u}(\boldsymbol{v})_x \right) \cdot \boldsymbol{\varphi} \ dx=0.
  \end{aligned}
  \end{equation*}
  The last equality follows from the equilibrium state solution.  This completes the proof of well-balanced property.
  \end{proof}

For the fully-discrete scheme, the well-balanced property is held through a straightforward induction.
\begin{theorem}\label{wbp_fully}
The fully-discrete path-conservative discontinuous Galerkin method, space discretization (\ref{scheme:2LSWE_mov}) - (\ref{linear_path_V}) together with the third-order SSP-RK time discretization (\ref{3rk_m}), is exact for one-dimensional moving equilibrium state (\ref{moving_water:2LSWE}).
\end{theorem}


The idea of the change of variables was also adopted by Hughes et.al. \cite{hughes1986new}, where the entropy variables were utilized for deriving symmetric forms of the Euler and Navier-Stokes equations. The resulting finite element methods satisfied the second law of thermodynamics and a discrete entropy inequality thanks to the definition of the entropy variables. 
Additionally, Mantri et.al. \cite{mantri2021well} developed a well-balanced DG scheme for 2$\times$2 hyperbolic balance law by solving a linear system of equations for the equilibrium variables given by the global fluxes. However, the nonconservative form of the equations for $\boldsymbol{v}_t$ obtained by chain rule yields the loss of mass conservation. 

In the moving equilibrium-preserving path-conservative discontinuous Galerkin method for the two-layer system, the well-balanced property is established through a reformulation of the DG piecewise polynomial space. This reformulation is carried out in terms of equilibrium variables rather than conservative variables, with a crucial emphasis on maintaining the original formulation of the target system. Here, conservative variables are treated exclusively as non-polynomial functions of the equilibrium variables. This strategic approach facilitates the element-by-element evaluation of nonlinear iterations, ensuring mass conservation throughout our calculations. 

 It's important to acknowledge that, due to a lack of rigorous analysis regarding the convexity and monotonicity of the two coupled cubic equations (\ref{v_u:2LSWE}), the Newton solver may face challenges in converging to a correct solution, especially in the presence of strong shock discontinuities and flows far from equilibrium. Our primary focus is directed towards evaluating the performance of our developed approach in scenarios involving nearly steady-state flow with relatively coarse meshes, as elaborated in Section \ref{se:nu}. Notably, the tolerance in the Newton solver for (\ref{v_u:2LSWE}) and (\ref{newiter_ncp}) is set at $\text{TOL} = 10^{-13}$.

\section{Numerical examples}\label{se:nu}
This section presents the implementation of our proposed path-conservative discontinuous Galerkin methods for two-layer shallow water equations, preserving still water and moving water equilibria. The tested schemes will be referred to as PCDG-still and PCDG-moving, respectively. We also compare their performance with that of the corresponding second-order central-upwind (CU) scheme developed by Kurganov and Petrova \cite{kurganov2009central} and path-conservative central-upwind (PCCU) scheme developed by Castro D{\'\i}az et al. \cite{diaz2019path} on several numerical examples.

Unless otherwise specified, the piecewise quadratic ($P^2$) polynomial basis is adopted in the spatial discretization. In all numerical experiments, we set the CFL number as 0.18 for one-dimensional cases and 0.1 for two-dimensional cases. 
The time step constraint is acquired by
$$\Delta t = \text{CFL}\frac{ h_m }{\max_{l,\boldsymbol u}{|\lambda_l(\boldsymbol u)|}},$$ 
where $h_m$ denotes the mesh size. 
The constant in the TVB slope limiter is set as 0, except for the accuracy tests in which no slope limiter is applied. Without further specification, free boundary conditions are imposed in the following examples.

\subsection{One-dimensional tests}
\begin{example}{\bf Accuracy test}\label{accuracy1D}
\end{example}
To verify the high-order accuracy of our developed well-balanced PCDG schemes for a smooth solution, we choose the same bottom topography and initial conditions as in \cite{kurganov2009central}
\begin{equation}\label{inismoo}
  \begin{aligned}
  &b(x) = \sin^2(\pi x)-10,\quad h_1(x,0)=5+e^{\cos(2\pi x)}, \\
  &w(x,0) = -5-e^{\cos(2\pi x)}, \quad m_1(x,0)= m_2(x,0) = 0,
  \end{aligned}
\end{equation}
and 1-periodic boundary conditions. The gravitation constant $g$ is taken as $9.81$ and the density ratio is $r=0.98$. We run the simulations up to $t=0.1$, by which the exact solution remains smooth. We treat the solution obtained by our proposed third-order PCDG-still method with 12800 cells over the periodic domain as the reference solution. Both the piecewise linear and quadratic polynomial basis are used for our test. The $L^1$ errors of the upper and lower water height $h_1$, $h_2$ coupled with the discharge $m_1$, $m_2$ and numerical orders of accuracy for both the PCDG-still and PCDG-moving methods are listed in Table \ref{DGacc1D_still} and Table \ref{DGacc1D_mov} respectively, where we can observe the optimal convergence rates.

\begin{table}[!htb]
  \centering
  \caption{Example \ref{accuracy1D}:  $L^1$ errors and numerical orders of accuracy with initial conditions (\ref{inismoo}), using PCDG-still method.}\label{DGacc1D_still}
  \begin{tabular}{c c c c c c c c c c }
  \toprule
   \multirow{2}{*}  & \multirow{2}{*} {$nx$} & \multicolumn{2}{c}{$h_1$}&\multicolumn{2}{c}{$ m_1 $}&\multicolumn{2}{c}{$ h_2 $}& \multicolumn{2}{c}{$m_2$}\\
   \cmidrule(lr){3-4} \cmidrule(lr){5-6} \cmidrule(lr){7-8} \cmidrule(lr){9-10}
   ~ & ~ &$L^1$ error &order&$L^1$ error &order&$L^1$ error &order&$L^1$ error &order\\
   \midrule
   \multirow{7}{*}{$P^1$}&25   & 1.62E-03 &   -- & 6.97E-04 &   -- & 1.15E-03 &   -- & 6.42E-04 &   -- \\
                        ~&50   & 3.94E-04 & 2.04 & 1.66E-04 & 2.07 & 2.74E-04 & 2.07 & 1.58E-04 & 2.02 \\
                        ~&100  & 9.74E-05 & 2.02 & 4.09E-05 & 2.02 & 6.72E-05 & 2.03 & 3.98E-05 & 1.99 \\
                        ~&200  & 2.42E-05 & 2.01 & 1.02E-05 & 2.00 & 1.67E-05 & 2.01 & 1.00E-05 & 1.99 \\
                        ~&400  & 6.05E-06 & 2.00 & 2.55E-06 & 2.00 & 4.15E-06 & 2.01 & 2.51E-06 & 2.00 \\
                        ~&800  & 1.51E-06 & 2.00 & 6.37E-07 & 2.00 & 1.04E-06 & 2.00 & 6.28E-07 & 2.00 \\
                        ~&1600 & 3.77E-07 & 2.00 & 1.59E-07 & 2.00 & 2.59E-07 & 2.00 & 1.57E-07 & 2.00 \\
   \midrule
   \multirow{7}{*}{$P^2$}&25   & 2.34E-04 &   -- & 1.72E-04 &   -- & 2.31E-04 &  --  & 1.72E-04 &  -- \\
                        ~&50   & 3.07E-05 & 2.93 & 3.40E-05 & 2.34 & 3.04E-05 & 2.92 & 3.41E-05 & 2.34 \\
                        ~&100  & 2.20E-06 & 3.81 & 7.10E-06 & 2.26 & 2.19E-06 & 3.79 & 7.10E-06 & 2.26 \\
                        ~&200  & 5.08E-07 & 2.11 & 1.13E-06 & 2.65 & 5.07E-07 & 2.11 & 1.13E-06 & 2.65 \\
                        ~&400  & 1.28E-07 & 1.99 & 1.45E-07 & 2.97 & 1.27E-07 & 1.99 & 1.45E-07 & 2.97 \\
                        ~&800  & 2.07E-08 & 2.63 & 1.74E-08 & 3.06 & 2.06E-08 & 2.63 & 1.74E-08 & 3.06 \\
                        ~&1600 & 2.87E-09 & 2.85 & 2.09E-09 & 3.05 & 2.86E-09 & 2.85 & 2.09E-09 & 3.05 \\
    \bottomrule
  \end{tabular}
\end{table}

\begin{table}[!htb]
  \centering
  \caption{Example \ref{accuracy1D}:  $L^1$ errors and numerical orders of accuracy with initial condition (\ref{inismoo}), using PCDG-moving method.}\label{DGacc1D_mov}
  \begin{tabular}{c c c c c c c c c c }
  \toprule
   \multirow{2}{*} & \multirow{2}{*} {$nx$} & \multicolumn{2}{c}{$h_1$}&\multicolumn{2}{c}{$ m_1 $}&\multicolumn{2}{c}{$ h_2 $}& \multicolumn{2}{c}{$m_2$}\\
   \cmidrule(lr){3-4} \cmidrule(lr){5-6} \cmidrule(lr){7-8} \cmidrule(lr){9-10}
   ~ & ~ &$L^1$ error &order&$L^1$ error &order&$L^1$ error &order&$L^1$ error &order\\
   \midrule
   \multirow{7}{*}{$P^1$}&25 & 1.64E{-03} &     -- &  6.98E{-04} &  --  & 1.18E{-03} &     -- & 6.42E{-04} &   --\\
                        ~&50 & 4.00E{-04} &   2.04 &  1.66E{-04} & 2.08 & 2.80E{-04} &   2.07 & 1.58E{-04} &   2.02\\
                        ~&100& 9.89E{-05} &   2.02 &  4.09E{-05} & 2.02 & 6.97E{-05} &   2.03 & 3.98E{-05} &   1.99\\
                        ~&200& 2.46E{-05} &   2.01 &  1.02E{-05} & 2.00 & 1.70E{-05} &   2.01 & 1.00E{-05} &   1.99\\
                        ~&400& 6.15E{-06} &   2.00 &  2.55E{-06} & 2.00 & 4.25E{-06} &   2.00 & 2.51E{-06} &   2.00\\
                        ~&800& 1.54E{-06} &   2.00 &  6.42E{-07} & 2.00 & 1.06E{-06} &   2.00 & 6.29E{-07} &   2.00\\
                       ~&1600& 3.84E{-07} &   2.00 &  1.59E{-07} & 2.00 & 2.65E{-07} &   2.00 & 1.57E{-07} &   2.00\\
   \midrule
   \multirow{7}{*}{$P^2$}&25 & 2.38E{-04} &     -- &  1.78E{-04} &  --  &  2.34E{-04} &  --  & 1.78E{-04} &   --\\
                        ~&50 & 3.11E{-05} &   2.94 &  3.50E{-05} & 2.34 &  3.07E{-05} & 2.93 & 3.52E{-05} &   2.34\\
                        ~&100& 2.31E{-06} &   3.75 &  7.28E{-06} & 2.27 &  2.30E{-06} & 3.74 & 7.28E{-06} &   2.27\\
                        ~&200& 4.73E{-07} &   2.29 &  1.17E{-06} & 2.64 &  4.71E{-07} & 2.29 & 1.17E{-06} &   2.64\\
                        ~&400& 1.24E{-07} &   1.94 &  1.52E{-07} & 2.94 &  1.23E{-07} & 1.94 & 1.52E{-07} &   2.94\\
                        ~&800& 2.01E{-08} &   2.62 &  1.85E{-08} & 3.03 &  2.00E{-08} & 2.62 & 1.85E{-08} &   3.03\\
                       ~&1600& 3.17E{-09} &   2.66 &  2.32E{-09} & 2.99 &  3.15E{-09} & 2.67 & 2.28E{-09} &   3.03\\
  \bottomrule
  \end{tabular}
\end{table}

\begin{example}{\bf Test for the still water well-balanced property}\label{exact1D:still}
\end{example}
In this example, which is taken from \cite{kurganov2009central}, we testify the still water well-balanced property of our proposed methods. To this end, both the continuous bottom topography
\begin{equation}\label{smo1D}
  b(x)=\left\{\begin{array}{lll}
    0.25(\cos(10\pi(x-0.5))+1)-2,&\text{if}\ 0.4<x < 0.6, \\
    -2&\text{otherwise},\\
    \end{array}\right.
\end{equation}
and the discontinuous bottom topography
\begin{equation}\label{dis1D}
  b(x)=\left\{\begin{array}{lll}
    -1.5,&\text{if}\ x>0.5, \\
    -2,&\text{otherwise},\\
    \end{array}\right.
\end{equation}
are considered. And the initial conditions are given by
\begin{equation}\label{stillwb_u1D}
  \begin{aligned}
    (h_1,m_1,w, m_2)(x,0) = (h_1,m_1,w, m_2)_{\text{eq}}(x) = (1,0,-1,0),
  \end{aligned}
\end{equation}
in the computational domain $[-0.2,1]$. We set the gravitation constant $g = 10$ and the density ratio $r=0.98$. We compute the numerical solutions up to $t = 0.1$ with 100 uniform cells, and calculate the $L^1$ and $L^{\infty}$ errors at double precision for the conservative and equilibrium variables with piecewise linear and quadratic polynomial basis. Tables \ref{2LSWE:wbstill_con1D_s}-\ref{2LSWE:wbstill_con1D_m} show the obtained results by two developed schemes over the continuous bottom topography, and the corresponding results over the discontinuous bottom function are displayed in Tables \ref{2LSWE:wbstill_dis1D_s}-\ref{2LSWE:wbstill_dis1D_m}. It can be seen that no matter what the bottom topography is, both the PCDG-still and PCDG-moving schemes can maintain the ``lake-at-rest'' steady state exactly, which is consistent with the well-balanced property we anticipated.
\begin{table}[!htb]
  \centering
  \caption{Example \ref{exact1D:still}: $L^1$ and $L^{\infty}$ errors for one-dimensional still water equilibrium state with the continuous bottom topography (\ref{smo1D}), using PCDG-still method.}\label{2LSWE:wbstill_con1D_s}
  \begin{tabular}{ c c c c c c c c}
   \toprule
      &{error}&{$h_1$}&{$m_1$}&{$h_2$}&{$m_2$}&{$w$}\\
    \midrule
    \multirow{2}{*}{$P^1$ }&$L^1$        & 2.55E-15 & 8.14E-16 & 2.54E-15 & 8.93E-16 & 2.54E-15 \\
                          ~&$L^{\infty}$ & 1.42E-14 & 2.72E-15 & 1.44E-14 & 3.32E-15 & 1.44E-14 \\
    \midrule
    \multirow{2}{*}{$P^2$ }&$L^1$        & 1.04E-14 & 2.02E-15 & 1.04E-14 & 1.98E-15 & 1.04E-14 \\
                         ~ &$L^{\infty}$ & 7.82E-14 & 6.43E-15 & 7.73E-14 & 7.22E-15 & 7.73E-14 \\
    \bottomrule
  \end{tabular}
\end{table}

\begin{table}[!htb]
  \centering
  \caption{Example \ref{exact1D:still}: $L^1$ and $L^{\infty}$ errors for one-dimensional still water equilibrium state with the continuous bottom topography (\ref{smo1D}), using PCDG-moving method.}\label{2LSWE:wbstill_con1D_m}
  \begin{tabular}{ c c c c c c c c}
   \toprule
      &{error}&{$h_1$}&{$m_1$}&{$h_2$}&{$m_2$}&{$E_1$}&{$E_2$}\\
    \midrule
    \multirow{2}{*}{$P^1$ }&$L^1$        & 7.83E{-14} & 1.97E{-14} &  7.65E{-14} &  1.89E{-14} &  1.40E{-12} &  1.39E{-12}\\
                          ~&$L^{\infty}$ & 2.27E{-13} & 7.82E{-14} &  2.00E{-13} &  7.88E{-14} &  1.23E{-12} &  1.20E{-12}\\
    \midrule
    \multirow{2}{*}{$P^2$ }&$L^1$        & 1.09E-13 & 3.54E-14 & 1.10E-13 & 3.50E-14 & 7.97E-13 & 7.89E-13\\
                         ~ &$L^{\infty}$ & 3.27E-13 & 1.25E-13 & 3.13E-13 & 1.11E-13 & 7.37E-13 & 7.27E-13\\
    \bottomrule
  \end{tabular}
\end{table}

\begin{table}[!htb]
  \centering
  \caption{Example \ref{exact1D:still}: $L^1$ and $L^{\infty}$ errors for one-dimensional still water equilibrium state with the discontinuous bottom topography (\ref{dis1D}), using PCDG-still method.}\label{2LSWE:wbstill_dis1D_s}
  \begin{tabular}{ c c c c c c c c}
   \toprule
      &{error}&{$h_1$}&{$m_1$}&{$h_2$}&{$m_2$}&{$w$}\\
    \midrule
    \multirow{2}{*}{$P^1$ }&$L^1$        & 2.61E-15 & 9.64E-16 & 2.54E-15 & 1.02E-15 & 2.54E-15 \\
                          ~&$L^{\infty}$ & 1.42E-14 & 3.74E-15 & 1.47E-14 & 3.33E-15 & 1.47E-14 \\
    \midrule
    \multirow{2}{*}{$P^2$ }&$L^1$        & 9.79E-15 & 2.31E-15 & 9.78E-15 & 2.02E-15 & 9.78E-15 \\
                         ~ &$L^{\infty}$ & 7.59E-14 & 1.04E-14 & 7.66E-14 & 8.35E-15 & 7.66E-14 \\
    \bottomrule
  \end{tabular}
\end{table}

\begin{table}[!htb]
  \centering
  \caption{Example \ref{exact1D:still}: $L^1$ and $L^{\infty}$ errors for one-dimensional still water equilibrium state with the discontinuous bottom topography (\ref{dis1D}), using PCDG-moving method.}\label{2LSWE:wbstill_dis1D_m}
  \begin{tabular}{ c c c c c c c c}
   \toprule
      &{error}&{$h_1$}&{$m_1$}&{$h_2$}&{$m_2$}&{$E_1$}&{$E_2$}\\
    \midrule
    \multirow{2}{*}{$P^1$ }&$L^1$        & 9.37E-14 & 2.86E-14 & 7.95E-14 & 3.15E-14 & 1.30E-12 & 1.28E-12\\
                          ~&$L^{\infty}$ & 7.75E-13 & 1.60E-13 & 6.63E-13 & 1.53E-13 & 1.16E-12 & 1.16E-12\\
    \midrule
    \multirow{2}{*}{$P^2$ }&$L^1$        & 1.06E-13 & 3.56E-14 & 1.04E-13 & 3.39E-14 & 7.69E-13 & 7.60E-13\\
                         ~ &$L^{\infty}$ & 5.84E-13 & 1.84E-13 & 6.43E-13 & 1.56E-13 & 7.70E-13 & 7.35E-13\\
    \bottomrule
  \end{tabular}
\end{table}


\begin{example}{\bf Small perturbation of still water equilibrium}\label{2lswe:perstill}
\end{example}
We now demonstrate the ability of our proposed schemes to capture the propagation of small perturbations of the ``lake-at-rest'' steady state. Almost the same setup in Example \ref{exact1D:still}, we add a small perturbation to the upper layer depth $(h_1)_{\text{eq}}$
\begin{equation}\label{stillper_u1D}
  \begin{aligned}
  &h_1(x,0)=(h_1)_{\text{eq}}(x) +
      \left\{\begin{array}{lll}
      0.00001,&  \text{if} \  x \in [0.1,0.2], \\
      0,      &  \text{otherwise},\\
      \end{array}\right.  \\
  &w(x,0)=w_{\text{eq}}(x),\quad m_1(x,0) = (m_1)_{\text{eq}}(x),\quad m_2(x,0) = (m_2)_{\text{eq}}(x).
  \end{aligned}
\end{equation}
The stopping time is $t=0.15$. We calculate the results on a coarse mesh with 200 cells and a finer mesh of 3000 cells for comparison.
The fluctuation of the water surface $h_1+w$ over two different bottom topographies is displayed in Fig. \ref{st_per}. As expected, our studied well-balanced schemes are successful in capturing the small perturbation of the hydrostatic equilibrium state, and no oscillations near the discontinuity have been generated by either of the developed methods.

\begin{figure}[htb!]
  \centering
    \includegraphics[width=7cm,scale=1]{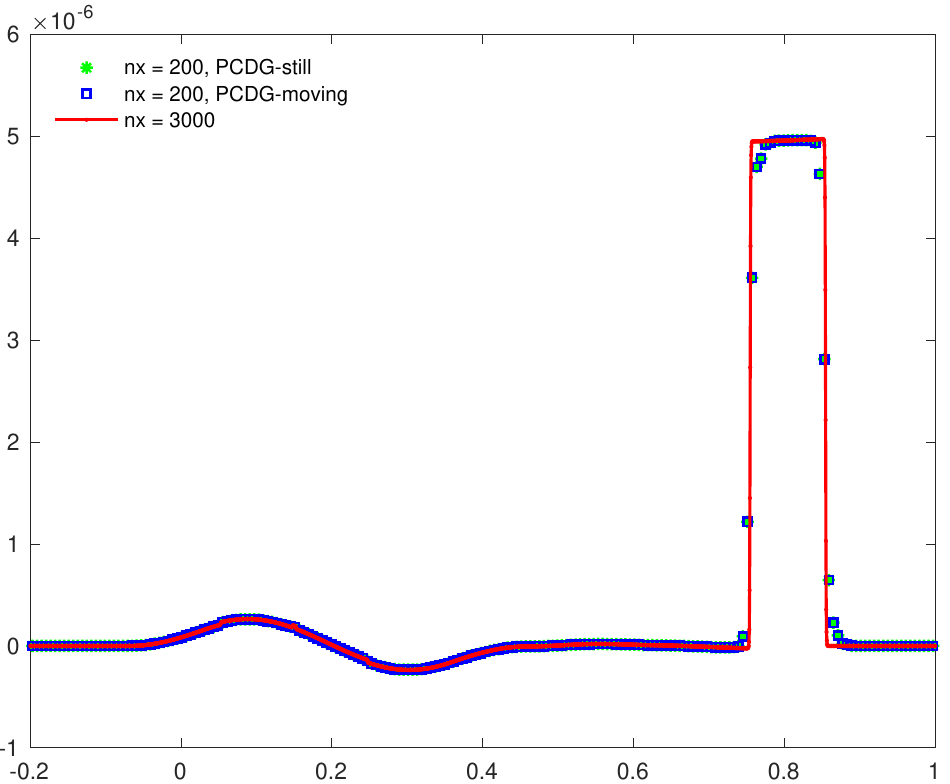}
    \includegraphics[width=7cm,scale=1]{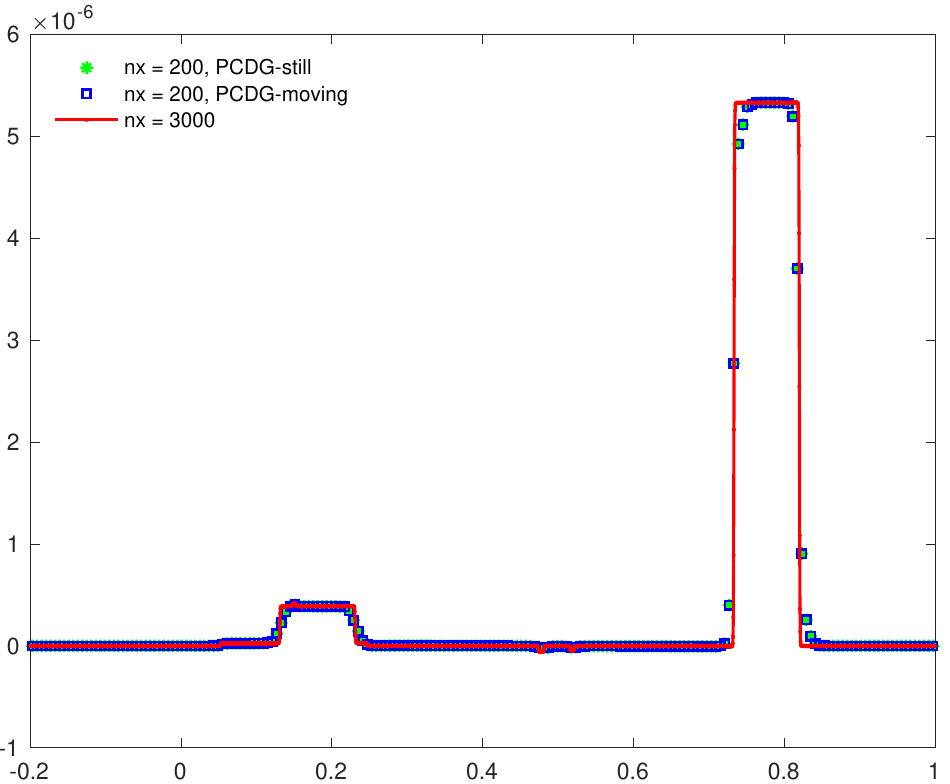}
  \caption{Example \ref{2lswe:perstill}: Numerical solutions of the water surface $h_1+w$ with the continuous bottom topography  $(\ref{smo1D})$ (left) and the discontinuous bottom topography  $(\ref{dis1D})$ (right), using 200 and 3000 grid cells. }\label{st_per}
\end{figure}

\begin{example}{\bf Interface propagation}\label{pro1d}
\end{example}
We use this example to study the propagation of the interface. Here we restrict our consideration to the case of a flat bottom topography $b(x)=-1$ since the nonconservative interlayer exchange terms on the right-hand side of (\ref{2LSWE}) is more challenging in computational than the corresponding geometric source terms. Two tests taken from \cite{kurganov2009central} are chosen. We take sufficiently large computational domains to make sure that no waves reach the boundaries until the final time. The first initial conditions on $[-1,1]$ are given by
\begin{equation}\label{pro}
(h_1,m_1,w, m_2)(x,0) = \left\{\begin{array}{lll}
    (0.50,1.250,-0.50,1.250),&  \text{if} \  x < 0.3, \\
    (0.45,1.125,-0.45,1.375),&\text{otherwise},\\
    \end{array}\right.
\end{equation}
which the interface initially located at $x=0.3$. The gravitational constant is $g = 10$ and the density ratio is $r=0.98$.
We simulate the solutions until $t=0.1$ with 400 uniform cells, and present the upper layer water height $h_1$, water surface $h_1+w$ and velocity of the upper layer $u_1$ in Fig. \ref{propa2}. It can be found that an intermediate flat state ($h_1\approx 0.475$) which is connected to the left ($h_1=0.5$) and right ($h_1=0.45$) states through shock discontinuities has emerged. The velocity $u_1$ develops a flat plateau in the region of that intermediate state. Both the PCDG-still and PCDG-moving methods realize non-oscillatory and high-resolution interface tracking of fluid motion.

\begin{figure}[htb!]
  \centering
  \subfigure[upper layer $h_1$]{
  \centering
  \includegraphics[width=4.9cm,scale=1]{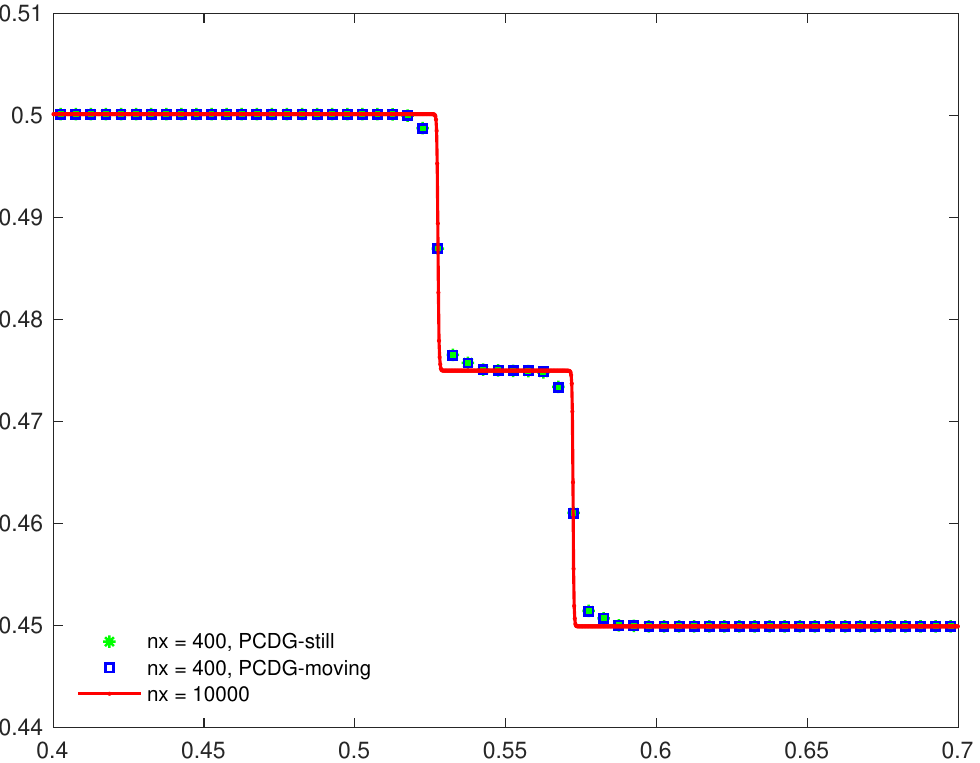}
  }
  \subfigure[water surface $h_1+w$]{
  \centering
  \includegraphics[width=4.9cm,scale=1]{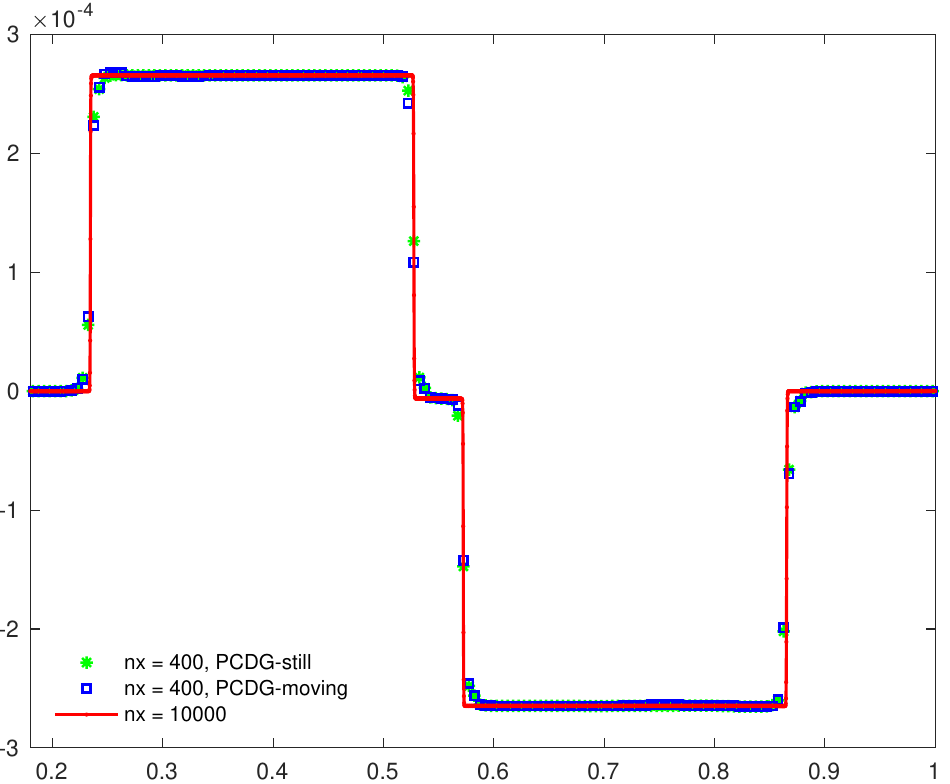}
  }
  \subfigure[ velocity $u_1$]{
  \centering
  \includegraphics[width=4.9cm,scale=1]{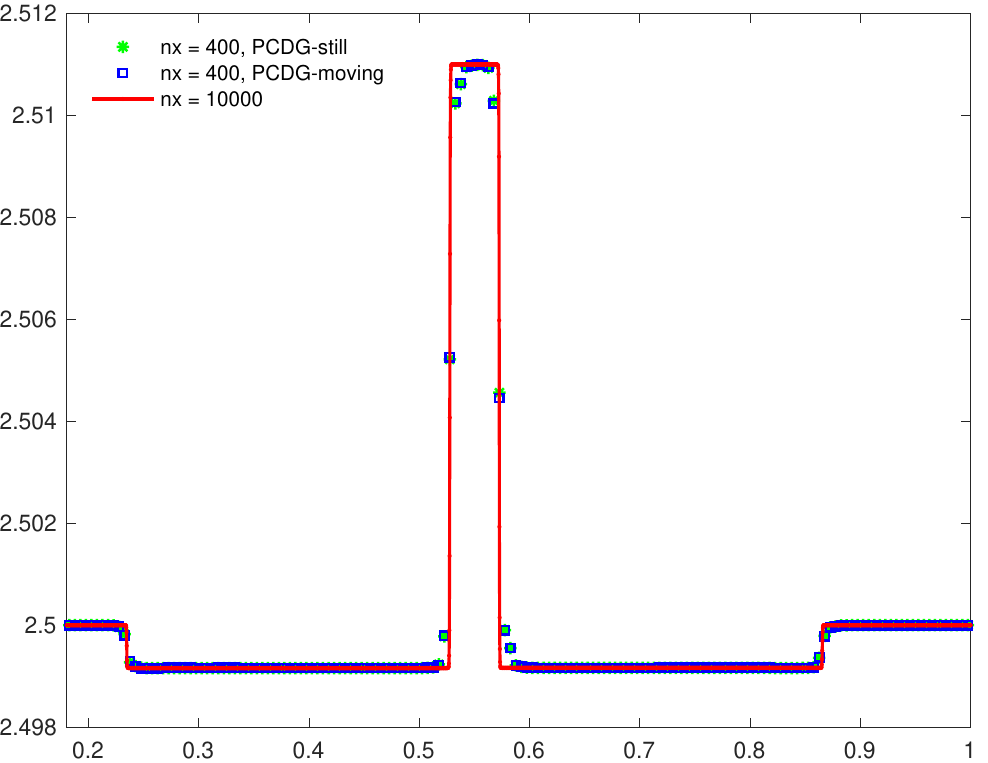}
  }
  \caption{Example \ref{pro1d}: From left to right: Numerical solutions of upper layer $h_1$ zoomed at the interface area, water surface $h_1+w$ and velocity of the upper layer $u_1$ with initial conditions (\ref{pro}). }\label{propa2}
\end{figure} 

Next, a much larger initial jump at the interface is considered
\begin{equation}\label{pro_large}
(h_1,m_1,w, m_2)(x,0) = \left\{\begin{array}{lll}
    (1.8,0,-1.8,0),&  \text{if} \  x < 0, \\
    (0.2,0,-0.2,0),&\text{otherwise},\\
    \end{array}\right.
\end{equation}
with the flat bottom topography $b(x)=-2$ on the computational domain $[-5,5]$. The gravitational constant is $g = 9.81$ and the density ratio is $r=0.98$. The stopping time is set as $t=1$. Fig. \ref{large2} plots the interface $w$ and water surface $h_1+w$ calculated by our two proposed methods with 200 and 5000 uniform grid cells. Moreover, we simulate the same tests with a refined 5000 uniform cells by second-order CU and PCCU schemes, third-order PCDG-still and PCDG-moving schemes, and illustrate the corresponding results in  Fig. \ref{large2} for comparison as well. This provides additional evidence of the robustness of our proposed path-conservative approach.

\begin{figure}[htb!]
 \centering
 \subfigure[interface $w$, $nx = 200$]{
  \centering
  \includegraphics[width=6.4cm,scale=1]{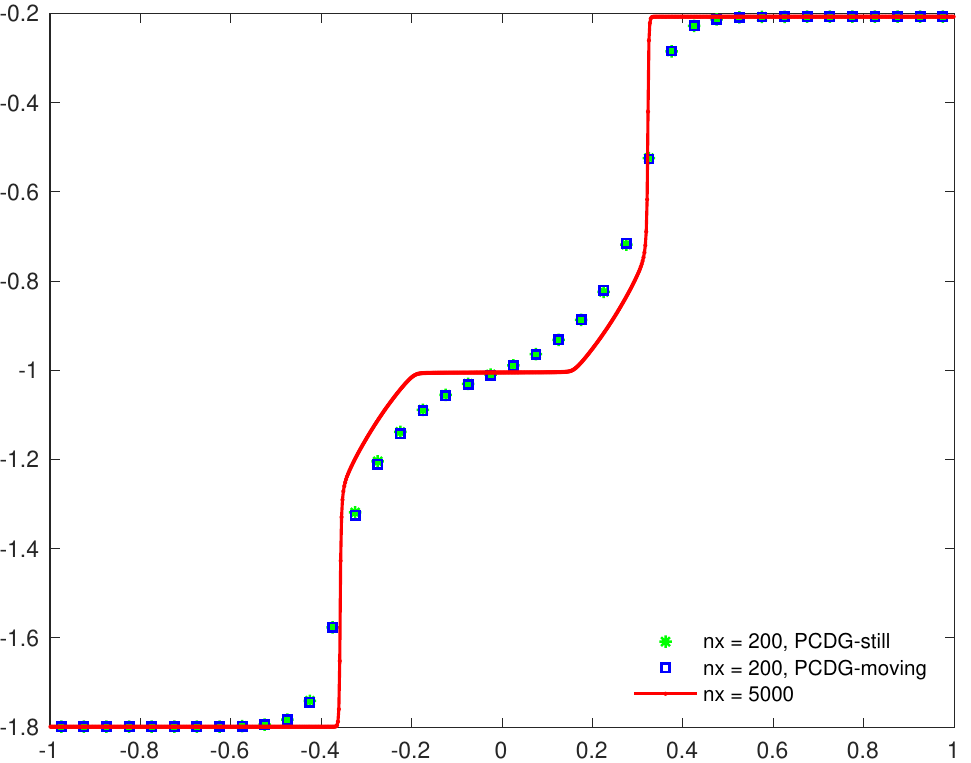}
  }
  \subfigure[water surface $h_1+w$, $nx = 200$]{
  \centering
  \includegraphics[width=6.5cm,scale=1]{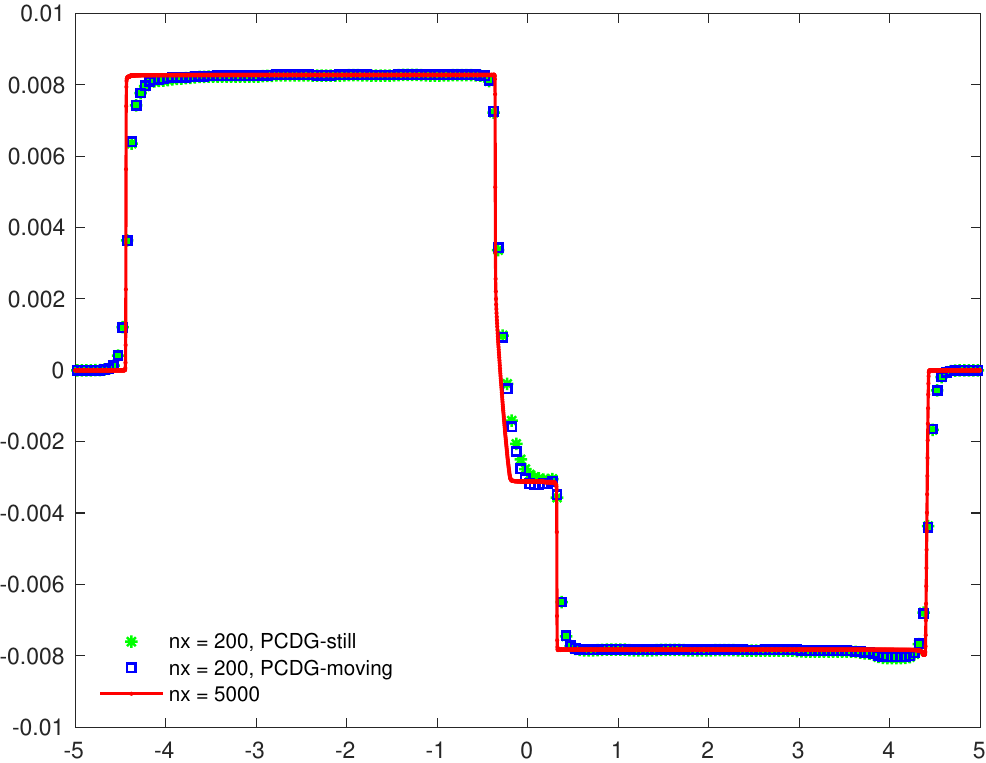}
  }
  \subfigure[interface $w$, comparison]{
  \centering
  \includegraphics[width=6.5cm,scale=1]{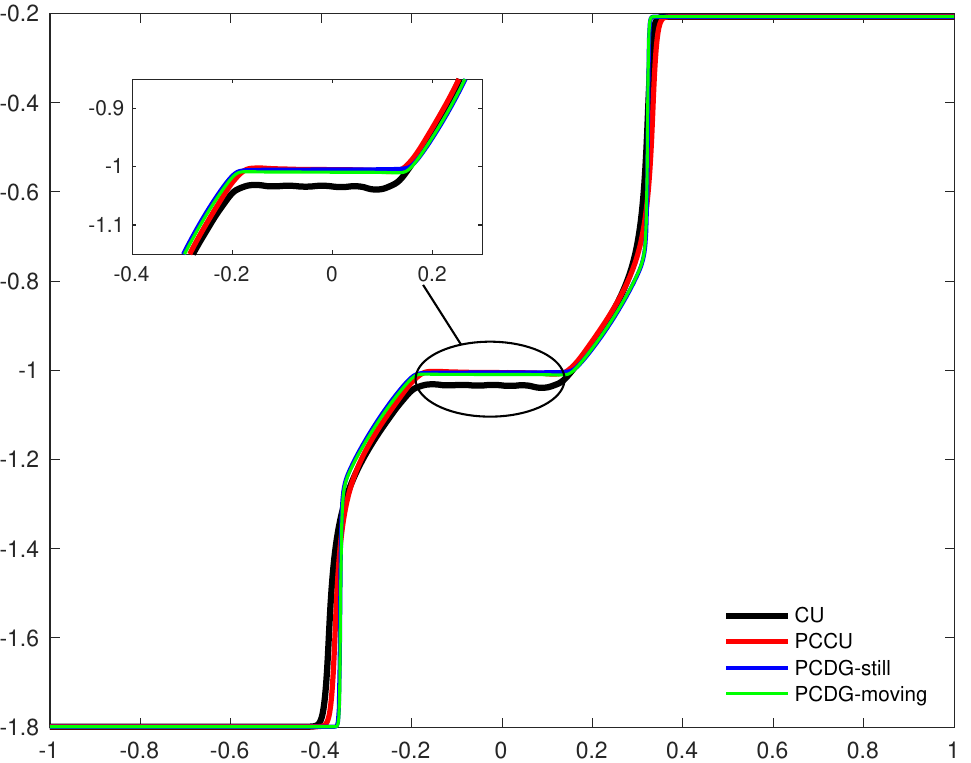}
  }
  \subfigure[water surface $h_1+w$, comparison]{
  \centering
  \includegraphics[width=6.4cm,scale=1]{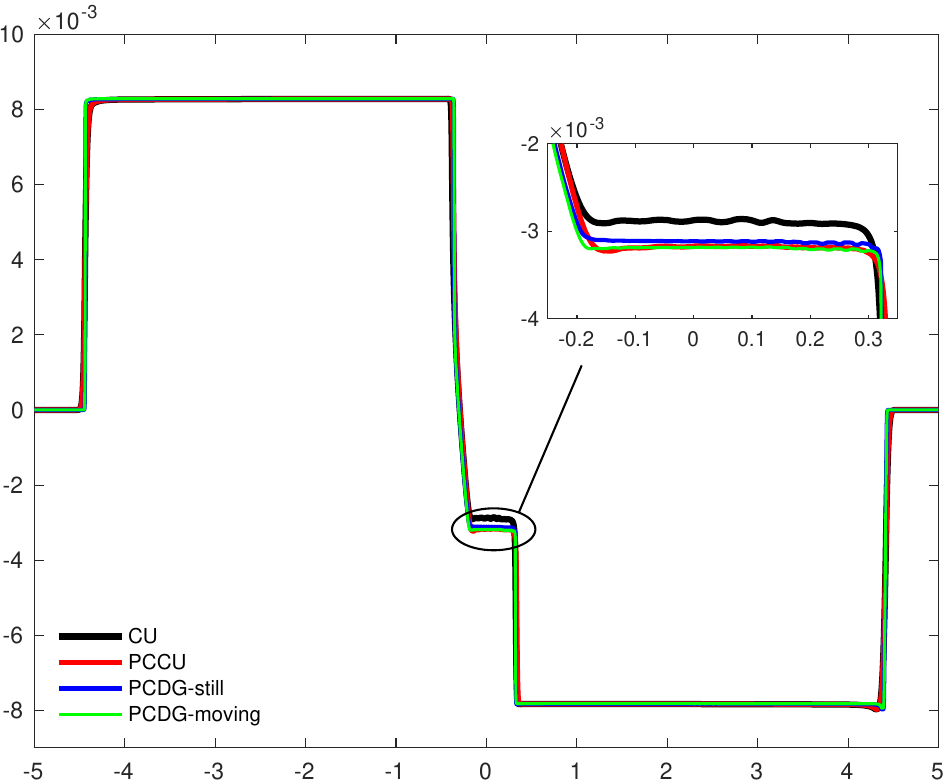}
  }
 \caption{Example \ref{pro1d}: Numerical solutions of interface $w$ (left) and water surface $h_1+w$ (right) with initial conditions (\ref{pro_large}). Top: using 200 and 5000 grid cells; bottom: using CU, PCCU, PCDG-still, and PCDG-moving schemes for comparison.}\label{large2}
\end{figure}

\begin{example}{\bf Internal dam break}\label{dam1d}
\end{example}
To testify the capability of our constructed still water equilibrium preserving PCDG method for the discontinuous steady state solutions,
We illustrate an internal dam break problem, which is taken from \cite{chu2022fifth} over a nonflat bottom topography
\begin{equation*}
  b(x)=0.25e^{-x^2}-2,
\end{equation*}
together with the initial conditions
\begin{equation}\label{dam}
(h_1,m_1,w, m_2)(x,0) = \left\{\begin{array}{lll}
    (1.6,0,-1.6,0),&  \text{if} \  x < 0, \\
    (0.7,0,-0.7,0),&\text{otherwise}.\\
    \end{array}\right.
\end{equation}
The computational domain is chosen as $[-5,5]$, which is large enough to eliminate any significant influence of boundary conditions on the computed solutions. The gravitational constant is $g = 9.81$ and the density ratio is $r=0.998$. In this initial boundary value problem, the solution is expected to converge to a discontinuous non-stationary steady state that contains a hydraulic jump. We simulate the solutions up to $t=300$ with 200 uniform grid cells and present the obtained results in Fig. \ref{dam_steady}. We also show the convergence solutions calculated by the second-order CU, PCCU scheme, and our proposed third-order PCDG-still scheme for comparison. We can see that the PCDG-still method achieves a high overall resolution of the discontinuous interface.

\begin{figure}[htb!]
  \centering 
  \subfigure[PCDG, $nx=200$]{
  \centering
  \includegraphics[width=7cm,scale=1]{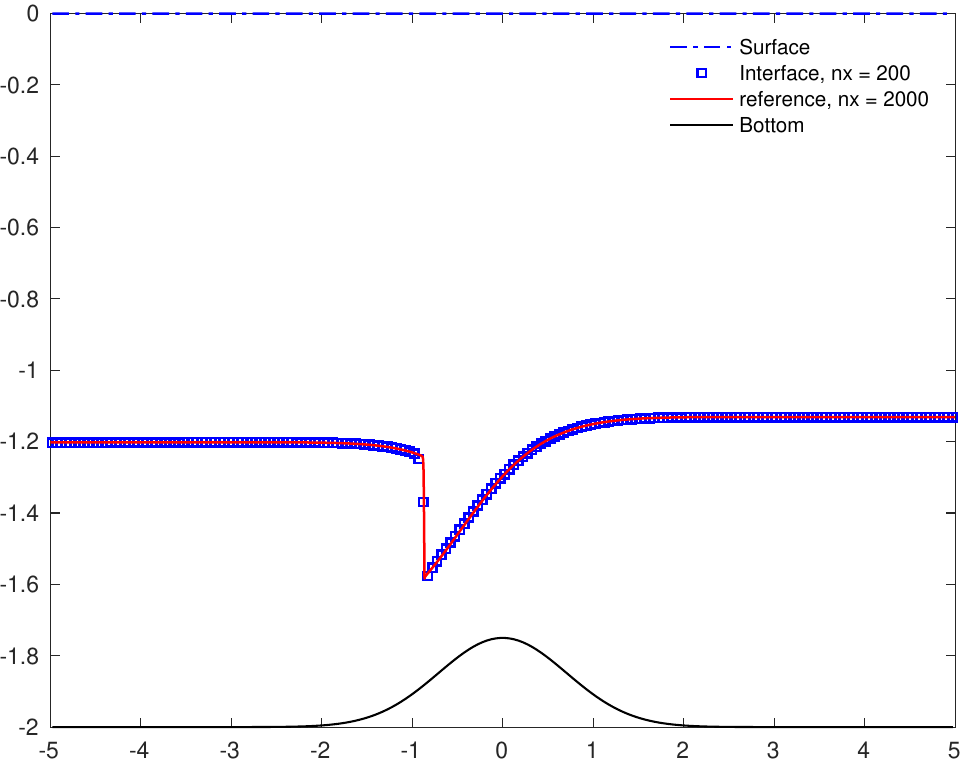}
  }
  \subfigure[comparison, $nx = 2000$]{ 
  \centering
  \includegraphics[width=7cm,scale=1]{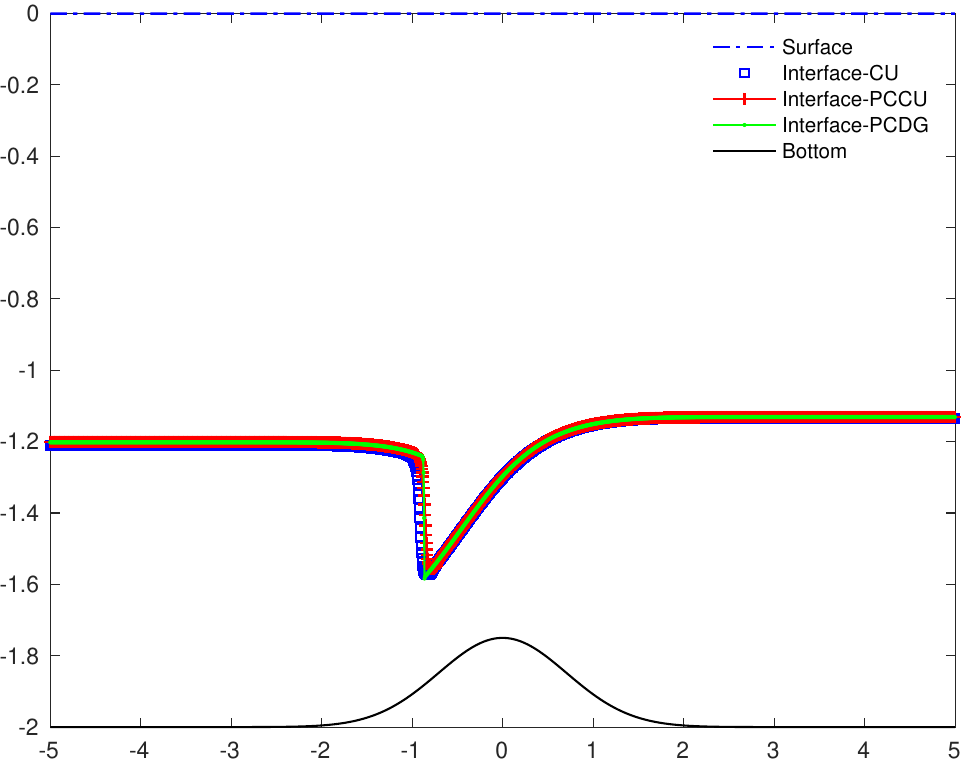}
  }
  \caption{Example \ref{dam1d}: Left: The discrete steady states for the water surface $h_1+w$ and interface $w$, with 200 cells; right: using CU, PCCU, PCDG-still schemes for comparison, with 2000 grid cells. }\label{dam_steady}
\end{figure}

\begin{example}{\bf Isolated internal shock}\label{iso1d}
\end{example}
We take the same example as in \cite{castro2013entropy} to demonstrate the advantage of the path-conservative discontinuous Galerkin method,
which corresponds to an isolated internal shock wave traveling to the right. The initial conditions on the computational domain $[0,1]$ are given by
\begin{equation}\label{iso}
\boldsymbol{u}(x,0)=(h_1,m_1,h_2, m_2)(x,0) = \left\{\begin{array}{lll}
  \boldsymbol{u}_L,&\text{if} \  x < 0.5, \\
  \boldsymbol{u}_R,&\text{otherwise},\\
    \end{array}\right.
\end{equation}
with
\begin{equation}\label{iso_u}
  \begin{aligned}
    \boldsymbol{u}_L =& ((h_1)_L,(m_1)_L,(h_2)_L,(m_2)_L)^T = (0.8817,-0.1738,1.091,0.1613)^T,\\
    \boldsymbol{u}_R =& ((h_1)_R,(m_1)_R,(h_2)_R,(m_2)_R)^T = (0.3700,-0.1868,1.593,0.1742)^T.
  \end{aligned}
\end{equation}
For the choice of the flat bottom $b$, we need to select the constant $b_{\text{ref}}$ properly to make sure that the water surface is relatively small, which corresponds to the free surface of the undisturbed water $\varepsilon = h_1+h_2+b_{\text{ref}}\approx 0$. We take the following three reasonable choices:
\begin{equation*}
  \begin{aligned}
    & (a) \ b_{\text{ref}} = -((h_1)_L + (h_2)_L);\\
    & (b) \ b_{\text{ref}} = -((h_1)_R + (h_2)_R);\\
    & (c) \ b_{\text{ref}} = -((h_1)_L + (h_2)_L + (h_1)_R + (h_2)_R)/2.
  \end{aligned}
\end{equation*}
The constant gravitational acceleration is $g = 9.81$ and the density ratio is $r=0.98$. We employ the PCDG-still method for our simulation over three different bottom choices until $t=0.6$ with 200 and 2000 uniform cells for comparison. Fig. \ref{iso_s} presents the corresponding results and zoom-in of the lower layer water height $h_2$ and water depth $h_1+h_2$. We also compare the current solutions with second-order CU and PCCU schemes under a finer mesh of 2000 cells and show the simulations in Fig. \ref{iso_s}. Herein CUa, CUb, CUc represent the central-upwind method with $b_{\text{ref}}$ given by three options, respectively. Since the bottom function is constant and appears in the equations only by its derivative, the exact solution is not expected to depend on the particular choice of $b$.  We can see that the PCDG solution is much more robust. It is independent of the reference water surface level, which yields almost the same results as the CUc and PCCU methods as expected. While the results acquired by the other two versions of the CU method are qualitatively different.

\begin{figure}[htb!]
  \centering 
  \subfigure[PCDG, $nx = 200$]{
  \centering
  \includegraphics[width=6.5cm,scale=1]{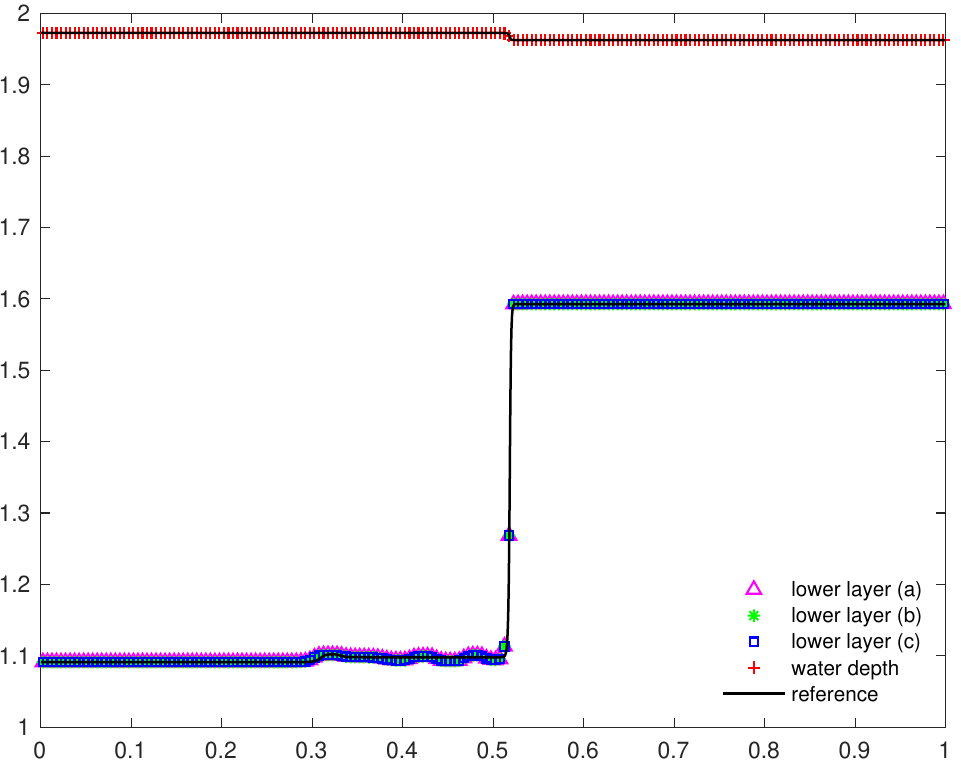}
  }
  \subfigure[comparison, $nx = 2000$]{
  \centering
  \includegraphics[width=6.5cm,scale=1]{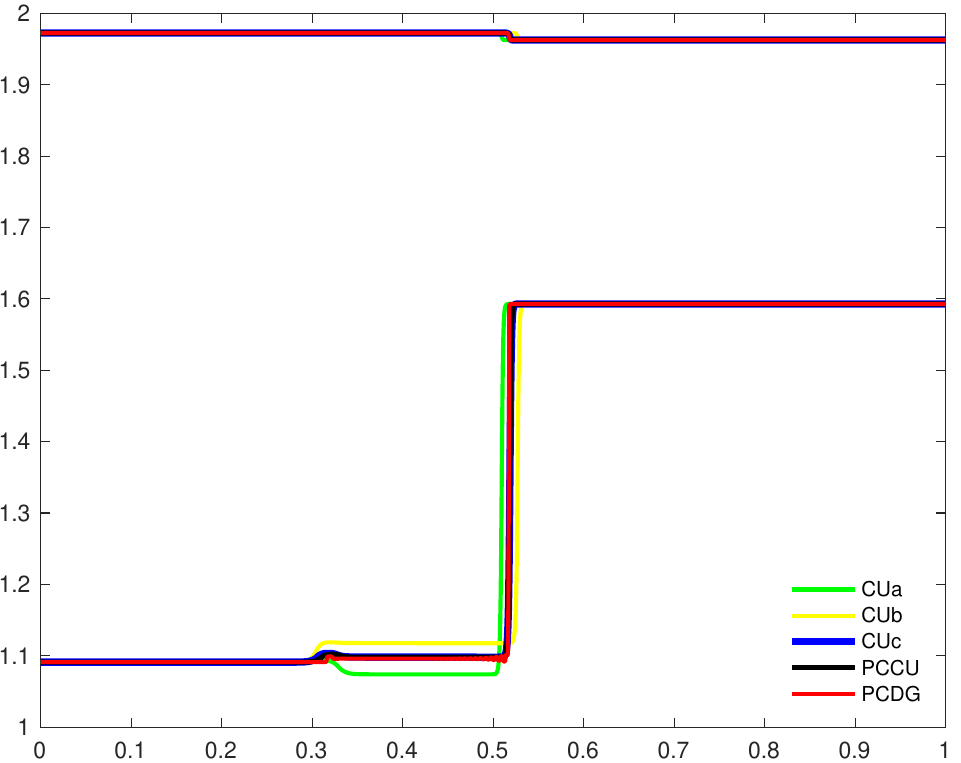}
  }
  \subfigure[lower layer $h_2$, $nx = 200$]{
  \centering
  \includegraphics[width=6.5cm,scale=1]{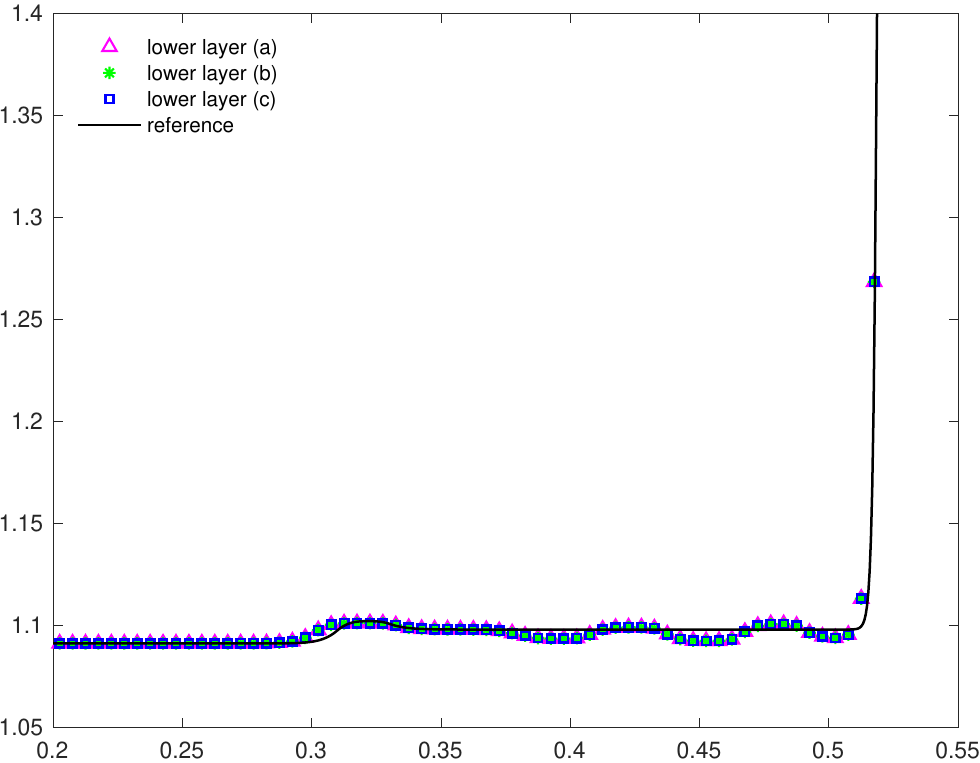}
  }
  \subfigure[lower layer $h_2$, comparison]{
  \centering
  \includegraphics[width=6.5cm,scale=1]{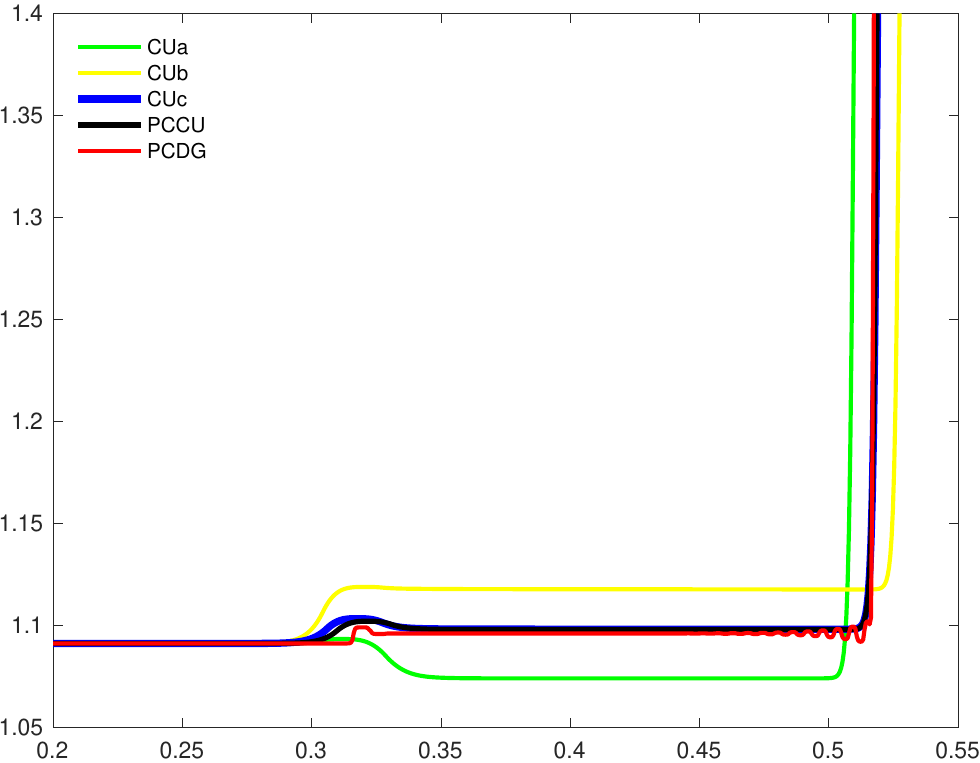}
  }
  \subfigure[water depth $h_1+h_2$, $nx = 200$]{
  \centering
  \includegraphics[width=6.5cm,scale=1]{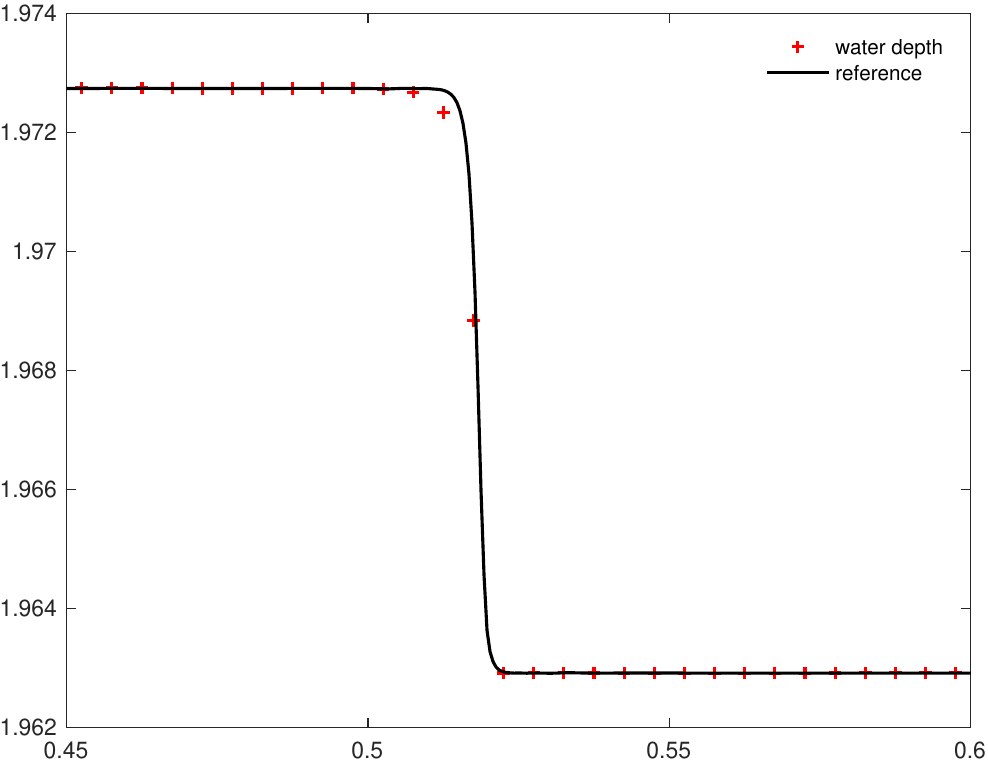}
  }
  \subfigure[water depth $h_1+h_2$, comparison]{
  \centering
  \includegraphics[width=6.5cm,scale=1]{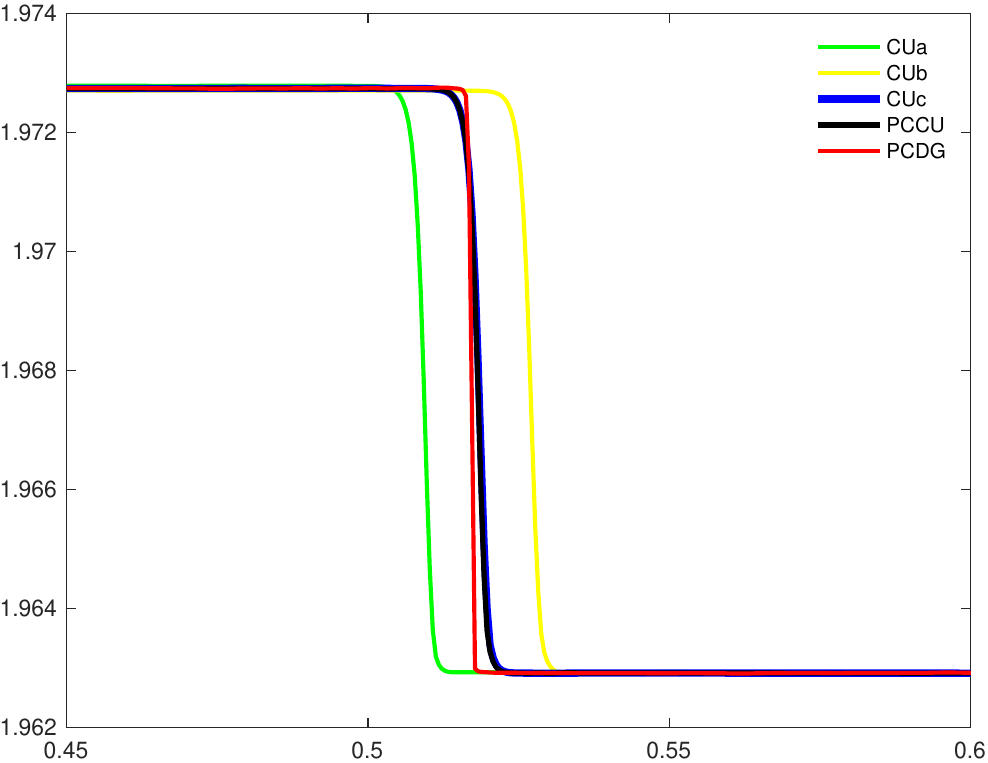}
  }
  \caption{Example \ref{iso1d}: From top to bottom: Numerical solutions of the lower layer $h_2$ and water depth $h_1+h_2$, zoom-in of the lower layer $h_2$ and the water depth $h_1+h_2$. Left: using 200 and 2000 grid cells; right: using CU, PCCU, PCDG-still schemes for comparison.}\label{iso_s}
\end{figure}


\begin{example}{\bf Barotropic tidal flow }\label{tidal1d}
\end{example}
This example is aimed at mimicking a tidal wave by imposing a barotropic periodic perturbation as the inflow boundary condition. Taken from \cite{diaz2019path}, the initial conditions with an internal shock are given by
\begin{equation}\label{tidal_ini}
  \boldsymbol{u}(x,0)=(h_1,m_1,h_2, m_2)(x,0) = \left\{\begin{array}{lll}
    \boldsymbol{u}_L,&\text{if} \  x \leq 0, \\
    \boldsymbol{u}_R,&\text{otherwise},\\
      \end{array}\right.
  \end{equation}
with
\begin{equation}\label{tidal_u}
  \begin{aligned}
  \boldsymbol{u}_L =& ((h_1)_L,(m_1)_L,(h_2)_L,(m_2)_L)^T = (0.69914,-0.21977,1.26932,0.20656)^T,\\
  \boldsymbol{u}_R =& ((h_1)_R,(m_1)_R,(h_2)_R,(m_2)_R)^T = (0.37002,-0.18684,1.59310,0.17416)^T.
  \end{aligned}
\end{equation}
As can be seen from the numerical experiments performed in Example \ref{iso1d}, the reference water surface level
\begin{equation}\label{b_tidal}
  b(x)\equiv b_{\text{ref}} = -\frac{1}{2}((h_1)_L + (h_2)_L + (h_1)_R + (h_2)_R)
\end{equation}
is the best option for the second-order CU scheme. We choose this flat bottom function and simulate the solutions at different times $t=10,25,60,64$ on the computational domain $[-10,10]$. The gravitation constant is $g = 9.81$ and the density ratio is $r=0.98$. Open boundary conditions are imposed on the right $x=10$. On the left $x=-10$, we impose the periodic in time boundary conditions for the $h_1$ and $h_2$ components
\begin{equation}\label{h_bc}
 \begin{aligned}
  &h_1(-10,0) = (h_1)_L + (h_1)_L\dfrac{0.03}{|b_{\text{ref}}|}\sin\left(\dfrac{\pi t}{50}\right), \\
  &h_2(-10,0) = (h_2)_L + (h_1)_L\dfrac{0.03}{|b_{\text{ref}}|}\sin\left(\dfrac{\pi t}{50}\right),
 \end{aligned}
\end{equation}
and zero-order interpolation for the $m_1$ and $m_2$ components. For comparison, we run the simulations by CU, PCCU, and PCDG-still schemes with 1000 uniform cells, and exhibit the water surface $h_1+w$ (left) and interface $w$ in Fig. \ref{tidal_com}. As stated in \cite{diaz2019path}, even though the CU solution behaves similarly with path-conservative methods at small times, it will produce spurious oscillations and becomes unstable eventually for larger times. In contrast, our developed PCDG-still solutions remain stable at all times, which demonstrates the significance of the path-conservative approach.

\begin{figure}[htb!]
  \centering
  \subfigure[water surface $h_1+w$, $t=10$]{
  \centering
  \includegraphics[width=5.1cm,scale=1]{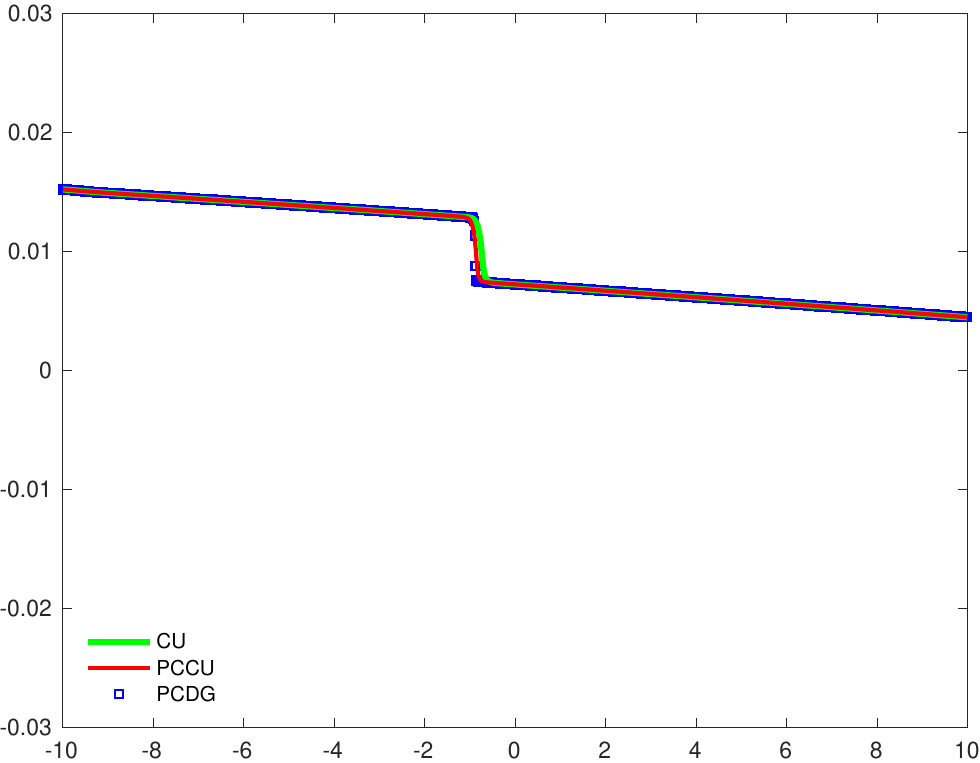}
  }
  \subfigure[interface $w$, $t=10$]{
  \centering
  \includegraphics[width=5.1cm,scale=1]{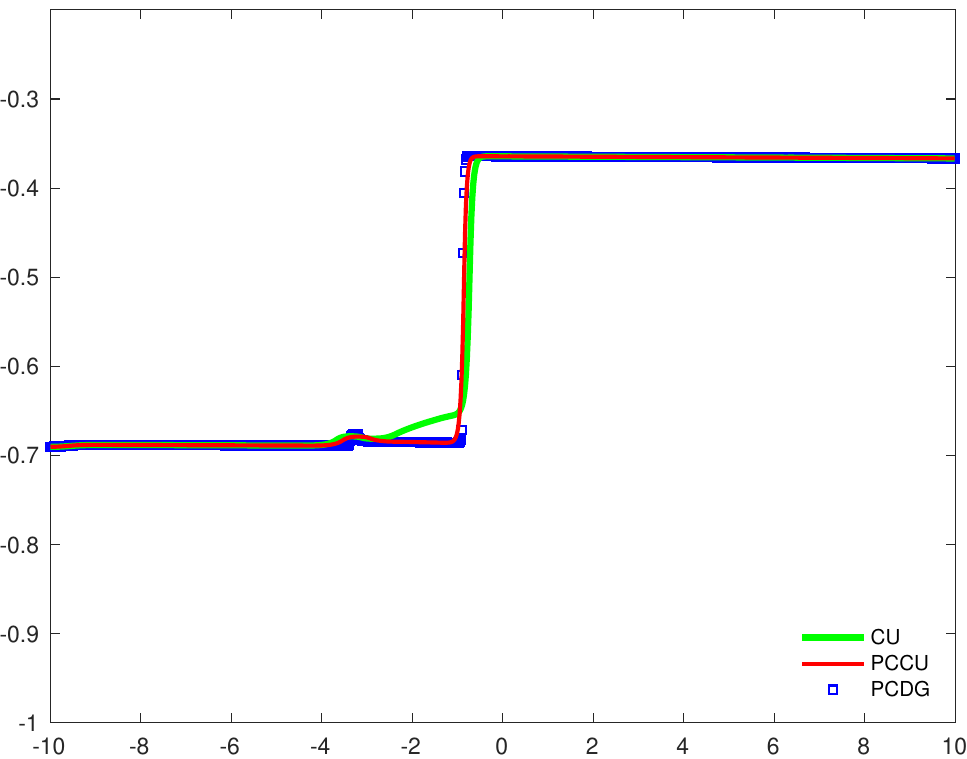}
  }
  \subfigure[water surface $h_1+w$, $t=25$]{
  \centering
  \includegraphics[width=5.1cm,scale=1]{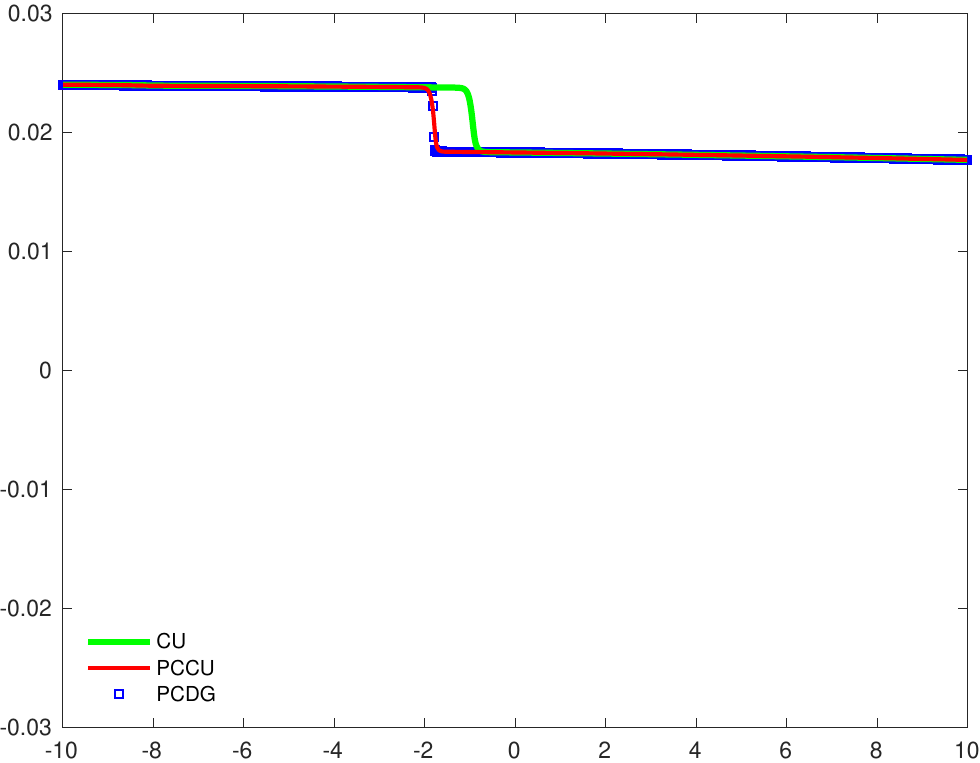}
  }
  \subfigure[interface $w$, $t=25$]{
  \centering
  \includegraphics[width=5.1cm,scale=1]{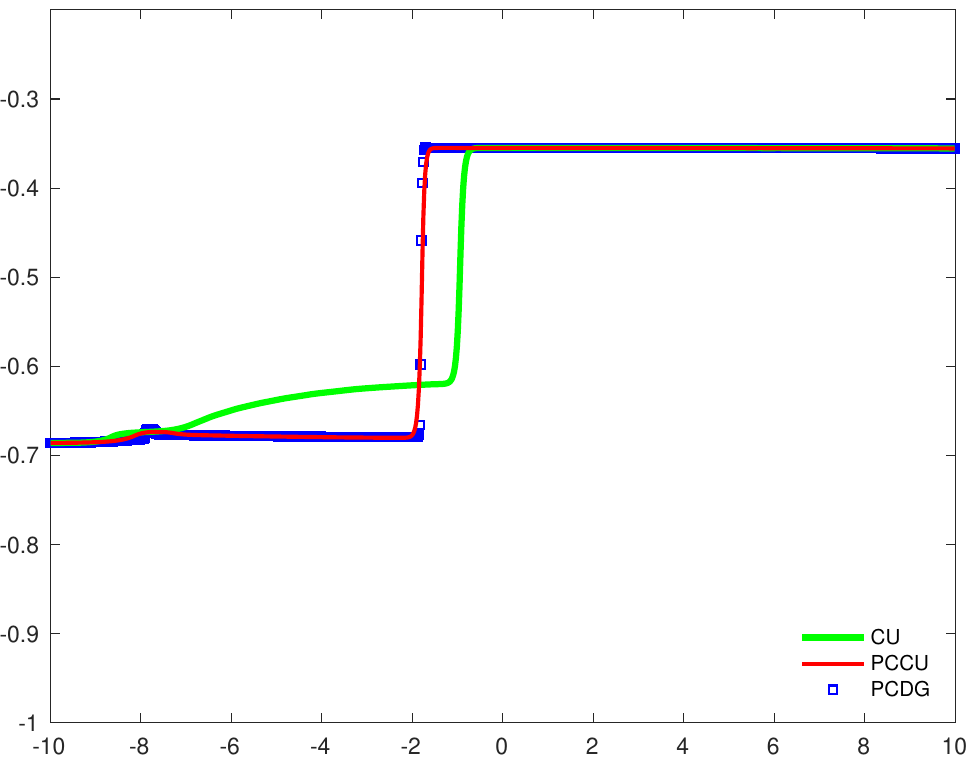}
  }
  \subfigure[water surface $h_1+w$, $t=60$]{
  \centering
  \includegraphics[width=5.1cm,scale=1]{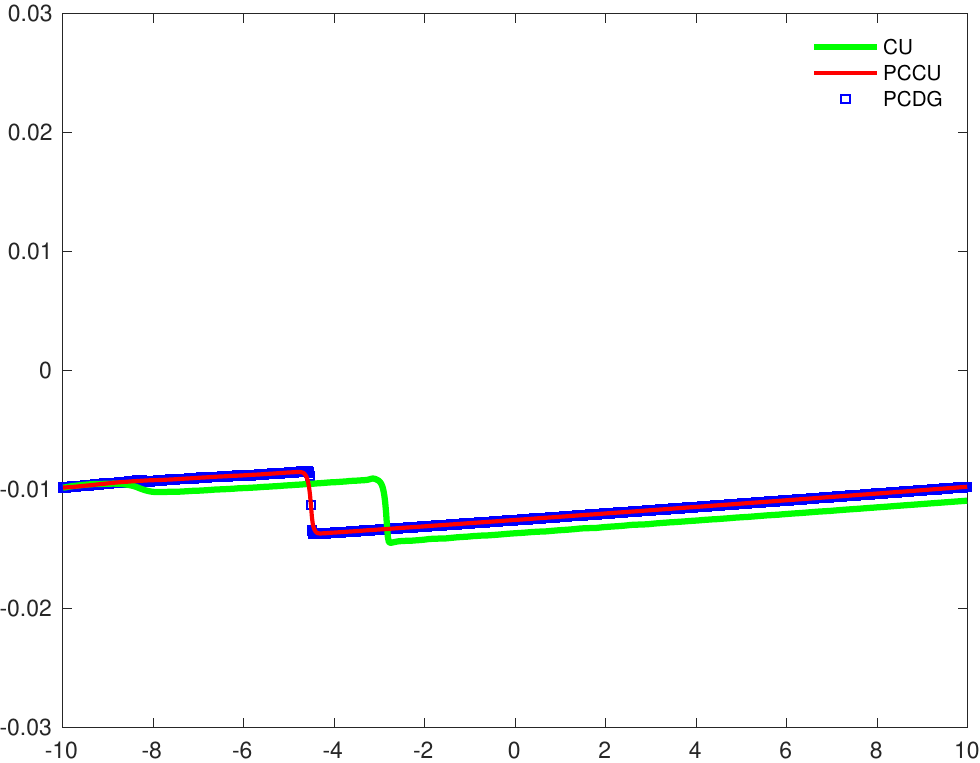}
  }
  \subfigure[interface $w$, $t=60$]{
  \centering
  \includegraphics[width=5.1cm,scale=1]{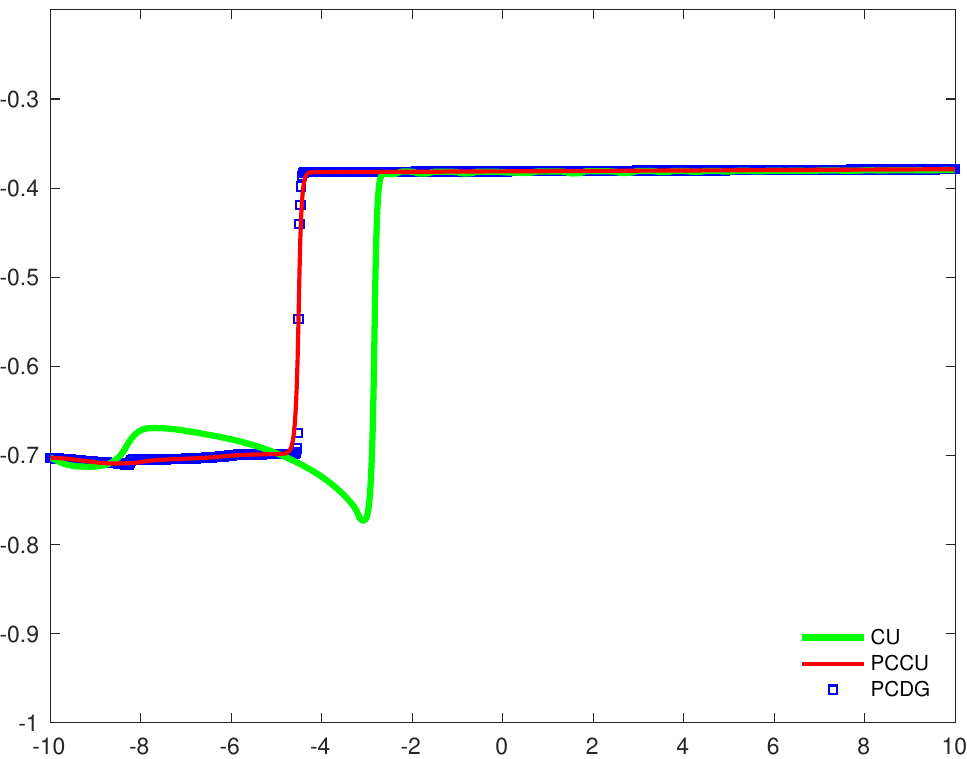}
  }
  \subfigure[water surface $h_1+w$, $t=64$]{
  \centering
  \includegraphics[width=5.1cm,scale=1]{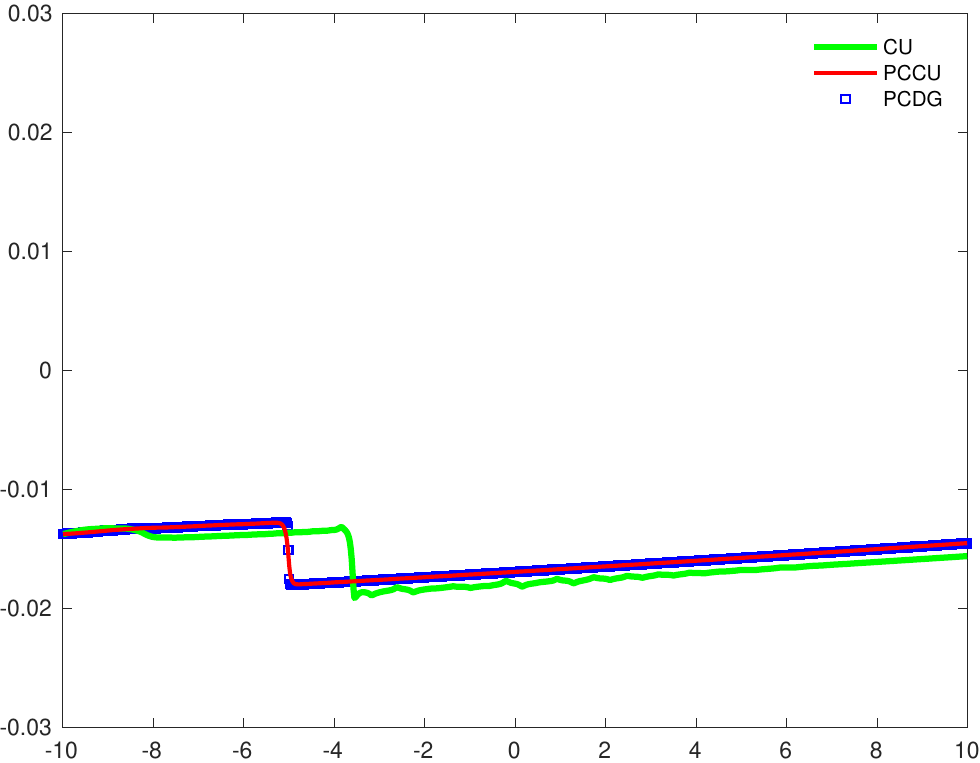}
  }
  \subfigure[interface $w$, $t=64$]{
  \centering
  \includegraphics[width=5.1cm,scale=1]{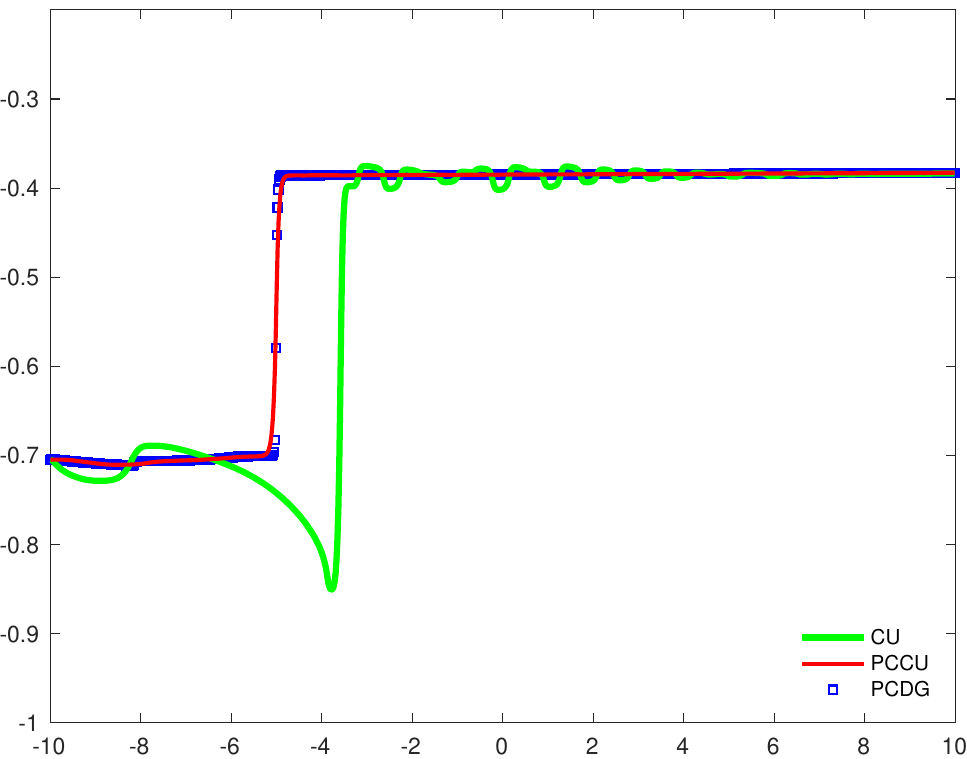}
  }
  \caption{Example \ref{tidal1d}: Numerical solutions of the water surface $h_1+w$ (left) and interface $w$ (right) at different times $t=10,25,60,64$ (from top to bottom), using CU, PCCU, PCDG-still schemes for comparison. }\label{tidal_com}
\end{figure}

\begin{example}{\bf  Test for the moving water well-balanced property}\label{exact1D_mov}
\end{example}
In this example, we demonstrate the well-balanced property for the generalized steady state (\ref{moving_water:2LSWE}), which consists of non-hydrostatic equilibria. A discontinuous bottom topography
\begin{equation}\label{mov_b1D}
  b(x)=\left\{\begin{array}{lll}
    -2,&\text{if}\ x<0, \\
    -1,&\text{otherwise},\\
    \end{array}\right.
\end{equation}
is considered and the initial conditions are given by
\begin{equation}\label{movwb_u1D}
  \begin{aligned}
  &h_1(x,0)=(h_1)_{\text{eq}}(x) = \left\{\begin{array}{lll}
      1.22373355048230,&  \text{if} \  x < 0, \\
      1.44970064153589,&  \text{otherwise},\\
      \end{array}\right.  \\
  &h_2(x,0)=(h_2)_{\text{eq}}(x) = \left\{\begin{array}{lll}
      0.968329515483846,&  \text{if} \  x < 0, \\
      1.12439026921484, &\text{otherwise},\\
      \end{array}\right.\\
  & m_1(x,0) = (m_1)_{\text{eq}}(x) =12,\quad m_2(x,0) = (m_2)_{\text{eq}}(x) =10.
  \end{aligned}
\end{equation}
Note that even though the water height $h_1$, $h_2$ and the bottom topography have a jump discontinuity at the point $x = 0$, it is still exactly in equilibrium for the initial data since these initial values $E_1\equiv50$, $E_2\equiv 55$. Following the same setup in \cite{kurganov2023well}, we take the gravitation constant $g = 10$ and the density ratio $r=0.98$. We simulate the solutions until $t=0.05$ on the computational domain $[-1,1]$ with 100 uniform cells. Table \ref{2LSWE:wbmov_1D} shows the corresponding $L^1$ and $L^{\infty}$ errors at double precision using piecewise linear and quadratic polynomial basis. We can see that the numerical errors based on our moving water equilibria preserving path-conservative discontinuous Galerkin method are at the level of machine accuracy, the moving equilibrium state can be exactly preserved even under a relatively coarse mesh.
\begin{table}[htb]
  \centering
  \caption{Example \ref{exact1D_mov}: $L^1$ and $L^{\infty}$ errors for one-dimensional moving water equilibrium state with the bottom topography  (\ref{mov_b1D}), using PCDG-moving method.}\label{2LSWE:wbmov_1D}
  \begin{tabular}{ c c c c c c c c}
   \toprule
      &{error}&{$h_1$}&{$m_1$}&{$h_2$}&{$m_2$}&{$E_1$}&{$E_2$}\\
    \midrule
    \multirow{2}{*}{$P^1$ }&$L^1$        & 9.05E-14 & 1.81E-13 & 6.54E-14 & 1.92E-13 & 4.01E-12 & 4.18E-12\\
                          ~&$L^{\infty}$ & 9.46E-14 & 4.26E-13 & 1.05E-13 & 6.00E-13 & 2.95E-12 & 3.30E-12\\
    \midrule
    \multirow{2}{*}{$P^2$ }&$L^1$        & 7.36E-14 & 6.71E-13 & 7.12E-14 & 6.69E-13 & 1.10E-12 & 9.45E-13\\
                           &$L^{\infty}$ & 1.29E-13 & 1.15E-12 & 1.37E-13 & 1.39E-12 & 3.52E-12 & 3.46E-12\\
    \bottomrule
  \end{tabular}
\end{table}

\begin{example}{\bf Small perturbation of moving water equilibrium}\label{2lswe:permoving}
\end{example}
Almost the same setup in Example \ref{exact1D_mov}, to investigate the performance of our constructed methods, we add a small perturbation to the upper layer depth $(h_1)_{\text{eq}}$
\begin{equation}\label{movper_u1D}
  \begin{aligned}
  &h_1(x,0)=(h_1)_{\text{eq}}(x) + \left\{\begin{array}{lll}
      0.001,&  \text{if} \  x \in [-0.6,-0.5], \\
      0,&  \text{otherwise},\\
      \end{array}\right.  \\
  &h_2(x,0)=(h_2)_{\text{eq}}(x),\quad m_1(x,0) = (m_1)_{\text{eq}}(x),\quad m_2(x,0) = (m_2)_{\text{eq}}(x).
  \end{aligned}
\end{equation}
The final time is set as $t=0.02,0.05,0.08$. We use both of our proposed well-balanced schemes consisting of 200 uniform cells in contrast with the results resolved by the  PCDG-moving scheme involving 1000 uniform grids. Fig. \ref{mov_percom} shows the difference $h_1(x,t) - (h_1)_{\text{eq}}$ at different times. Although the PCDG-still scheme exhibits good behaviors in most experiments, it shows a failed simulation for non-hydrostatic flows and generates unphysical waves beyond the propagating perturbation. In contrast, our PCDG-moving method gives well-resolved solutions that agree with the reference solutions under a relatively coarse mesh, and the solutions are free of spurious numerical oscillations near the discontinuity thanks to the well-balanced property.

\begin{figure}[htb!]
  \centering
  \subfigure[ $t=0$ ]{
  \centering
  \includegraphics[width=6.5cm,scale=1]{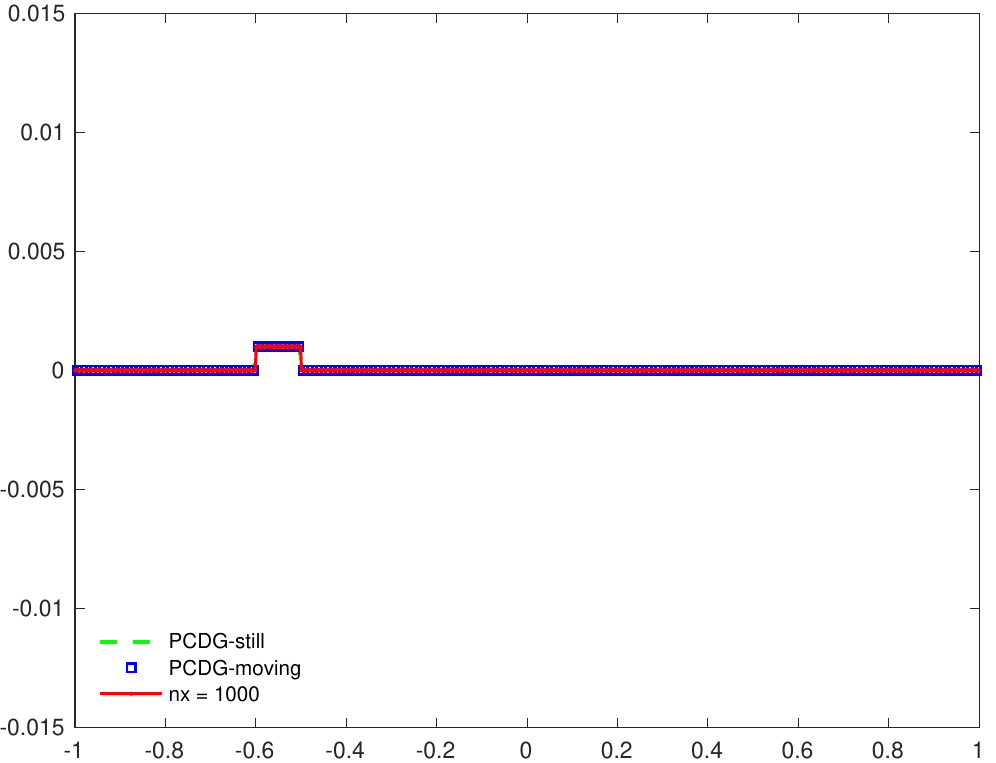}
  }
  \subfigure[ $t=0.02$ ]{
  \centering
  \includegraphics[width=6.5cm,scale=1]{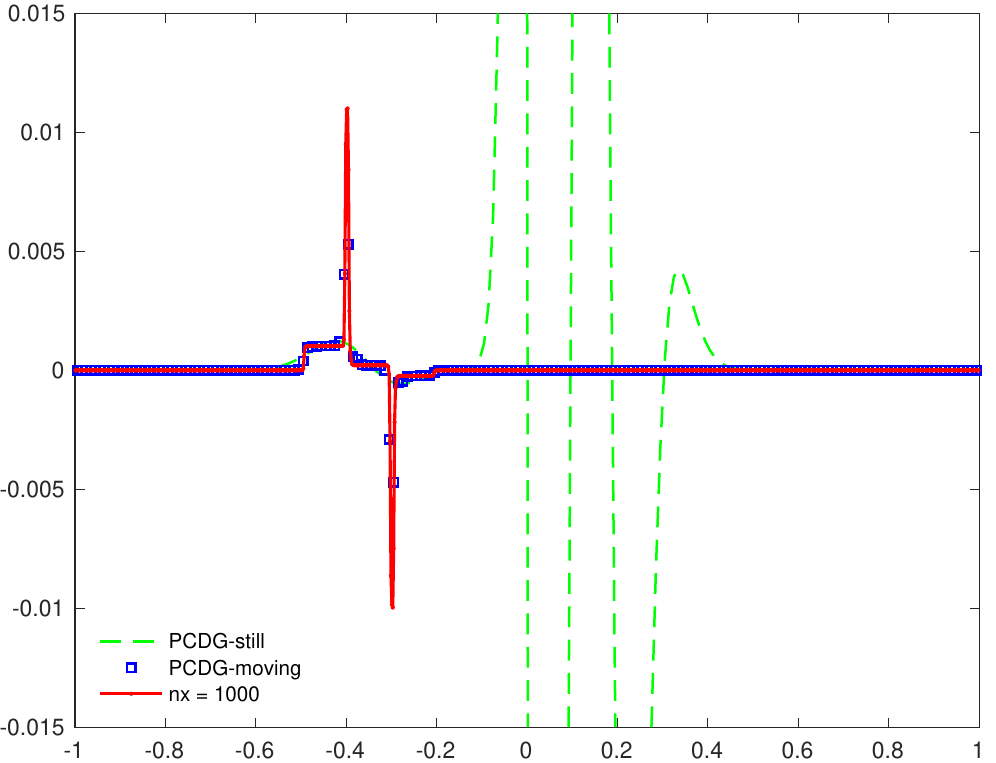}
  }
  \subfigure[ $t=0.05$ ]{
  \centering
  \includegraphics[width=6.5cm,scale=1]{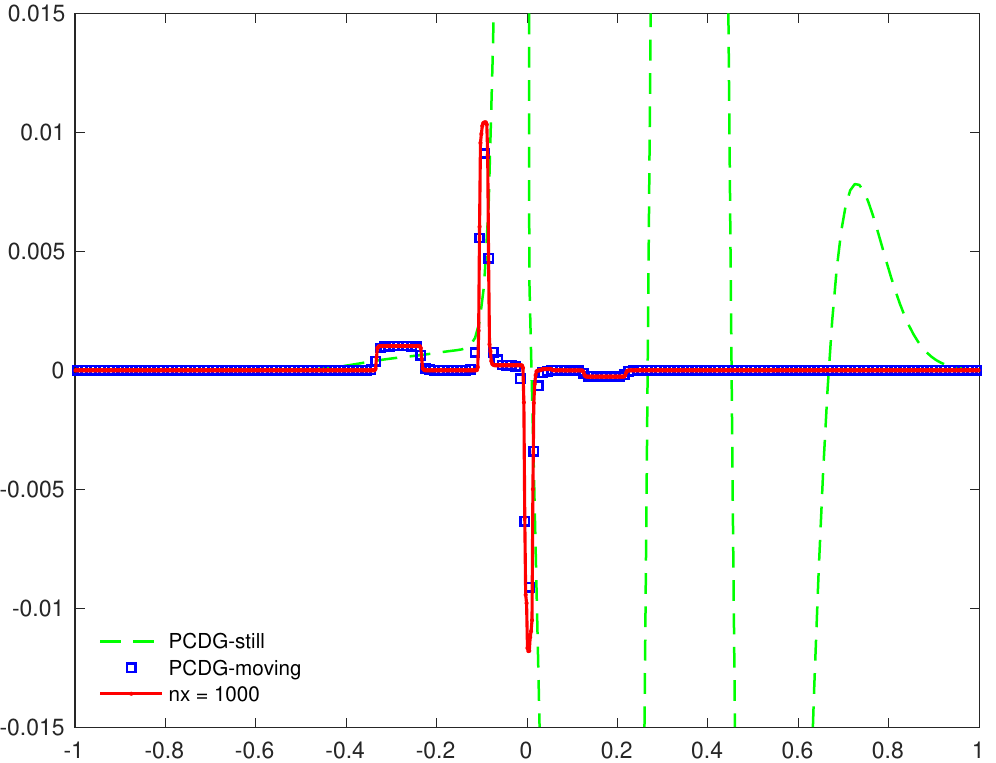}
  }
  \subfigure[ $t=0.08$ ]{
  \centering
  \includegraphics[width=6.5cm,scale=1]{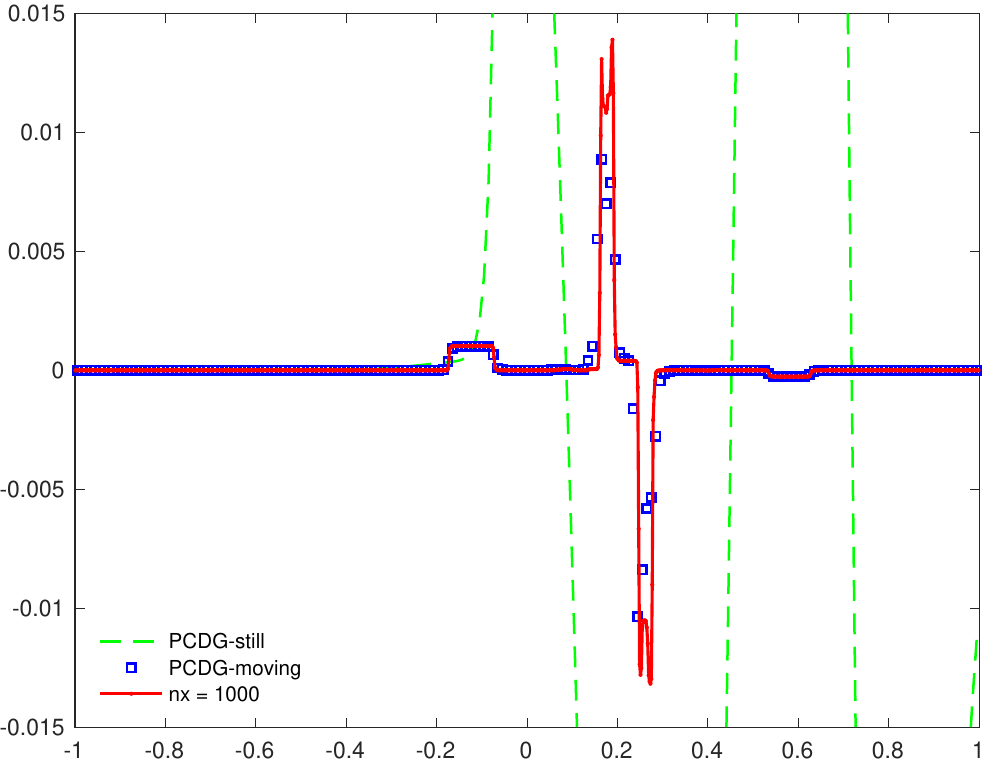}
  }
  \caption{Example \ref{2lswe:permoving}: Numerical solutions of the difference $h_1(x,t) - (h_1)_{\text{eq}}$ at different times $t=0,0.02,0.05,0.08$, computed by PCDG-still, PCDG-moving schemes with 200 and 1000 grid cells. }\label{mov_percom}
\end{figure}

\begin{example}{\bf Riemann problem over discontinuous bottom topography}\label{rie1d}
\end{example}
To illustrate the performance of our developed well-balanced PCDG scheme when containing discontinuous solutions, we implement two dam breaking test cases used in \cite{chu2022fifth} over a discontinuous bottom topography on a domain $[-1,1]$
\begin{equation}\label{riemov_b1d}
  b(x)=\left\{\begin{array}{lll}
    -2,&\text{if}\ x<0, \\
    -1.5,&\text{otherwise}.\\
    \end{array}\right.
\end{equation}
Two different Riemann initial conditions are taken as follows:
\begin{equation}\label{riemov_u1}
\text{Test 1:} \ (h_1,m_1,h_2, m_2)(x,0) = \left\{\begin{array}{lll}
    (1.0,1.5,1.0,1.0),&  \text{if} \  x < 0, \\
    (0.8,1.2,1.2,1.8),&\text{otherwise},\\
    \end{array}\right.
\end{equation}
\begin{equation}\label{riemov_u2}
\text{Test 2:} \ (h_1,m_1,h_2, m_2)(x,0) = \left\{\begin{array}{lll}
    (1.5,1.1,1.0,1.4),&  \text{if} \  x < 0, \\
    (1.2,1.6,0.9,1.2),&\text{otherwise}.\\
    \end{array}\right.
\end{equation}
The constant gravitational acceleration is $g = 10$ and the density ratio is $r=0.98$. The final time of our simulation is set as $t=0.1$.  
Fig. \ref{riemov1} displays the numerical results of the conservative and equilibrium variables for Test 1, calculated by our developed methods on 1000 uniform grid cells. For comparison, we also run the same tests by the second-order PCCU scheme. The corresponding solutions for Test 2 are plotted in Fig. \ref{riemov2}. We can see that both the PCCU and PCDG-still schemes exist some unphysical spikes and oscillations located at the point $x=0$, where the bottom function $b$ is discontinuous. By contrast, more accurate fluid motion captured by the moving water equilibria preserving PCDG method can be observed even though the bottom topography is in poor condition.

\begin{figure}[htb!]
 \centering
 \subfigure[upper layer $h_1$]{
 \centering
 \includegraphics[width=6.5cm,scale=1]{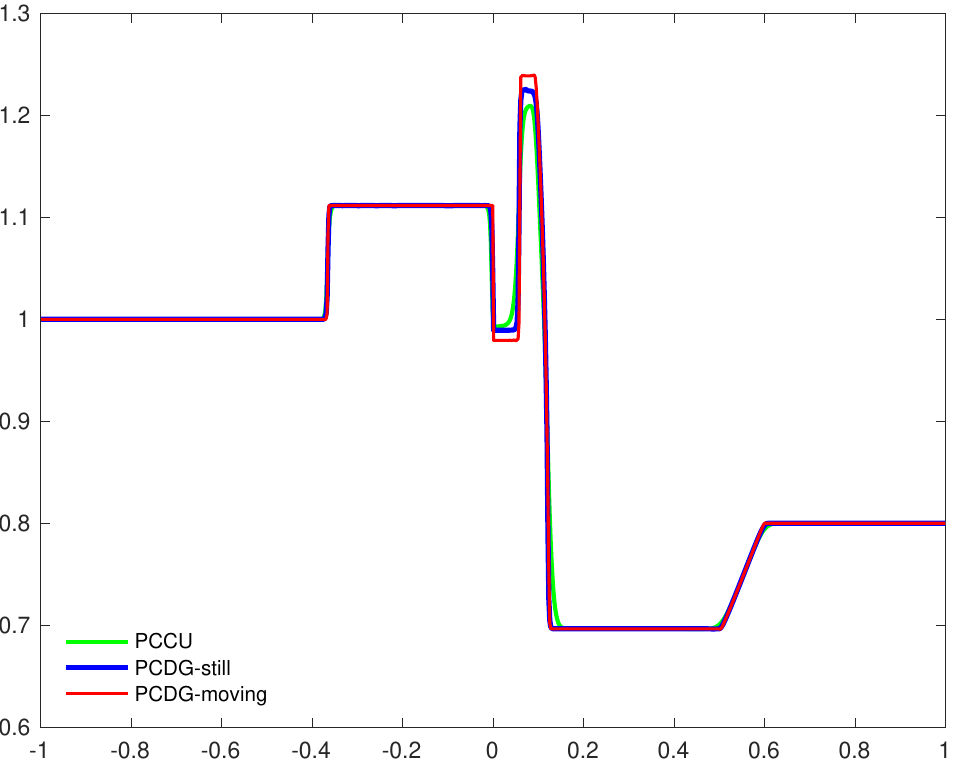}
 }
 \subfigure[discharge $m_1$]{
 \centering
 \includegraphics[width=6.5cm,scale=1]{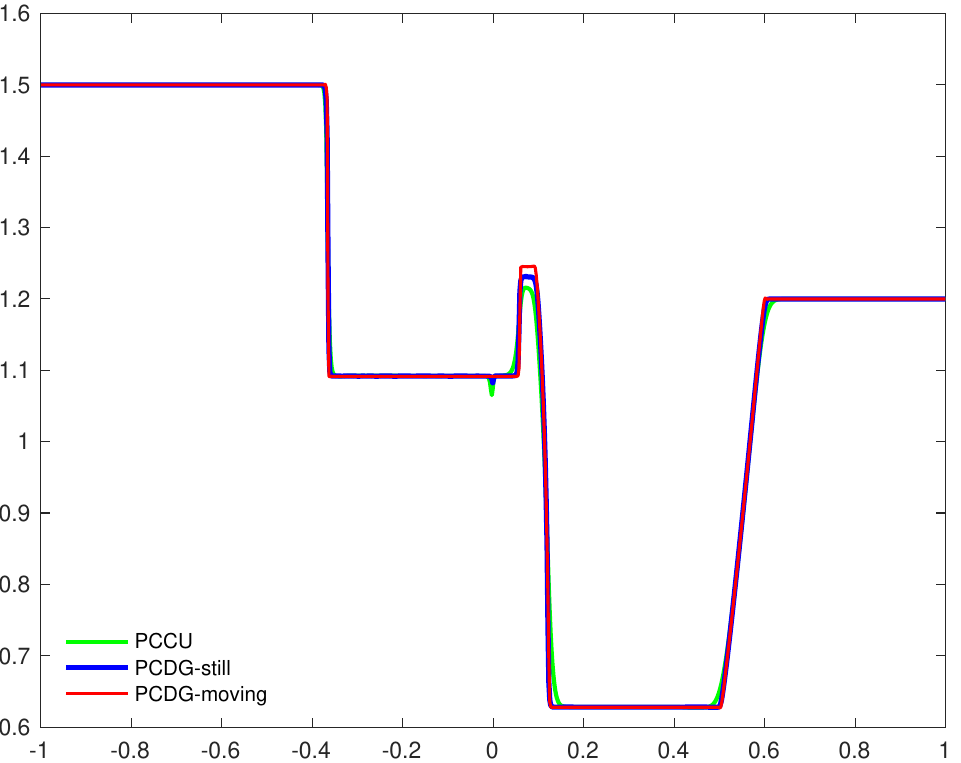}
 }
 \subfigure[ energy $E_1$]{
 \centering
 \includegraphics[width=6.5cm,scale=1]{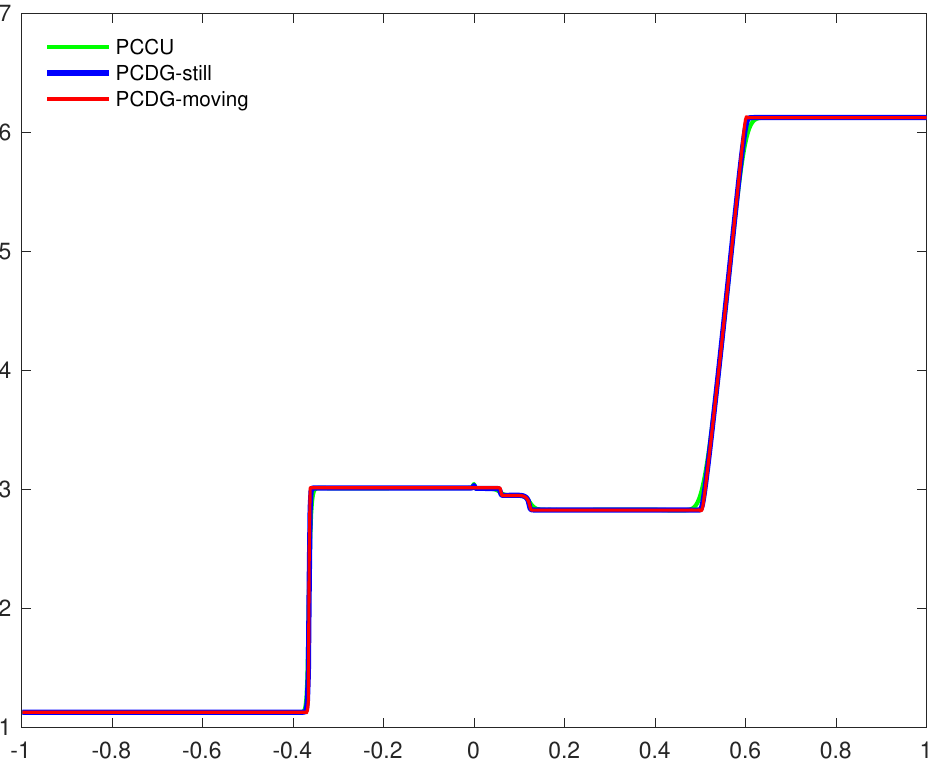}
 }
 \subfigure[lower layer $h_2$]{
 \centering
 \includegraphics[width=6.5cm,scale=1]{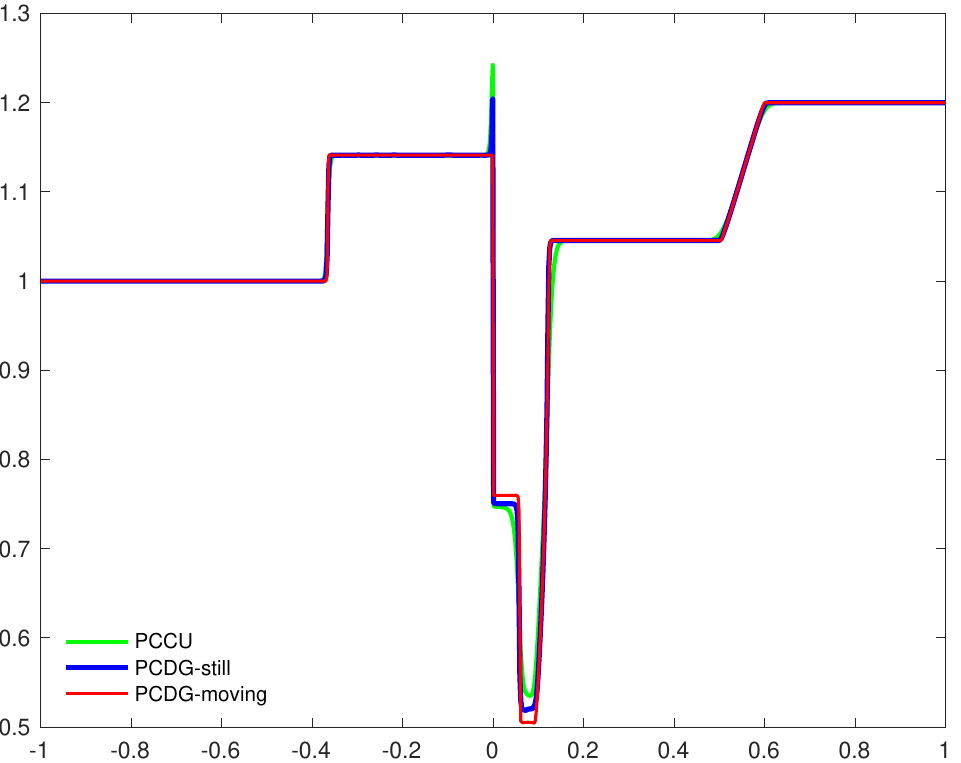}
 }
 \subfigure[discharge $m_2$]{
 \centering
 \includegraphics[width=6.5cm,scale=1]{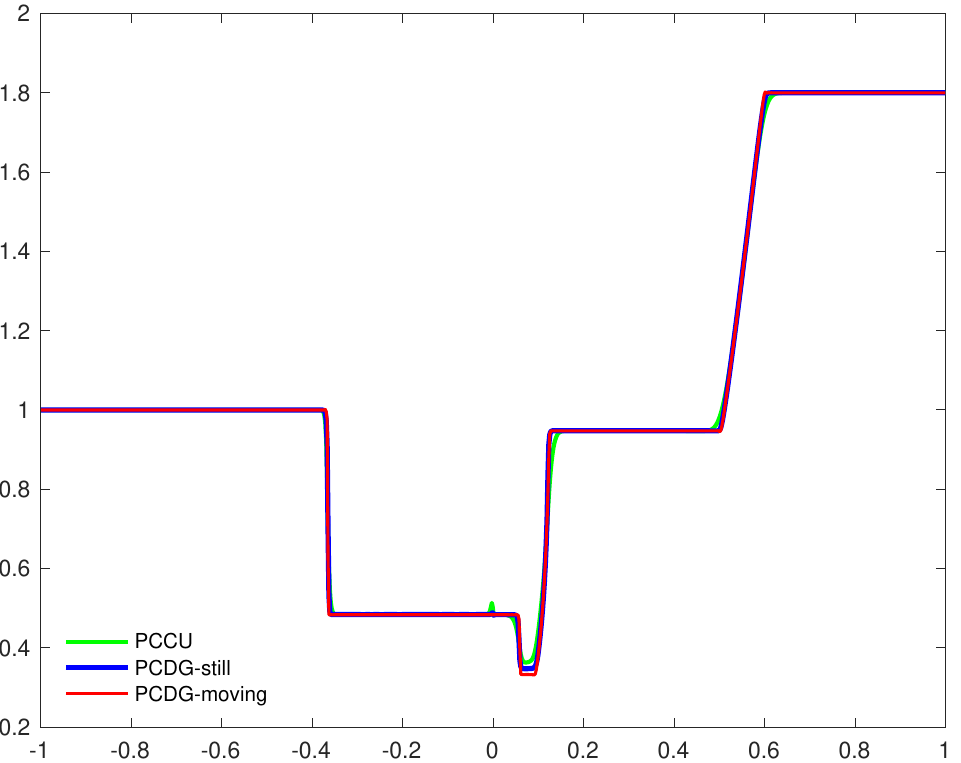}
 }
 \subfigure[energy $E_2$]{
 \centering
 \includegraphics[width=6.5cm,scale=1]{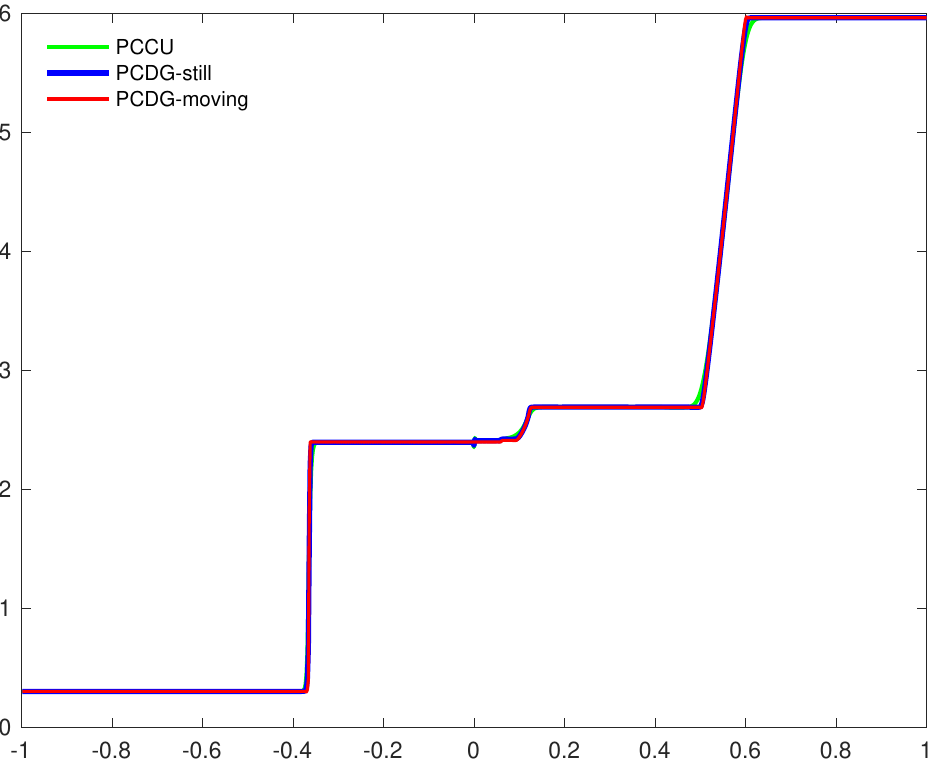}
 }
 \caption{Example \ref{rie1d}: Numerical solutions for Test 1 (\ref{riemov_u1}), computed by PCCU, PCDG-still, PCDG-moving schemes, with 1000 grid cells. }\label{riemov1}
\end{figure}

\begin{figure}[htb!]
  \centering
  \subfigure[upper layer $h_1$]{
  \centering
  \includegraphics[width=6.5cm,scale=1]{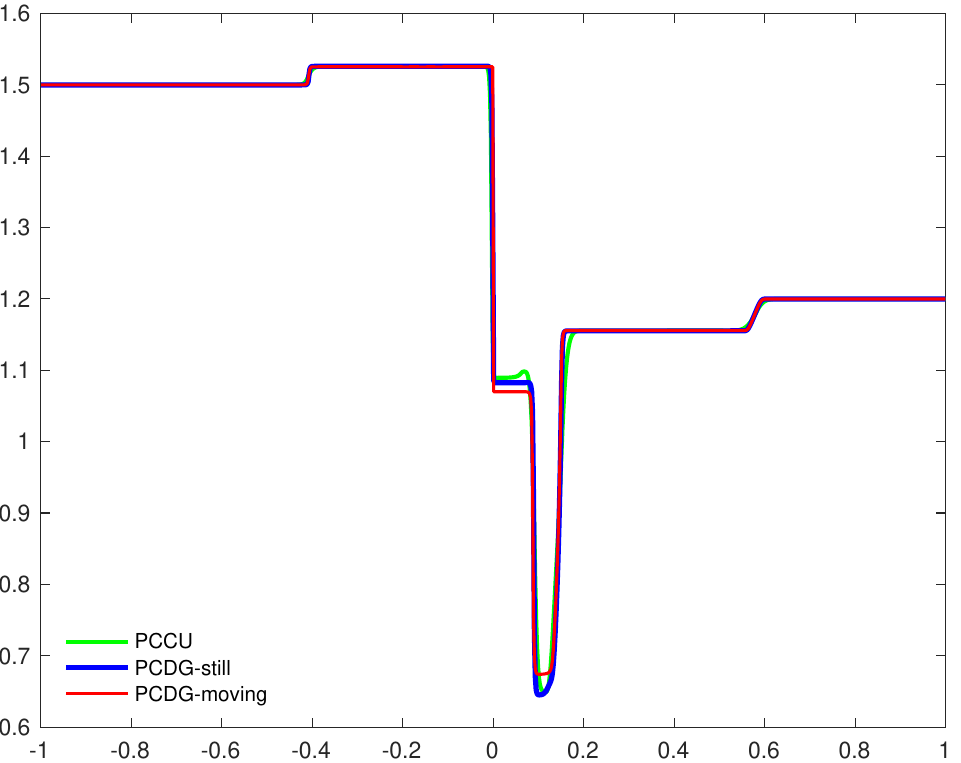}
  }
  \subfigure[discharge $m_1$]{
  \centering
  \includegraphics[width=6.5cm,scale=1]{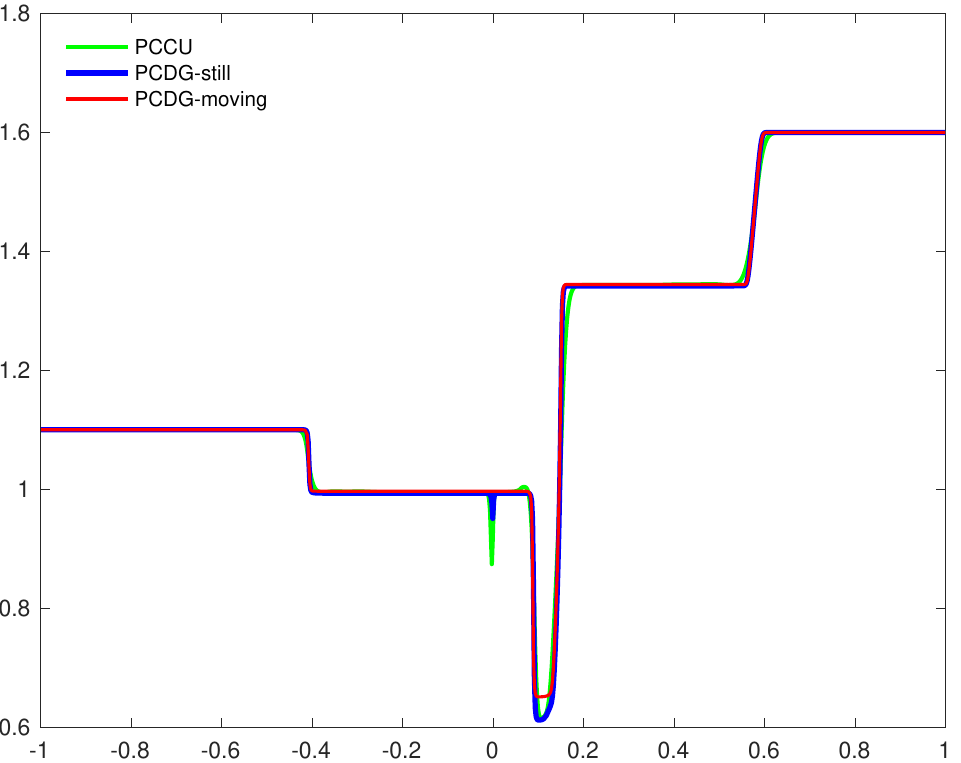}
  }
  \subfigure[ energy $E_1$]{
  \centering
  \includegraphics[width=6.5cm,scale=1]{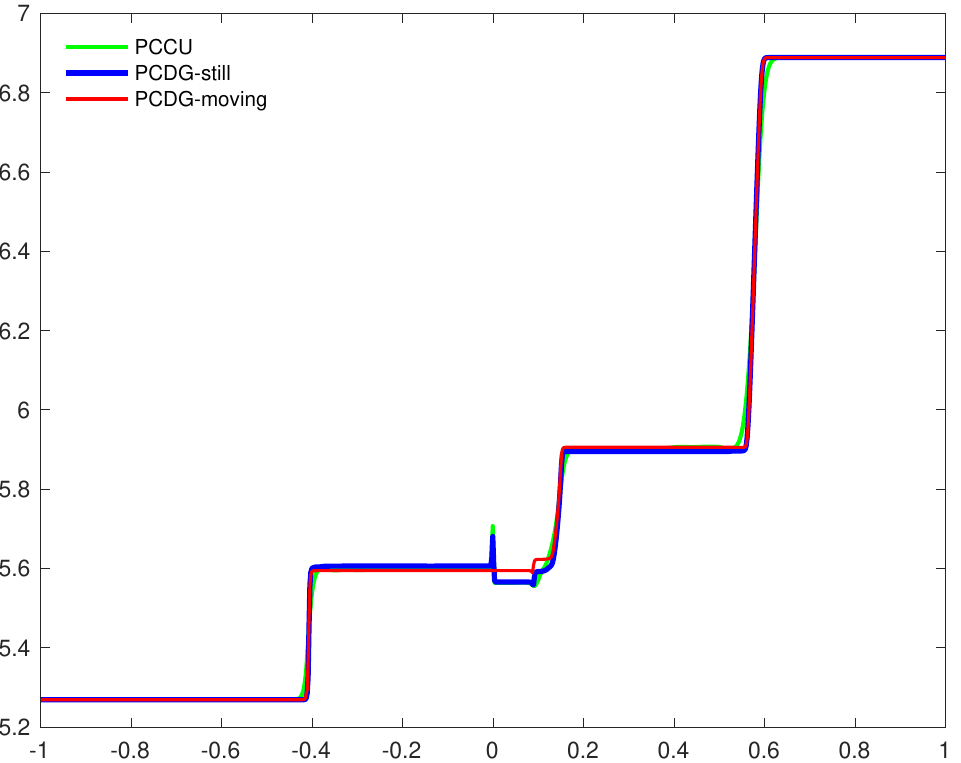}
  }
  \subfigure[lower layer $h_2$]{
  \centering
  \includegraphics[width=6.5cm,scale=1]{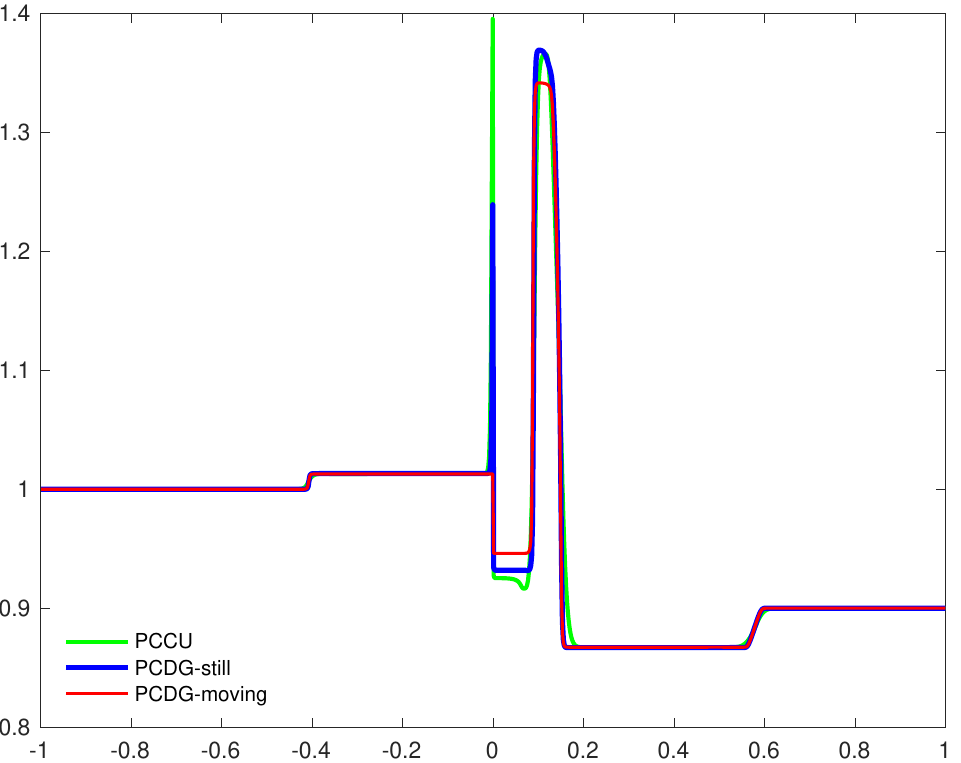}
  }
  \subfigure[discharge $m_2$]{
  \centering
  \includegraphics[width=6.5cm,scale=1]{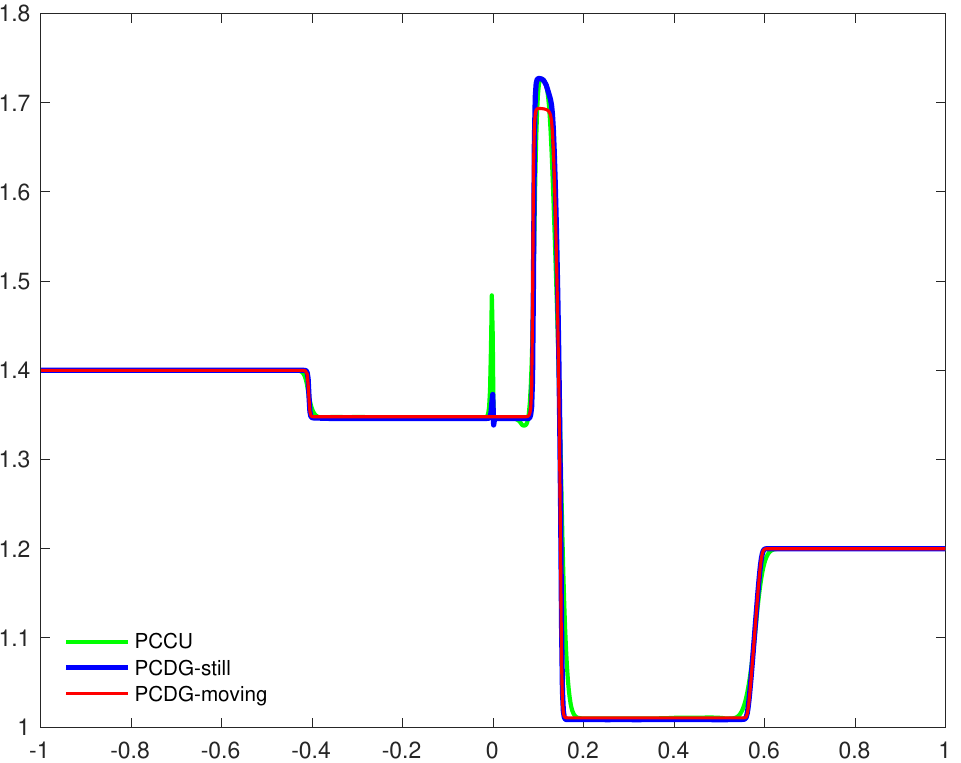}
  }
  \subfigure[energy $E_2$]{
  \centering
  \includegraphics[width=6.5cm,scale=1]{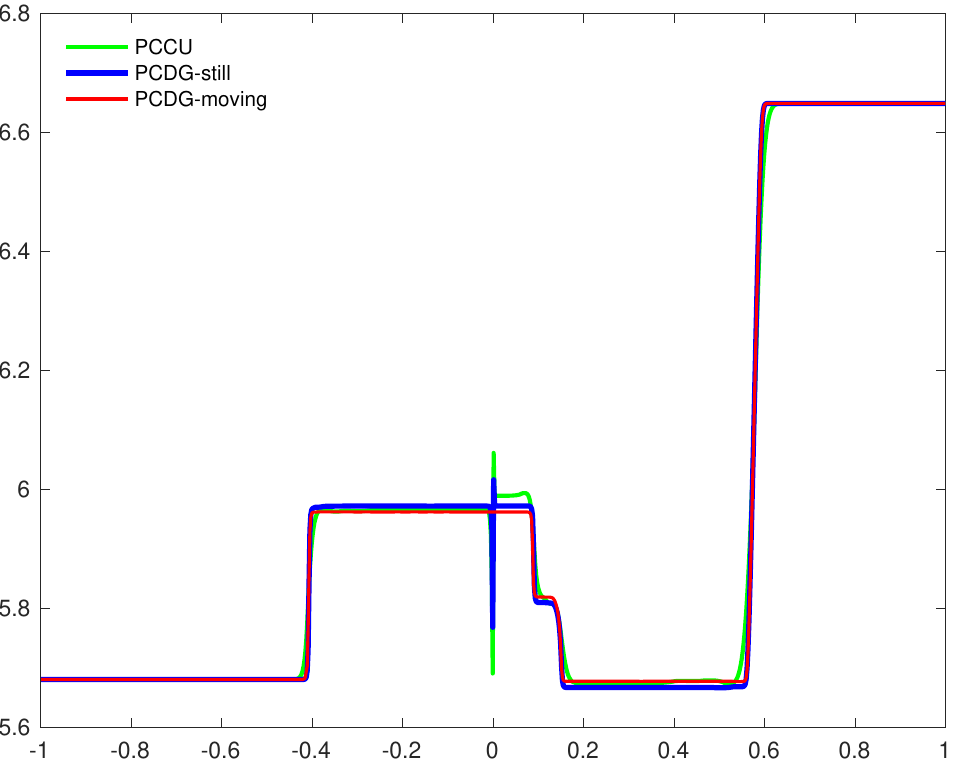}
  }
  \caption{Example \ref{rie1d}: Numerical solutions for Test 2 (\ref{riemov_u2}), computed by PCCU, PCDG-still, PCDG-moving schemes with 1000 grid cells. }\label{riemov2}
\end{figure}

\subsection{Two-dimensional tests}
As there is no general form of two-dimensional generalized moving water equilibrium states, we restrict ourselves to the still water equilibria preserving path-conservative discontinuous Galerkin method for two-dimensional cases.
\begin{example}{\bf Interface propagation over a flat bottom}\label{interpro_flat2d}
\end{example}
This example is a 2-D extension of the 1-D Example \ref{pro1d}, which is taken from \cite{kurganov2009central}, a round-shape interface propagates in the northeast direction. We consider the following initial conditions
\begin{equation}\label{pro_flat_ini}
(h_1,m_1,n_1,w,m_2,n_2)(x,y,0) = \left\{\begin{array}{lll}
    (0.50, 1.250, 1.250, -0.50, 1.250, 1.250), &\text{if} \  (x,y)\in \Omega, \\
    (0.45, 1.125, 1.125, -0.45, 1.375, 1.375), &\text{otherwise},\\
    \end{array}\right.
\end{equation}
over the flat bottom topography $b(x,y)=-1$. Here
\begin{equation}\label{domain}
  \Omega = \{x<-0.5,y<0\}\cup \{(x+0.5)^2 + (y+0.5)^2 <0.25\} \cup \{x<0,y<-0.5\},
\end{equation}
and the computational domain is $[-0.55,0.7]\times [-0.55,0.7]$. We take the gravitation constant $g = 10$ and the density ratio $r=0.98$. We compute the solutions until $t=0.1$ on a sequence of 250$\times$250 and 500$\times$500 uniform cells. Four waves are developed and propagated with four different characteristic speeds along the dominant diagonal $y = x$. Fig. \ref{propa_flat} presents the contours of the water surface $h_1+w$ and the upper layer $h_1$. As in the one-dimensional case, an intermediate state emerges, seeing the corresponding results along the line $y = x$ in Fig. \ref{propa_flat1}. We can observe that the numerical results reach an amicable agreement with the reference solution, and the discontinuity is well captured without oscillation.

\begin{figure}[htb!]
  \centering
  \subfigure[water surface $h_1+w$]{
  \centering
  \includegraphics[width=7cm,scale=1]{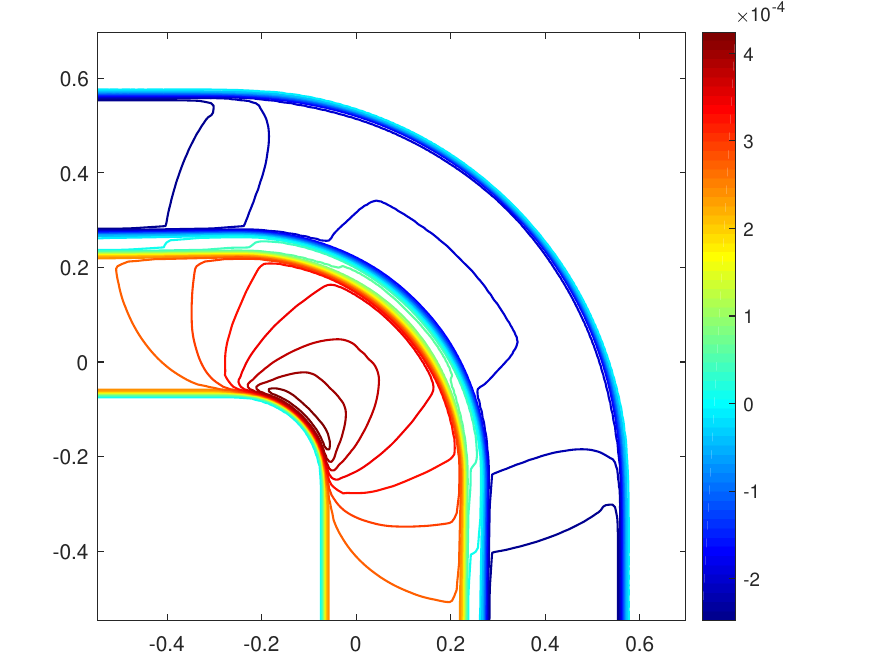}
  }
  \subfigure[upper layer $h_1$]{
  \centering
  \includegraphics[width=6.7cm,scale=1]{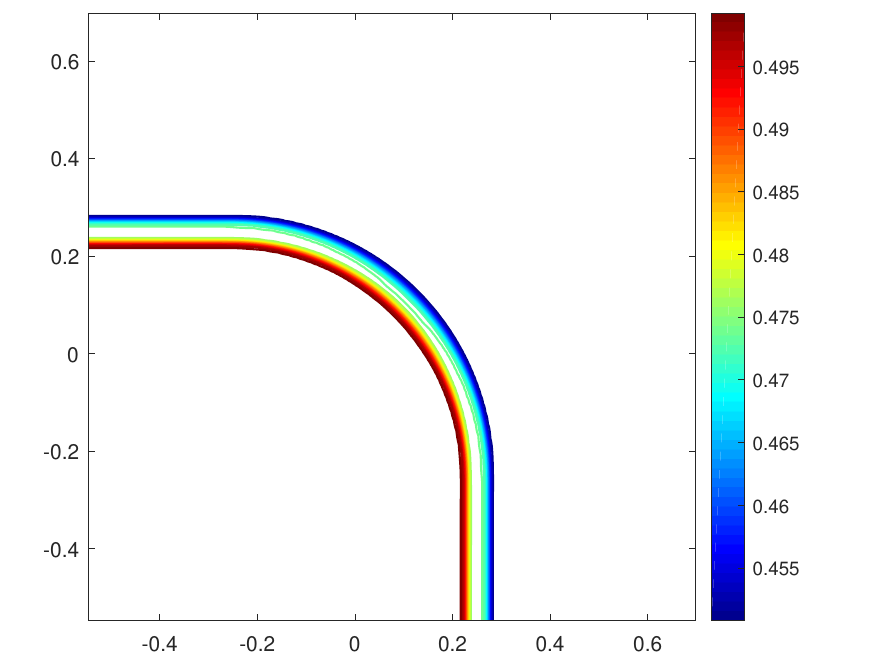}
  }
  \subfigure[water surface $h_1+w$]{
  \centering
  \includegraphics[width=7cm,scale=1]{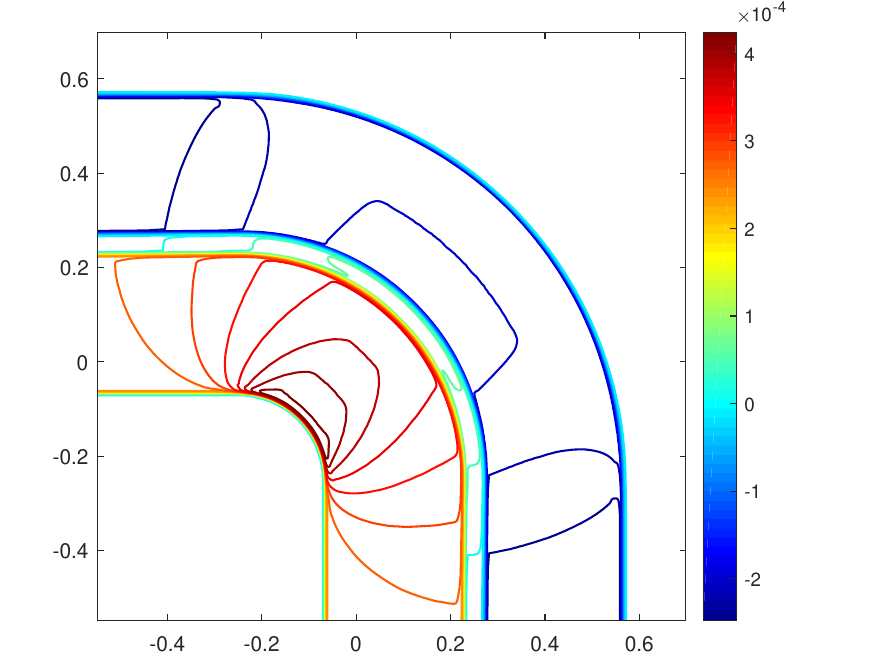}
  }
  \subfigure[upper layer $h_1$]{
  \centering
  \includegraphics[width=6.7cm,scale=1]{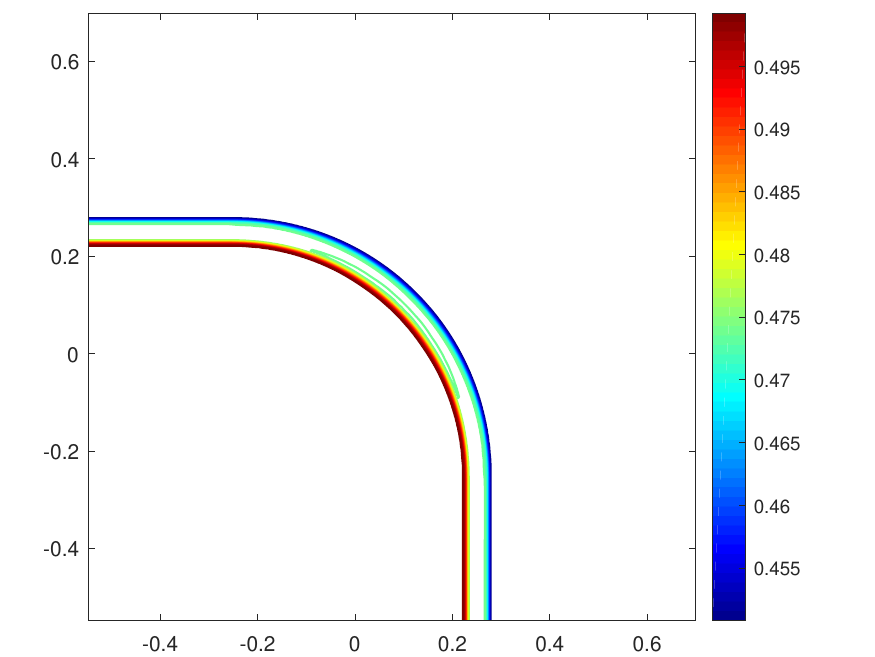}
  }
  \caption{Example \ref{interpro_flat2d}: The contours of the water surface $h_1+w$ (left) from -3E-4 to 4.5E-4; the upper layer $h_1$ (right) from 0.449 to 0.501, with the initial value (\ref{pro_flat_ini}). 30 uniformly spaced contour lines with 250$\times$250 (top) and 500$\times$500 cells (bottom).  }\label{propa_flat}
\end{figure}

\begin{figure}[htb!]
  \centering
  \subfigure[water surface $h_1+w$]{
  \centering
  \includegraphics[width=7cm,scale=1]{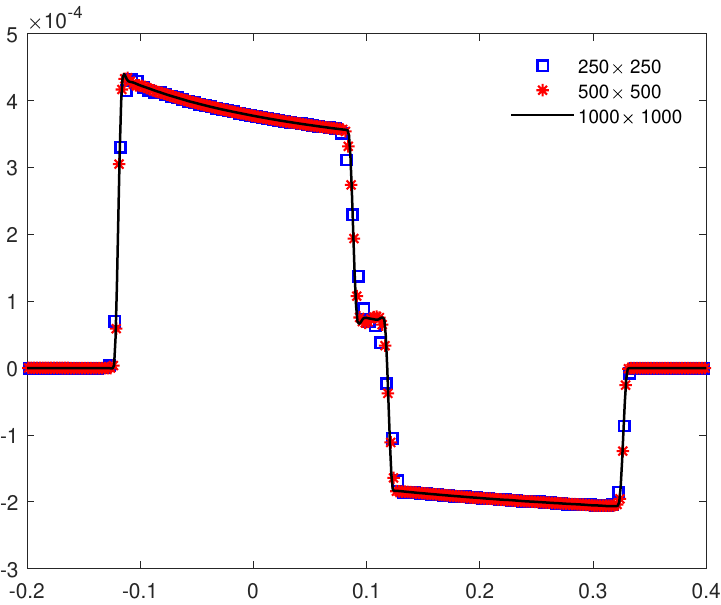}
  }
  \subfigure[upper layer $h_1$]{
  \centering
  \includegraphics[width=7.3cm,scale=1]{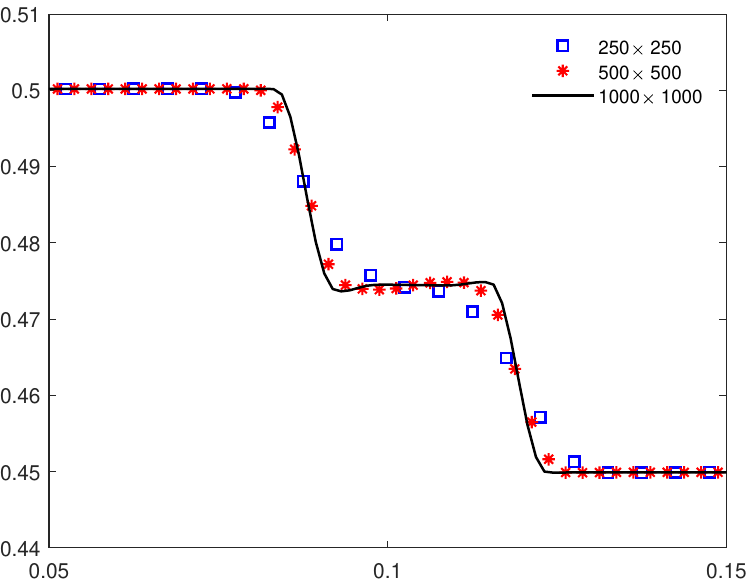}
  }
  \caption{Example \ref{interpro_flat2d}: Scatter plot of the cross sections of the water surface $h_1 + w$ and the upper layer $h_1$ along the line $y=x$, with the initial value (\ref{pro_flat_ini}), using 250$\times$250, 500$\times$500, 1000$\times $1000 cells for comparison. }\label{propa_flat1}
\end{figure}

\begin{example}{\bf Interface propagation over a nonflat bottom}\label{interpro_nonflat2d}
\end{example}
This example we test is the extension of the previous example to the case of a nonflat bottom topography. Gaussian-shaped bottom topography
$$b(x,y)=0.05\exp^{-100(x^2+y^2)}-1,$$
is considered and the initial conditions are given by
\begin{equation}\label{pro_nonflat_ini}
(h_1,u_1,v_1,w,u_2,v_2)(x,y,0) = \left\{\begin{array}{lll}
    (0.50, 2.5, 2.5, -0.50, 2.5, 2.5),&  \text{if} \  (x,y)\in \Omega, \\
    (0.45, 2.5, 2.5, -0.45, 2.5, 2.5), &\text{otherwise},\\
    \end{array}\right.
\end{equation}
where $\Omega$ is the same as in (\ref{domain}). We take the computational domain  $[-0.55,0.7]\times [-0.55,0.7]$, the gravitation constant $g = 10$ and the density ratio $r=0.98$. We apply our resulting scheme to calculate the solutions up to $t = 0.1$ with 250$\times$250 and 500$\times$500 uniform rectangular meshes respectively, and plotted the water surface $h_1+w$ and the upper layer $h_1$ in Fig. \ref{propa_nonflat}. A more complicated wave structure can be observed due to the existence of the nonflat bottom function. It can realize non-oscillatory and match pretty well the results reported in \cite{kurganov2009central}, which indicates the effectiveness of our proposed well-balanced PCDG scheme.

\begin{figure}[htb!]
  \centering
  \subfigure[water surface $h_1+w$]{
  \centering
  \includegraphics[width=7cm,scale=1]{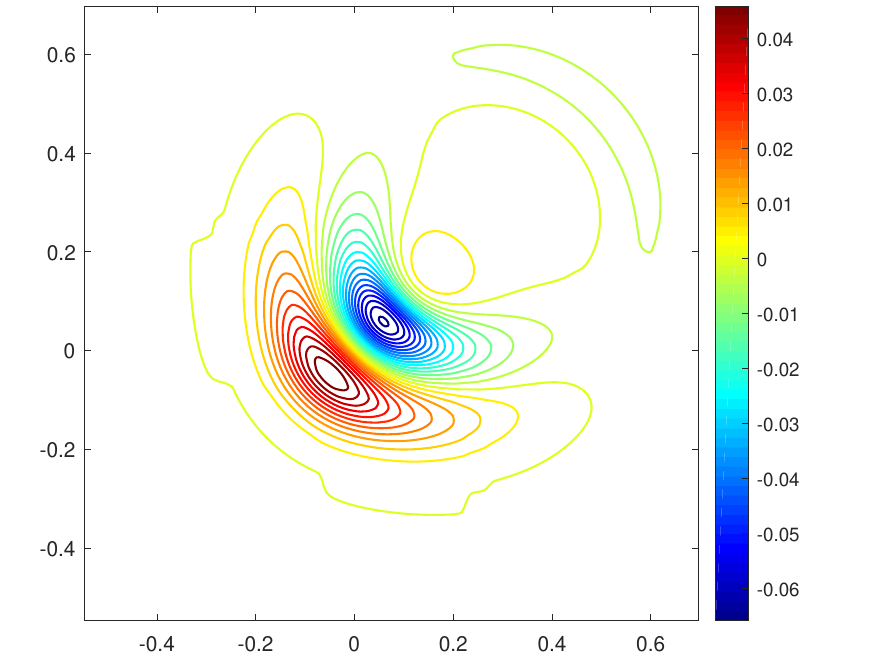}
  }
  \subfigure[upper layer $h_1$]{
  \centering
  \includegraphics[width=7cm,scale=1]{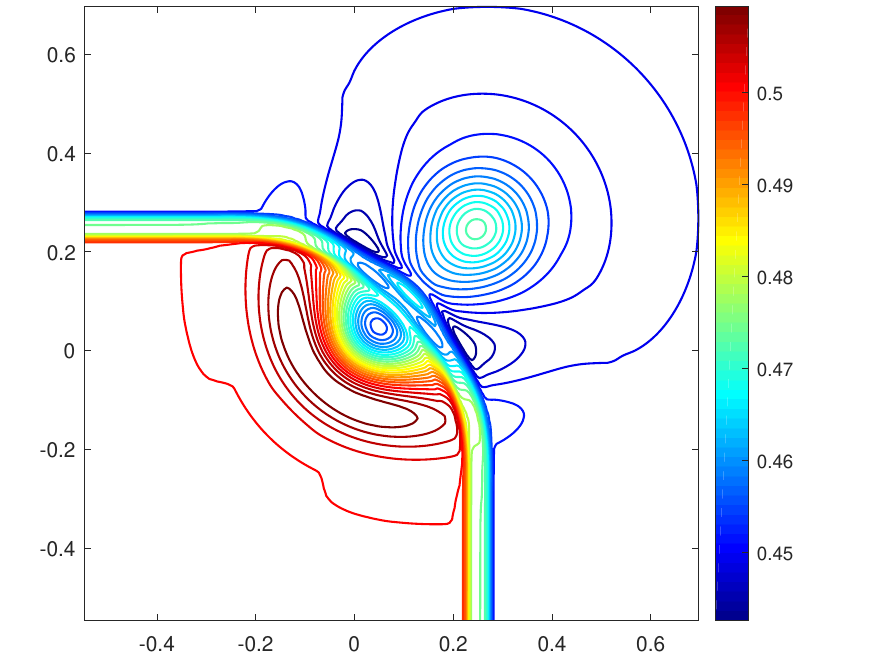}
  }
  \subfigure[water surface $h_1+w$]{
  \centering
  \includegraphics[width=7cm,scale=1]{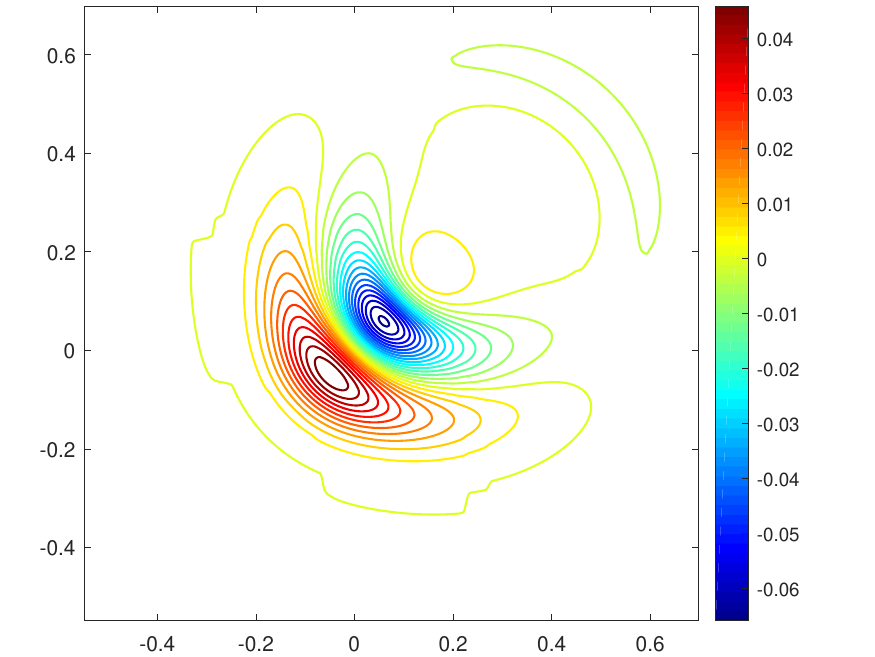}
  }
  \subfigure[upper layer $h_1$]{
  \centering
  \includegraphics[width=7cm,scale=1]{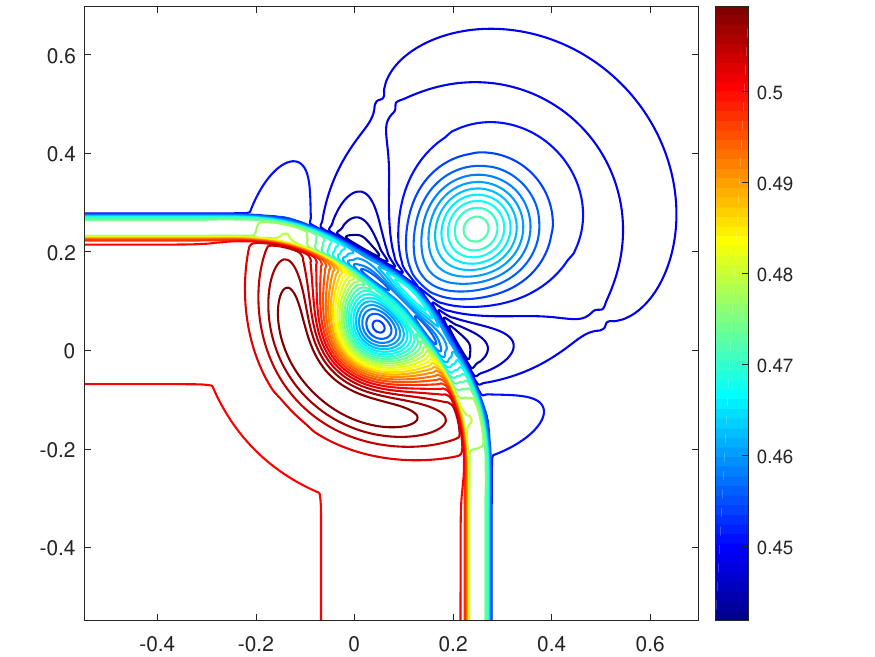}
  }
  \caption{Example \ref{interpro_nonflat2d}: The contours of the water surface $h_1+w$ (left) from -0.07 to 0.05; the upper layer $h_1$ (right) with the initial value (\ref{pro_nonflat_ini}). 30 uniformly spaced contour lines with 250$\times$250 (top) and 500$\times$500 cells (bottom). }\label{propa_nonflat}
\end{figure}

\begin{example}{\bf Internal circular dam break over a flat bottom}\label{dam_flat2d}
\end{example}
We test the same dam breaking problem taken from \cite{chu2022fifth} over a flat bottom topography $b(x,y)=-2$. On the computational domain $[-5,5]\times [-5,5]$, the initial conditions are given by
\begin{equation}\label{dam_flat_ini}
(h_1,m_1,n_1,w,m_2,n_2)(x,y,0) = \left\{\begin{array}{lll}
    (1.8, 0, 0, -1.8, 0, 0),&  \text{if} \  x^2+y^2>4, \\
    (0.2, 0, 0, -0.2, 0, 0), &\text{otherwise}.\\
    \end{array}\right.
\end{equation}
The constant gravitational acceleration is $g = 9.81$ and the density ratio is $r=0.998$. We simulate the results at different times $t=4,6,10,14,16,20$ with 200$\times$200 and 400$\times$400 uniform cells for comparison. A scatter plot of the cross sections of the interface $w$, the water surface $h_1 + w$, and the bottom topography $b$ along the line $y=x$ is provided in Fig. \ref{dam_flat}, which yields well simulation of our studied scheme.

\begin{figure}[htb!]
  \centering
  \subfigure[$t=4$]{
  \centering
  \includegraphics[width=6.5cm,scale=1]{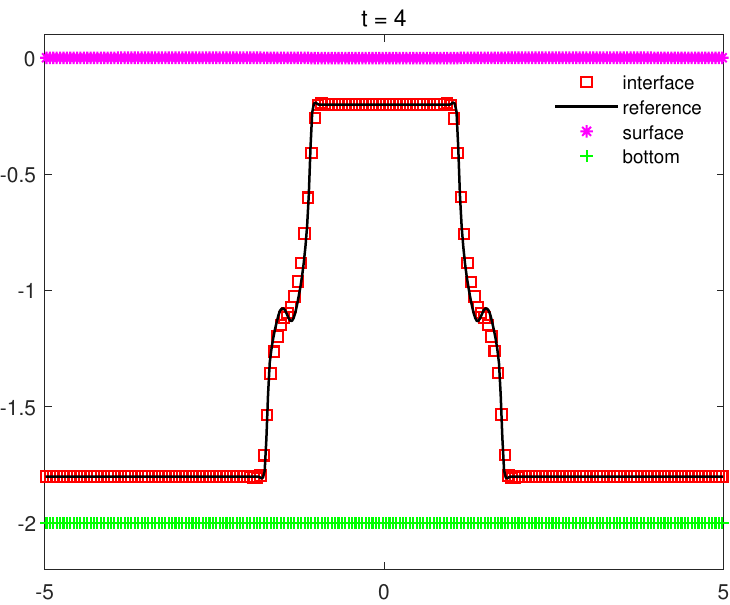}
  }
  \subfigure[$t=6$]{
  \centering
  \includegraphics[width=6.5cm,scale=1]{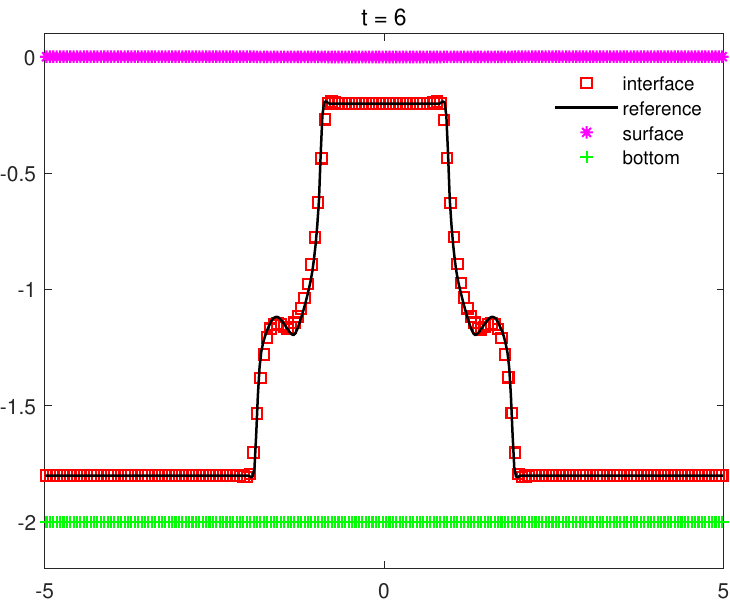}
  }
  \subfigure[$t=10$]{
  \centering
  \includegraphics[width=6.5cm,scale=1]{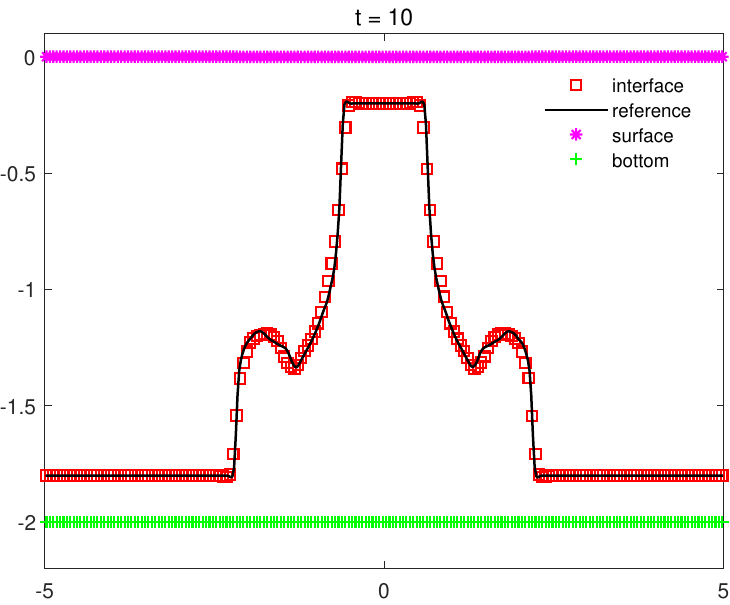}
  }
  \subfigure[$t=14$]{
  \centering
  \includegraphics[width=6.5cm,scale=1]{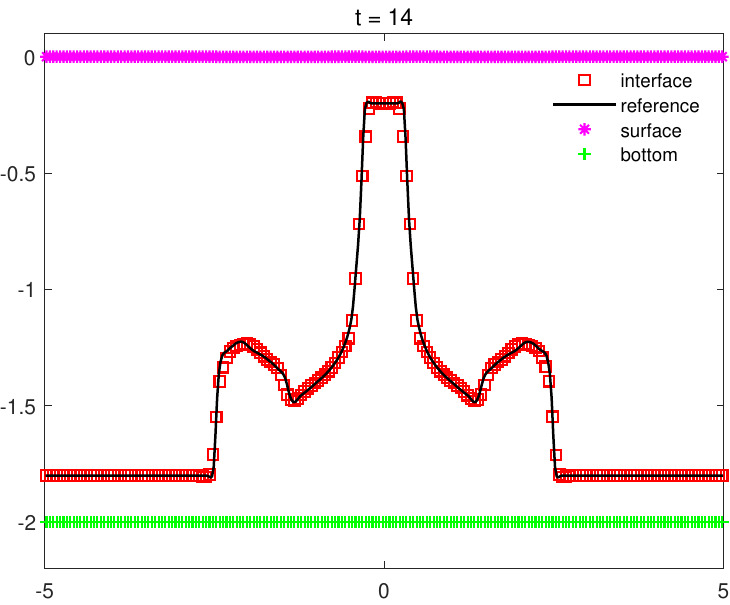}
  }
  \subfigure[$t=16$]{
  \centering
  \includegraphics[width=6.5cm,scale=1]{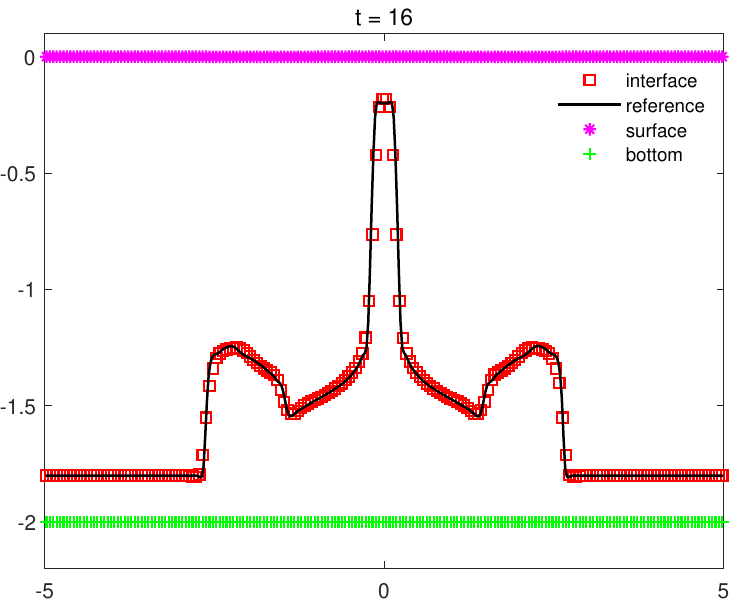}
  }
  \subfigure[$t=20$]{
  \centering
  \includegraphics[width=6.5cm,scale=1]{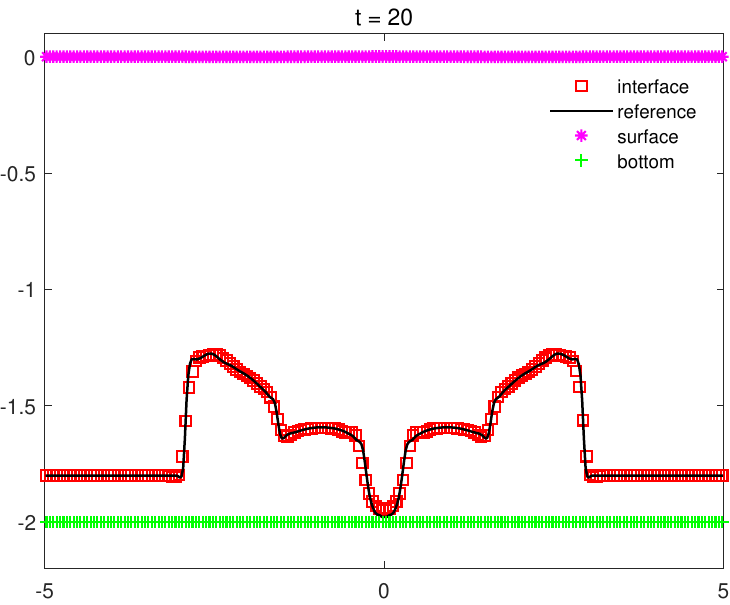}
  }
  \caption{Example \ref{dam_flat2d}: Scatter plot of the cross sections of the interface $w$, the water surface $h_1 + w$ and the bottom topography $b$ along the line $y=x$, with the initial value (\ref{dam_flat_ini}) at different times $t=4,6,10,14,16,20$, using 200$\times $200 and 400$\times$400 cells for comparison. }\label{dam_flat}
\end{figure}

\begin{example}{\bf Internal circular dam break over a nonflat bottom}\label{dam_nonflat2d}
\end{example}
The last example we test is given in \cite{castro2012central}, in which an internal circular dam breaking problem is considered over a nonflat bottom function
$$b(x,y)=0.5\exp^{-5(x^2+y^2)}-2.$$
The initial conditions are given by
\begin{equation}\label{dam_nonflat_ini}
(h_1,m_1,n_1,w,m_2,n_2)(x,y,0) = \left\{\begin{array}{lll}
    (1.8, 0, 0, -1.8, 0, 0),&  \text{if} \  x^2+y^2>1, \\
    (0.2, 0, 0, -0.2, 0, 0), &\text{otherwise},\\
    \end{array}\right.
\end{equation}
on the computational domain $[-2,2]\times [-2,2]$.
The gravitational constant is $g = 9.81$ and the density ratio is $r=0.98$. We run the simulation till times $t=1,2$ with 200$\times$200 uniform cells. The contours of the interface $w$ and the corresponding slices along the line $y=0$ are illustrated in Figs. \ref{dam_nonflat1} and \ref{dam_nonflat2} respectively. The results indicate that our developed well-balanced PCDG scheme can resolve complex features of the flow very well and provide oscillation-free solutions.

\begin{figure}[htb!]
  \centering
  \subfigure[$t=1$]{
  \centering
  \includegraphics[width=7cm,scale=1]{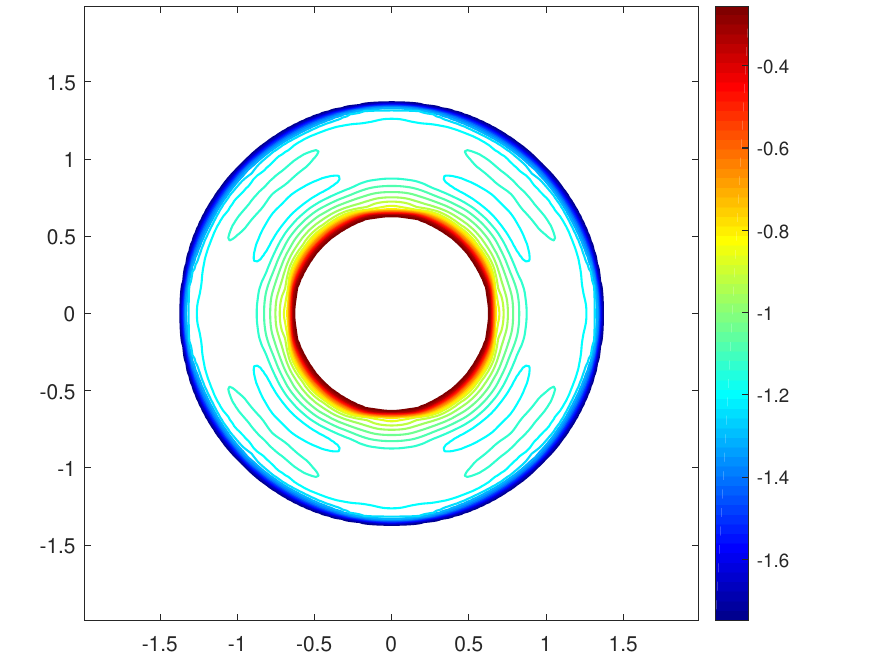}
  }
  \subfigure[$t=2$]{
  \centering
  \includegraphics[width=7cm,scale=1]{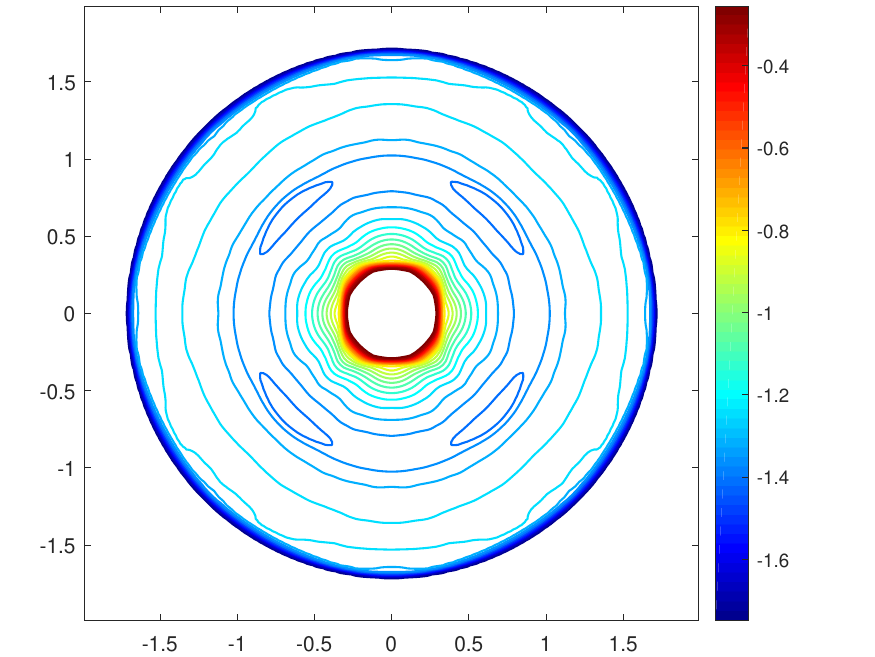}
  }
  \caption{Example \ref{dam_nonflat2d}: The contours of the interface $w$ from -1.805 to -0.2 with the initial value (\ref{dam_nonflat_ini}). 30 uniformly spaced contour lines at different times $t=1$ (left) and $t=2$ (right) with 200$\times$200 cells. }\label{dam_nonflat1}
\end{figure}

\begin{figure}[htb!]
  \centering
  \subfigure[$t=1$]{
  \centering
  \includegraphics[width=7cm,scale=1]{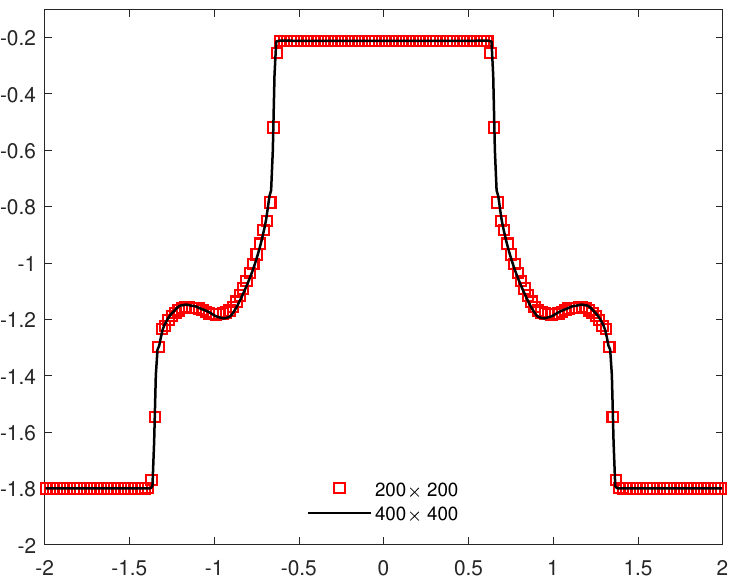}
  }
  \subfigure[$t=2$]{
  \centering
  \includegraphics[width=7cm,scale=1]{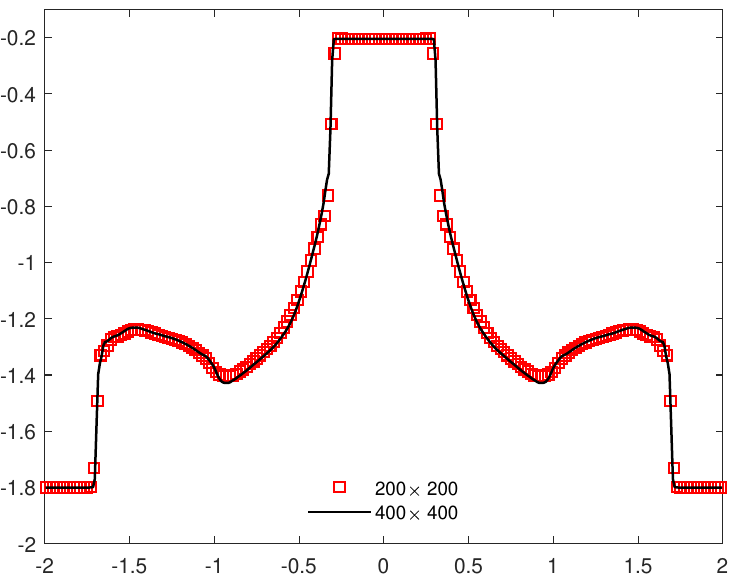}
  }
  \caption{Example \ref{dam_nonflat2d}: Scatter plot of the cross sections of the interface $w$ along the line $y=0$, with the initial value (\ref{dam_nonflat_ini}) at different times $t=1$ (left) and $t=2$ (right), using 200$\times $200 and 400$\times$400 cells for comparison. }\label{dam_nonflat2}
\end{figure}

\section{Conclusion}\label{se:co}
We have developed two new high-order well-balanced discontinuous Galerkin methods for the one- and two-dimensional two-layer shallow water equations within the framework of the path-conservative approach introduced in \cite{dal1995definition}. By employing linear path functions applied to the equilibrium variables, our developed schemes were designed to precisely preserve the still water and moving water equilibrium states within the adopted equilibrium preserving DG space. Theoretical analysis and numerical validation under various conditions have demonstrated that both of our proposed schemes are well-balanced, accurately maintaining the still water and moving water steady states. Furthermore, our PCDG approaches exhibit high-order accuracy, robustness, essentially non-oscillatory properties, and high resolution in capturing complex small features at or near the equilibrium flow. We anticipate that this framework can be extended to various nonconservative models, including those for which the equilibrium variables contain a global (in space) integral quantity. Our future work will consider more complex nonconservative systems, such as the Baer-Nunziato model of compressible two-phase flow.

\setcounter{equation}{0}
\renewcommand\theequation{\Alph{section}.\arabic{equation}}
\renewcommand{\appendixname}{Appendix~\Alph{section}}
\begin{appendix}

\section{Eigenstructures of PCDG scheme for moving water equilibria}\label{a2}
The Jacobi matrix in the TVB limiter is replaced by
\begin{equation}\label{a_m}
  \begin{aligned}
   \mathcal{A}^{*}= \left(\dfrac{\partial {\boldsymbol {v}^e}}{\partial\boldsymbol u} \right)
   \mathcal{A}(\boldsymbol{u})
   \left(\dfrac{\partial {\boldsymbol {v}^e}}{\partial\boldsymbol u}\right)^{-1}  =
  \left[\begin{array}{cccc}
    u_1 & g & 0 & g  \\
    h_1 & u_1 & 0 & 0 \\
    0 & gr & u_2 & g  \\
    0 & 0 & h_2 & u_2  \\
  \end{array}\right].
  \end{aligned}
\end{equation}
Herein the variables $\vec{u}$, $\boldsymbol {v}^e$ are taken in \eqref{2LSWE_co} and \eqref{v_e}, while $\mathcal{A}(\boldsymbol{u})$ in \eqref{Jaco} respectively. 
We can refer to \cite{zhang2023moving} for more details. Let $\lambda_{k}$ be the eigenvalues calculated by the analytical solution to the characteristic quadratic of the matrix (\ref{Jaco}), see more details in Section \ref{se:mo}.
The right eigenvectors $r_k$ corresponding to the matrix (\ref{a_m}) are given by
\begin{equation}\label{r_m}
  r_{k} = \left(u_1-\lambda_k,-\dfrac{c_1^2}{g},\dfrac{(\lambda_k-u_2)(c_1^2 - (u_1-\lambda_k)^2)}{c_2^2},\dfrac{c_1^2 - (u_1-\lambda_k)^2}{g} \right)^T,
\end{equation}
and the matrix $R = [r_1\ r_2\ r_3\ r_4]$, whose columns are right eigenvalues defined in (\ref{r_m}). The left eigenvalues $l_k$ are obtained from the inverse matrix $L = R^{-1} = [l_1\ l_2\ l_3\ l_4]^T$, it has the following expressions
\begin{equation}\label{l_m}
  l_k = \left\{ \dfrac{(c_1^2-u_1^2) - 2u_1\delta_k + \kappa_k}{\zeta_k} \quad -g\dfrac{(c_1^2+u_1^2)\delta_k-u_1(c_1^2-u_1^2+\kappa_k)+\delta_k }{c_1^2\zeta_k} \quad \dfrac{c_2^2}{\zeta_k} \quad g\dfrac{u_2-\delta_k}{\zeta_k}\right\},
\end{equation}
where
\begin{equation}\label{v_s}
  \delta_k = \sum_{j=1,j\neq k}^{4} \lambda_j - 2u_1,\quad \xi_k = \prod_{j=1,j\neq k}^{4}\lambda_j, \quad \kappa_k =  \sum_{j=1,j\neq k}^{4}\prod_{i=1,i\neq j,k}^{4}\lambda_j,\quad \zeta_k = \prod_{j=1,j\neq k}^{4}(\lambda_j - \lambda_k).
\end{equation}.

\end{appendix}

 \section*{Declarations}

 \subsection*{Data Availability}
 The datasets generated during the current study are available from the corresponding author upon reasonable request. They support
 our published claims and comply with field standards.

 \subsection*{Competing of interest}
 The authors declare that they have no known competing financial interests or personal relationships that could have appeared to
 influence the work reported in this paper.

 \subsection*{Funding}
 The research of Yinhua Xia was partially supported by National Key R\&D Program of China No. 2022YFA1005202/2022YFA1005200, and NSFC grant No. 12271498. The research of Yan Xu was partially supported by NSFC grant No. 12071455.


\addcontentsline{toc}{section}{References}
\bibliographystyle{abbrv}
\normalem
\bibliography{shortreference}
\end{document}